\documentclass[11pt,reqno]{amsart}

\usepackage{And-Cha-Gra}

\title[Shock formation in 1D conservation laws: vanishing viscosity]{Shock formation in 1D conservation laws II:\\Vanishing viscosity}

\author{John Anderson}
\address{JA: Department of Mathematics, Stony Brook University, Stony Brook, NY 11794, USA}
\email{\href{mailto:jrlanderson@math.stonybrook.edu}{\tt jrlanderson@math.stonybrook.edu}}

\author{Sanchit Chaturvedi}
\address{SC: Courant Institute of Mathematical Sciences, New York University, New York, NY 10012, USA}
\email{\href{mailto:chaturvedisanchit@nyu.edu}{\tt chaturvedisanchit@nyu.edu}}

\author{Cole Graham}
\address{CG: Department of Mathematics, University of Wisconsin--Madison, Madison, WI 53706, USA}
\email{\href{mailto:graham@math.wisc.edu}{\tt graham@math.wisc.edu}}

\date{\today}

\dedicatory{In memory of Dr. S.~C. Chaturvedi}

\begin{document}

\begin{abstract}
  We study the effects of weak viscosity on shock formation in 1D hyperbolic conservation laws.
  Given an inviscid solution that forms a nondegenerate shock, we add a small viscous regularization and study the limit as the viscosity vanishes.
  Using a matched asymptotic expansion, we determine the sharp rate of convergence in strong norms up to the time of inviscid shock formation, and we identify universal viscous behavior near the first singularity.
  To treat the complex interactions between multiple characteristics and the viscosity, we develop an approximation scheme that exploits a certain decoupling between shocking and nonshocking characteristics.
  Our analysis makes minimal assumptions on the equation, and in particular applies to the compressible Navier--Stokes equations with degenerate physical viscosity.
\end{abstract}

\maketitle

\section{Introduction}

We study systems of viscous conservation laws in one spatial dimension:
\begin{equation}
  \label{eq:main}
  \rd_t\psi^{(\nu)}+A(\psi^{(\nu)})\rd_x \psi^{(\nu)} = \nu\partial_x[B(\psi^{(\nu)})\partial_x\psi^{(\nu)}], \quad t,x \in \R.
\end{equation}
The solution $\psi^{(\nu)}$ is vector-valued with $N$ components, and $A$ and $B$ are $N \times N$ advection and diffusion matrices, respectively.
We index the solution by the viscosity $\nu$, which is small and nonnegative.

Conservation laws arise widely in continuum mechanics, modeling phenomena ranging from the motion of gases to the dynamics of elastic materials~\cite{Dafermos}.
For example, if $(\rho, u, E)$ denotes the mass density, velocity, and energy density of a gas, the compressible Navier--Stokes(--Fourier) equations can be written in the form
\begin{equation}
  \label{eq:NSE}
  \begin{aligned}
    \partial_t \rho + \partial_x (\rho u) &= 0,\\
    \partial_t (\rho u) + \partial_x(\rho u^2 + P) &= \nu \partial_x(\mu \partial_xu),\\
    \partial_t E + \partial_x(Eu + Pu) &= \nu \partial_x(\mu u \partial_xu + \kappa \partial_x \theta).
  \end{aligned}
\end{equation}
The pressure $P$ and temperature $\theta$ are functions of $(\rho, u, E)$, while the dynamic viscosity $\mu$ and thermal conductivity $\kappa$ depend on $\theta$.
These functions are material properties of the gas in question.
Here the small parameter $\nu > 0$ represents the Knudsen number.
The system \eqref{eq:NSE} provides the primary motivation for this work.

In the inviscid setting ($\nu = 0$), initially smooth solutions of \eqref{eq:main} typically develop singularities in finite time.
Fast parcels of fluid can overtake slow, forming a shock singularity in which $\psi$ remains bounded but $\partial_x \psi$ blows up~\cite{Lax64,John74}.
In contrast, for certain viscous equations with $\nu > 0$, the diffusion $B$ can prevent shock formation, producing global-in-time classical solutions~\cite{Kawashima87,MelVas}.
Inviscid and viscous models thus differ qualitatively at the moment of shock formation.
This discrepancy raises important questions regarding the vanishing viscosity limit:
\smallskip
\begin{quote}
  \textit{%
    As a shock forms, does the viscous solution $\psi^{(\nu)}$ converge to the inviscid solution $\psi^{(0)}$ as $\nu \to 0$?
    In which norms, at what rates?
    What is the precise structure of the viscous regularization?
  }  
\end{quote}
\smallskip
In a companion work~\cite{AndChaGra25a}, we examine the detailed structure of inviscid shock formation.
Building on those results, we here assume $\psi^{(0)}$ forms a nondegenerate shock at a time $t_* > 0$, and study $\psi^{(\nu)}$ as $\nu \to 0$ on $[0, t_*] \times \R$.
\begin{theorem}
  \label{thm:approx-informal}
  Under natural conditions on $A$, $B,$ and $\psi^{(0)}$, $\psi^{(\nu)}$ can be approximated to arbitrary precision and regularity on $[0, t_*] \times \R$ using simpler building blocks.
\end{theorem}
\noindent
Using this approximation, we can answer the questions raised above.
For example:
\begin{corollary}
  \label{cor:informal}
  Under the same conditions, $\norms{\psi^{(\nu)} - \psi^{(0)}}_{L^\infty([0, t_*] \times \R)} \leq C \nu^{1/4}$.
  In H\"older spaces, $\psi^{(\nu)} \to \psi^{(0)}$ in $L_t^\infty \m{C}_x^\al$ if and only if $\al < 1/3$.
  Near shock formation, $\psi^{(\nu)}$ resembles a universal profile solving the viscous Burgers equation.
\end{corollary}
\noindent
We make these statements precise in Theorem~\ref{thm:approx} and Corollaries~\ref{cor:rate}--\ref{cor:universal} below.

Our approach flexibly handles a wide range of equations, including the compressible Navier--Stokes equations \eqref{eq:NSE} with degenerate physical viscosity.
To our knowledge, this is the first study in strong norms of the vanishing viscosity limit up to shock formation for systems of conservation laws.

\subsection{Setup}

We now describe our problem in greater detail.
The solution $\psi$ takes values in an $N$-dimensional ``state space'' $\m{V} \cong \R^N$, which represents a collection of physical variables such as mass, momentum, and energy density.
The advection $A \colon \m{V} \to \m{V} \otimes \m{V}^*$ satisfies:
\begin{enumerate}[label={(H\arabic*)},series=hyp]
\item
  \label{hyp:advection}
  $A$ is the derivative $\Der f$ of a smooth ``flux function'' $f \colon \m{V} \to \m{V}$ and has $N$ distinct real eigenvalues.
  Also, there exists a smooth ``entropy-flux pair'' $\eta, q \colon \m{V} \to \R$ such that $\nab \eta = (\nab q)^\top A$ and $\Der^2 \eta > 0$ as a quadratic form.
\end{enumerate}
We use $\Der$ to denote the total derivative, so $\Der^2\eta$ is the Hessian of $\eta$.

Conservation laws arising in physical systems often admit strictly convex entropies, which play an important role in the theory of weak solutions~\cite{Dafermos}.
Here, they prove beneficial for the well-posedness of the viscous problem \eqref{eq:main} with $\nu > 0$.
To this end, we assume that the diffusion $B \colon \m{V} \to \m{V} \otimes \m{V}^*$ satisfies:
\begin{enumerate}[hyp]
\item
  \label{hyp:diffusion}
  $B$ has constant nullity $s \leq N-1$, and there is a fixed basis for $\m{V}$ in which the first $s$ rows of $B$ vanish.
  Moreover, there is a constant $c > 0$ such that
  \begin{equation}
    \label{eq:entropy-dissipation}
    \Der^2 \eta(B \xi, \xi) \geq c \abs{B \xi}^2 \ForAll \xi \in \m{V}.
  \end{equation}
\end{enumerate}
Here we interpret $\Der^2 \eta$ as a bilinear form.

This condition ensures that \eqref{eq:main} is locally well-posed, as shown by Serre~\cite{Serre10b}, building on seminal work of Kawashima and Shizuta~\cite{Kawashima,ShiKaw85}.
Many physically-motivated diffusions naturally satisfy \ref{hyp:diffusion}.
For example, in the Navier--Stokes equations, physical viscosity vanishes in the mass equation and dissipates the entropy density $\eta$ in the sense of \eqref{eq:entropy-dissipation}.

Next, we place conditions on the inviscid solution $\psi^{(0)}$.
For convenience, we use the same initial condition $\psi^{(\nu)}(0, \anon) = \mr{\psi}$ for all $\nu > 0$ and assume:
\begin{enumerate}[hyp]
\item $\mr{\psi} \in \m{C}^\infty(\R; \m{V})$, $\mr{\psi}|_{(-\infty, -R]} \equiv c_-$, and $\mr{\psi}|_{[R,\infty)} \equiv c_+$ for some $R > 0,$ $c_\pm \in \m{V}$.
\end{enumerate}
In fact, Corollary~\ref{cor:informal} only requires a large but finite degree of regularity.
In general, if $\mr{\psi} \in \m{C}^k$ for large $k$, Theorem~\ref{thm:approx-informal} will hold to some large but finite degree of approximation depending on $k$.
Tracking this dependence is cumbersome, so we assume $\mr{\psi} \in \m{C}^\infty$ for simplicity.
Similarly, the assumption that $\mr{\psi}$ is constant outside a compact set is a mere technical convenience.
Our analysis is local, so we could treat much more general behaviors at infinity.

We are interested in the local structure of $\psi^{(\nu)}$ near shock formation, so we assume $\psi^{(0)}$ forms a first isolated shock at $(t_*, x_*) \in \R_+ \times \R$.
Generically, it should form in a single characteristic, meaning $\partial_x \psi$ diverges in one eigenspace of $A$ but remains bounded in the rest.
The corresponding ``shocking'' eigenvalue is the spatial velocity of the incipient shock.
Using gauge freedom, we shift into a co-moving frame so this velocity becomes zero.
More broadly, using a Galilean transformation of spacetime and an affine change in state space, we can arrange:
\begin{enumerate}[hyp]
\item $x_* = 0$, $\psi^{(0)}(t_*, 0) = 0$, and there is a basis $\{e_1,\ldots,e_n\}$ of $\m{V}$ in which $A(0)$ is diagonal, the shock forms in $e_1$, and $A_1^1(0) = 0$.
  \label{hyp:gauge}
\end{enumerate}
Here we use $A_i^j$ to denote $e_j^\top A e_i$.
Because $A(0)$ is diagonal, $A_1^1(0)$ is the instantaneous velocity of shock formation at $(t_*, x_* = 0)$, set to zero.
We eventually shift time so that $t_* = 0$, but for the sake of exposition, we keep $t_* > 0$ in the introduction.

We next require the diffusion to act nontrivially on the shocking component:
\begin{enumerate}[hyp]
\item $B_1^1(0) > 0$.
  \label{hyp:shocking-diffusion}
\end{enumerate}
Thus the diffusion matrix $B$ may be degenerate, but it must directly influence the shocking characteristic at shock formation.
For example, in the Navier--Stokes equations \eqref{eq:NSE}, physical viscosity vanishes in the first equation, but acts nontrivially on the fluid characteristics that form shocks.

We now turn to the detailed structure of generic inviscid shock formation.
In the scalar case ($N = 1$), shocks generally form as in Figure~\ref{fig:shock-formation}.
\begin{figure}[t]
  \centering
  \includegraphics[width = 0.5\linewidth]{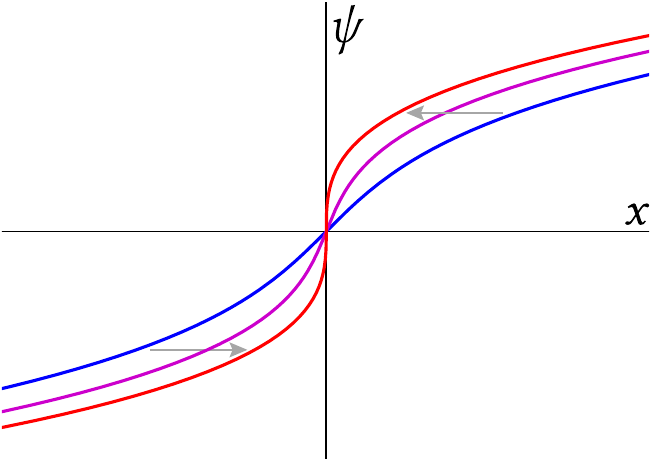}
  \caption[Scalar shock formation.]{%
    Shock formation in a scalar conservation law.
    As time evolves, the profile progressively steepens (blue, then purple) toward a preshock (red) at time $t_*$.
  }
  \label{fig:shock-formation}
\end{figure}
The spatial profile $\psi^{(0)}(t, \anon)$ steepens as $t$ approaches $t_*$ and forms a cusp-like ``preshock'' with infinite slope at $t = t_*$.
A true shock discontinuity only develops after time $t_*$.
Solutions must then be interpreted in a weak sense, and selection principles such as entropy conditions are required for uniqueness~\cite{Dafermos}.
We will not explore this regime here, as we restrict our attention to $t \leq t_*$.

Near the nascent shock, scalar solutions $\psi^{(0)}$ generically resemble a self-similar inverse cubic profile $\cub$ satisfying
\begin{equation*}
  x = a (t_* - t) \cub + b \cub^3
\end{equation*}
for constants $a,b \neq 0$ of the same sign.
(In the case of the Burgers equation, see the work~\cite{CG23} of the second and third authors).
Figure~\ref{fig:shock-formation} depicts the self-similar profile $\cub$.
Its character becomes most evident at the preshock $t = t_*$, when $\cub$ develops a cubic cusp with infinite slope at the origin.
We note that this behavior is generic but not exhaustive.
For certain initial conditions of positive codimension, the first shock can be more degenerate with higher inverse degree~\cite{ColGhoMas_2020}.
Solutions can even degenerate on an entire open interval, producing a discontinuity at $t = t_*$.

Returning to the vector setting with arbitrary $N$, we assume the first component resembles $\cub$ near $(t_*, 0)$, the other components are more regular, and all admit asymptotic expansions in functions that scale like $\cub$:
\begin{enumerate}[label = (H6$'$)]
\item \!(Informal) Near $(t_*, 0),$ $\psi^{(0)}$ admits an expansion in terms that scale like $\cub$, $\cub^2$, $\cub^3$, etc.
  The leading term is $\cub e_1$.
  \label{hyp:expansion-informal}
\end{enumerate}
This nondegenerate formation expansion is our key hypothesis.
For a precise statement, see \ref{hyp:expansion} in Section~\ref{sec:hom} below.
Loosely, \ref{hyp:expansion-informal} provides a fractional Taylor series for $\psi^{(0)}$ with terms obeying the natural inverse-cubic scaling of generic shock formation.
In fact, the cubic nature of $\psi^{(0)}$ in \ref{hyp:expansion-informal} implies that the shocking eigenvalue is genuinely nonlinear in the sense of Lax~\cite{Lax64}.
For this reason, we need not state genuine nonlinearity as a separate hypothesis.
For more details, see Section~\ref{sec:hom} below.

In~\cite{AndChaGra25a}, we show that \ref{hyp:expansion-informal} is satisfied for an open set of initial data.
In particular, it holds for small perturbations of nondegenerate ``simple waves,'' in which $\partial_x \psi$ occupies a single $A$-eigenspace.
We expect this scenario to be generic, in the sense that shocking solutions should typically resemble a nondegenerate simple wave shortly before shock formation.
We leave this question to future investigation.

\subsection{Matched expansions}
\label{sec:expansion}

As indicated in Theorem~\ref{thm:approx-informal}, we analyze $\psi^{(\nu)}$ by assembling an approximate solution $\app{\psi}{}$ of \eqref{eq:main} from solutions of simpler problems.
This amounts to a matched asymptotic expansion linking an outer expansion away from shock formation with an inner expansion near $(t_*, 0)$.
To be more precise, let $\euc(t, x) \coloneqq [(t_* - t)^2 + x^2]^{1/2}$ denote the Euclidean distance to the preshock $(t_*, 0)$ in spacetime.
We choose an exponent $\beta \in (0, 1/2)$ and use a conventional expansion in $\nu$ where $\euc \geq \nu^\beta$:
\begin{equation}
  \label{eq:outer-intro}
  \psi^{(\nu)} \approx \out{\psi}{0} + \nu \out{\psi}{1} + \nu^2 \out{\psi}{2} + \ldots
\end{equation}
The leading term $\out{\psi}{0}$ is simply the inviscid solution $\psi^{(0)}$, and thus solves a hyperbolic conservation law.
We study the precise structure of $\psi^{(0)}$ in our companion paper~\cite{AndChaGra25a}.
The higher-order correctors $\out{\psi}{k}$ solve linearizations of the conservation law about $\psi^{(0)}$, so their behavior is closely tied to that of $\psi^{(0)}$.

The outer expansion \eqref{eq:outer-intro} is valid where $\euc \geq \nu^\beta$ because the viscosity $\nu \partial_x(B\partial_x \psi)$ remains perturbative in this region: it is much smaller than the other terms in \eqref{eq:main}.
This breaks down near $(t_*, 0)$, so we employ a different inner expansion where $\euc \leq \nu^\beta$ based on proximity to the preshock.
Following~\cite{CG23}, we introduce blow-up coordinates
\begin{equation}
  \label{eq:blow-up}
  T \coloneqq \nu^{-1/2} (t - t_*), \quad X \coloneqq \nu^{-3/4} x, \quad \Psi = \nu^{-1/4} \psi
\end{equation}
and propose an inner expansion of the form
\begin{equation}
  \label{eq:inner-intro}
  \Psi^{(\nu)} = \inn{\Psi}{0} + \nu^{1/4} \inn{\Psi}{1} + \nu^{1/2} \inn{\Psi}{2} + \ldots
\end{equation}
The cubic shock $\cub$ is preserved under the blow-up \eqref{eq:blow-up}, while higher powers of $\cub$ are suppressed.
In light of \ref{hyp:expansion-informal}, the shocking component $e_1$ dominates the inner expansion: $\inn{\Psi}{0}$ lies solely in the shocking component and is thus effectively scalar.
Moreover, \eqref{eq:blow-up} preserves the leading behavior of each term in \eqref{eq:main}, so $\inn{\Psi}{0}$ solves the scalar viscous Burgers equation.
The higher correctors $\inn{\Psi}{\ell}$ solve linear equations that conveniently decouple.
Nonshocking components solve constant-coefficient transport equations, and shocking components solve linearizations of viscous Burgers.

To link the two expansions, we choose data for $\inn{\Psi}{\ell}$ in the distant $T$-past so that partial sums of the outer and inner expansions coincide.
At leading order, $\inn{\Psi}{0}$ resembles the cubic profile $\cub$ at infinity in $(T, X)$.
The second and third authors constructed and extensively studied this solution of viscous Burgers in~\cite{CG23}.
The higher correctors are simple to analyze once the leading part $\inn{\Psi}{0}$ is well understood.

\subsection{Main results}

Once we have constructed the outer \eqref{eq:outer-intro} and inner \eqref{eq:inner-intro} expansions in their respective domains, we truncate them to some order $K$ and use a partition of unity to glue them into an approximate solution $\app{\psi}{K}$.
Using local well-posedness estimates, we show that $\app{\psi}{K}$ approximates $\psi^{(\nu)}$.
\begin{theorem}
  \label{thm:approx}
  Assume \ref{hyp:advection}--\ref{hyp:expansion}.
  Then there exists $\bar{\nu} > 0$ such that for all $\nu \in (0, \bar{\nu}]$, \eqref{eq:main} admits a solution $\psi^{(\nu)}$ on $[0, t_*] \times \R$.
  Moreover, for any $s,p \geq 0$, there exist $K \in \N$ and $C > 0$ such that for all $\nu \in (0, \bar{\nu}]$,
  \begin{equation*}
    \norms{\psi^{(\nu)} - \app{\psi}{K}}_{H^s([0, t_*] \times \R)} \leq C \nu^p.
  \end{equation*}
\end{theorem}
\noindent
We note that large-data global well-posedness is presently open for the full Navier--Stokes--Fourier system \eqref{eq:NSE}, so even the existence of a solution to time $t_*$ cannot be taken for granted.

The approximate solution $\app{\psi}{K}$ incorporates solutions to a hyperbolic conservation law, the viscous Burgers equation, linear advection equations, and linear advection-diffusion equations.
Each is considerably simpler to analyze than the original nonlinear parabolic-hyperbolic system \eqref{eq:main}.
This justifies Theorem~\ref{thm:approx-informal}.

To illustrate the utility of this approximation, we present a few applications based on the leading behavior of the outer and inner expansions.
First, we determine the sharp $L^\infty$-rate of the vanishing viscosity limit.
For systems, this is the first result of its kind in norms stronger than $L^1$.
\begin{corollary}
  \label{cor:rate}
  Assume \ref{hyp:advection}--\ref{hyp:expansion}.
  Then there exists $C \geq 1$ such that for all $\nu \in (0, \bar{\nu}]$,
  \begin{equation*}
    C^{-1} \nu^{1/4} \leq \norms{\psi^{(\nu)} - \psi^{(0)}}_{L^\infty([0, t_*] \times \R)} \leq C \nu^{1/4}.
  \end{equation*}
\end{corollary}
We can also track convergence in higher regularity, which reveals a difference between the shocking and nonshocking components.
Writing $\psi = (\psi^1, \ldots, \psi^N)$ in the basis from \ref{hyp:gauge}, let $\sigma \coloneqq \psi^1$ and $\omega \coloneqq (\psi^2, \ldots, \psi^N)$ denote the shocking and nonshocking components, respectively.
\begin{corollary}
  \label{cor:sharp-conv}
  Assume \ref{hyp:advection}--\ref{hyp:expansion}.
  Then $\sigma^{(\nu)} \to \sigma^{(0)}$ in $L_t^\infty\mathcal{C}_x^{\al}([0,t_*] \times \R)$ as $\nu \to 0$ if and only if $\al < 1/3$, while $\omega^{(\nu)} \to \omega^{(0)}$ in $L_t^\infty\mathcal{C}_x^{\beta}$ if $\beta < 2/3$.
\end{corollary}
\noindent
The regularity threshold $1/3$ is naturally tied to the leading inverse-cubic behavior of the shocking component $\sigma^{(0)}$ from \ref{hyp:expansion-informal}.
However, the nonshocking threshold $2/3$ is somewhat surprising, as there are natural hyperbolic models in which $\omega^{(0)}$ is $\m{C}_x^{1,1/3}$.
Nonetheless, we find that the viscous perturbation $\omega^{(\nu)} - \omega^{(0)}$ is generally only bounded in $\m{C}_x^{2/3}$, due to the off-diagonal influence of $\nu \partial_x^2 \sigma^{(\nu)}$ on $\omega^{(\nu)}$.
We detail this effect in Remark~\ref{rem:unfrozen} below.
So while the convergence of $\omega^{(\nu)}$ to $\omega^{(0)}$ may be improved for certain special systems, Corollary~\ref{cor:sharp-conv} is sharp in general.

Finally, we show that $\psi^{(\nu)}$ has a universal viscous character near shock formation.
Recall the blown-up solution $\Psi^{(\nu)}(T, X)$ defined through \eqref{eq:blow-up}.
\begin{corollary}
  \label{cor:universal}
  Assume \ref{hyp:advection}--\ref{hyp:expansion}.
  Then there exists a unique solution $U$ of the viscous Burgers equation such that, modulo scaling symmetries, $\Psi^{(\nu)} \to U e_1$ locally uniformly in $(T, X)$ as $\nu \to 0$.
\end{corollary}
\noindent
That is, sufficiently near shock formation, $\psi^{(\nu)}$ resembles a universal profile solving viscous Burgers.

Collectively, Corollaries~\ref{cor:rate}--\ref{cor:universal} provide a precise description of the nature of viscous regularization near nondegenerate shock formation.
Of the hypotheses \ref{hyp:advection}--\ref{hyp:expansion}, the first five hold for many natural dissipative conservation laws, including the Navier--Stokes--Fourier equations \eqref{eq:NSE} (in the absence of vacuum) and certain regimes in elasticity and magnetohydrodynamics.
On its face, the final assumption \ref{hyp:expansion} of a nondegenerate formation expansion may seem rather strong.
However, in~\cite{AndChaGra25a} we show that it holds on an open set of initial data, and as described above, we expect it to hold generically.

We also expect our results to extend to 1D reductions of higher-dimensional systems.
For example, we anticipate that Corollaries~\ref{cor:rate}--\ref{cor:universal} hold for the Navier--Stokes regularization of the Euler equations under azimuthal symmetry considered in~\cite{NSV23}.
This is not immediate, as radial coordinates introduce variable coefficients in~\eqref{eq:main}, which we do not explicitly consider.
Nonetheless, our approach is essentially local, so (smooth) variable coefficients should pose no difficulty.

In a similar vein, while we restrict our attention to strictly hyperbolic models, this is mostly a matter of convenience.
In fact, our analysis immediately applies to systems in which the shocking eigenvalue is simple, but others are repeated.
Moreover, we believe the overall approach should extend to shocking eigenvalues with multiplicity, which arise naturally in settings such as magnetohydrodynamics.
In this situation, the leading inner expansion would involve a vector-valued form of viscous Burgers.
We believe energy estimates could be used to treat this inner equation and recover our main results.
We leave this fascinating question to future inquiry.

\subsection{Related work}

Our examination of vanishing viscosity and shock formation fits within a very broad investigation of singularities and regularization in conservation laws.
Here, we discuss a modest selection of works that strike us as particularly relevant to the present effort.
Our companion paper~\cite{AndChaGra25a} includes a similar survey of hyperbolic results, so we here place greater emphasis on the viscous setting.

\subsubsection{Well-posedness}

As shock formation indicates, the well-posedness of conservation laws is itself a subtle matter.
For hyperbolic conservation laws in one dimension, Glimm famously constructed global weak solutions for small-BV data~\cite{Glimm65}.
This inspired a number of other compactness-based constructions~\cite{DiPerna76, Bressan92, BaiJen98, Risebro93} and uniqueness results~\cite{Bressan95,BreLeF97,BreLew_2000,CheKruVas_2022,BreGue_2024}.
For more details, see \cite{Bressan00}.
In contrast, there is no general global theory for inviscid conservation laws in higher dimensions, and in fact such equations need not be well-posed in BV~\cite{rauch1986bv}.

In the viscous setting, degenerate diffusion greatly complicates well-posedness.
The seminal work of Kawashima and Shizuta~\cite{Kawashima, ShiKaw85, KawShi88} identified conditions on the diffusion matrix $B$ that ensure local well-posedness.
Serre has since refined these conditions~\cite{Serre10a, Serre10b}, and our hypothesis \ref{hyp:diffusion} makes use of this framework.
For long times, small $H^2$ data yield global solutions in quite general systems~\cite{Kawashima,Kawashima87}, while rougher data can be propagated weakly in the isentropic Navier--Stokes equations~\cite{LiuYu22}.
Moreover, Mellet and Vasseur have shown that isentropic Navier--Stokes admits global-in-time strong solutions started from \emph{large} $H^1$ data~\cite{MelVas}.
As mentioned earlier, there is no large-data global theory for viscous conservation laws in the generality considered here, so the existence until time $t_*$ in Theorem~\ref{thm:approx} is new for many systems, including Navier--Stokes--Fourier \eqref{eq:NSE}.

\subsubsection{Shock formation}

The study of hyperbolic shock formation has a long history dating back to work of Riemann~\cite{Riemann} on 1D scalar conservation laws.
This was extended by Lax~\cite{Lax64} to $2 \times 2$ systems and John~\cite{John74} to more general systems with small data.
In higher dimensions, John~\cite{John79} and Sideris~\cite{Sideris85} showed that small smooth data break down, and Alinhac provided a more detailed account of nondegenerate shock formation~\cite{Alinhac01, Alinhac02, Alinhac03}.
Christodoulou pioneered a powerful geometric approach to shock formation~\cite{Christodoulou01}, which has since been extended to resolve a number of challenging problems in compressible fluids~\cite{LukSpe01,LukSpe02,AbbSpe}.

A parallel program of Buckmaster, Shkoller, and Vicol is particularly relevant to the present work.
In \cite{BucShkVic_2019, BucShkVic_2020, BucShkVic_2020.2}, the authors use modulation theory to precisely describe the nature of nondegenerate shock formation in compressible Euler.
Their analysis identifies Burgers-like behavior at leading order, and makes use of a multidimensional variant of the inverse cubic profile $\cub$ studied here.
Buckmaster and Iyer have also constructed unstable modes of shock formation in 2D compressible Euler~\cite{BuckIye_2020}.

Beyond shock formation, there has been much recent interest in the construction of maximal developments for hyperbolic equations.
These capture the largest region in which classical solutions make sense, and their study seems to require a detailed understanding of singularity (e.g., shock) formation.
We refer the reader to \cite{Christodoulou01,AbbSpe,ShkVic_2024,AndChaGra25a} for more information in this direction.

\subsubsection{Vanishing viscosity for general data}

The convergence of viscous conservation laws to their inviscid counterparts as $\nu \to 0$ has received a great deal of attention, beginning with pivotal work of Hopf~\cite{Hopf_1950}, Oleinik~\cite{Oleinik_1957}, and Kru\v{z}kov~\cite{Kruzkov_1965} on 1D scalar equations.
Vanishing viscosity for systems is considerably more delicate.
In the case of isentropic Euler, equations with nondegenerate ``artificial'' viscosity can be used to construct $L^\infty$ solutions as subsequential limits~\cite{DiPerna_1983,GQChen_1997,LioPerSou_1996}.
In general systems, the seminal work of Bianchini and Bressan~\cite{BiaBre05} established the vanishing viscosity limit for small-BV data, again with artificial viscosity.

For more delicate \emph{physical} viscosity (which does not act on the density), G.-Q.~Chen and Perepelitsa~\cite{CP10} used $L^2$-methods to prove the convergence of Navier--Stokes to isentropic Euler for finite-energy data.
To treat small-BV solutions with physical viscosity, Kang, Leger, and Vasseur pioneered the method of ``$a$-contraction with shifts''~\cite{LegVas_2011,KanVas_2016}.
In recent breakthrough work with G.~Chen, Kang and Vasseur used this approach to establish vanishing viscosity convergence and stability for barotropic Euler~\cite{ChKaVa24}.

\subsubsection{Vanishing viscosity for developed shocks}

The above works treat vanishing viscosity limits for general data in various classes.
In parallel, there is an extensive literature on the viscous regularization of hyperbolic solutions with fully developed shocks.
This begins with Hopf~\cite{Hopf_1950} and Gilbarg~\cite{Gilbarg_1951}, who constructed viscous shock waves for the Burgers and Navier--Stokes--Fourier equations \eqref{eq:NSE}, respectively.
In the context of vanishing viscosity, such waves should appear in an inner layer near the shock, where viscosity is non-perturbative.
Goodman and Xin developed this intuition into a complete matched asymptotic expansion for a sequence of (artificially) viscous solutions converging to a solution of a general hyperbolic system with a small, isolated shock~\cite{GoodXin_1992}.

This groundbreaking effort inspired much subsequent work, including our own use of a matched expansion.
Yu~\cite{Yu_1999} explored the early time layer in which viscosity smooths an initial discontinuity, and Rousset~\cite{Rousset_2003} extended \cite{GoodXin_1992} to large shocks satisfying certain stability criteria.
In remarkable work in multiple dimensions, Gu\`{e}s, M\'{e}tivier, Williams, and Zumbrun~\cite{GuMeWiZu_2005,GuMeWiZu_2006} constructed viscous expansions about shocks satisfying the Majda stability conditions~\cite{Majda_1983}.
In a different direction, Tang and Teng used a matched approximation to identify the vanishing viscosity rate in $L_x^1$ about piecewise-smooth solutions of scalar conservation laws~\cite{TanTen_1997}.
Wang subsequently showed that their bound is achieved in Burgers at shock formation and the corners of rarefaction waves~\cite{Wang_1999}.

The complete expansions developed above all concern fully developed shocks.
In the context of the present paper, they would apply after time $t_* + \tau$ for some fixed $\tau > 0$, which allows a full discontinuous shock to develop out of the preshock at time $t_*$.
In contrast, as in~\cite{CG23}, we focus on the times $[0, t_*]$ leading up to the preshock.
This leaves the period $[t_*, t_* + \tau]$ of \emph{shock development} open.
It remains a compelling subject for future work.

\subsubsection{Singularities despite regularization}

In equations like ours, the addition of any amount of positive viscosity suffices to prevent the onset of singularities (here, shocks).
However, weaker forms of regularization may fail to prevent singularities.
Such blow-up \emph{despite} regularization has attracted a great deal of recent attention.

In one dimension, the fractional Burgers and KdV equations are notable examples~\cite{Whitham}.
These involve fractional degrees of dissipation or dispersion that can be tuned to exhibit a number of behaviors, including singularity formation under weak regularization.
In fractional (or ``fractal'') Burgers, the critical degree of regularization separating blow-up and well-posedness was independently established in~\mbox{\cite{KiNaSh_2008, DongDuLi_2009, AliDroVov_2007}}.
The precise nature of generic singularities is now understood for equations spanning fractal Burgers~\cite{ChiMorPan_2021}, Burgers--Hilbert~\cite{Yang_2020, yang22}, and many other dissipative and dispersive models \cite{OhPas_2021}.
If diffusion acts only in an additional transverse direction, Burgers can likewise form a shock admitting a precise description~\cite{ColGhoMas_2020}.

In higher dimensions, even classical diffusion may be insufficiently strong to prevent singularities.
In a remarkable effort,  Merle, Rapha\"{e}l, Rodnianski, and Szeftel have constructed solutions of the 3D barotropic Euler and Navier--Stokes equations exhibiting finite-time implosion~\cite{MeRaRoSz01,MeRaRoSz02}.
These results have spectacular applications~\cite{MeRaRoSz03} and have inspired a number of other impressive works~\cite{BCLCS_2025,cao2023non,CCSV,Chen}.

\subsection{Organization}

In Section~\ref{sec:proof-over}, we present an overview of our proof strategy.
We describe the conventional outer expansion in Section~\ref{sec:outer}.
In Section~\ref{sec:inner}, we develop the more delicate inner expansion and an auxiliary ``grid'' expansion that aids the gluing of the outer to the inner (see Section~\ref{sec:proof-over} for more details).
The core of our analysis is Section~\ref{sec:grid-est}, where we control the terms of the grid expansion.
Using these bounds, we demonstrate close matching between the grid and both the outer (Section~\ref{sec:outer-matching}) and inner expansions (Section~\ref{sec:inner-matching}).
Finally, we construct an approximate solution and prove our main results in Section~\ref{sec:closure}.

\section*{Acknowledgments}

JA was partially supported by the National Science Foundation under grant DMS-2103266.
SC was supported by the Simons Foundation Award 1141490.
CG was partially supported by the National Science Foundation through the grants DMS-2103383 and DMS-2516786.

We are grateful to Jonathan Luk for many helpful discussions.

\section{Proof overview}
\label{sec:proof-over}

In this section, we outline our construction of an approximate solution to \eqref{eq:main}.
This yields Theorem~\ref{thm:approx}, and Corollaries~\ref{cor:rate}--\ref{cor:universal} follow from the nature of the approximate solution.

\subsection{Notation}

Throughout, we use $f \lesssim_\al g$ to indicate that $\abs{f} \leq C(\al) g$ for some constant $C(\al) \in \R_+$ that may depend on a parameter $\al$ but is independent of the variables $t,x,$ and $\nu.$
We emphasize the absolute value $\abs{f}$ in our convention, which allows us to treat sign-indefinite functions $f$.
We write $\m{O}_\al(g)$ to denote an implicit function $f$ satisfying $f \lesssim_\al g$.
We write $f \asymp_\al g$ if $f \lesssim_\al g$ and $g \lesssim_\al f$.
Finally, we occasionally write $f \ll g$ if $f/g \to 0$ in a given limit.

\subsection{Homogeneous expansion}
\label{sec:hom}

To simplify scaling relative to the location of shock formation, we shift time so that our Cauchy problem begins at a negative time $t_0 < 0$ and the first singularity forms at the origin in spacetime.
That is,
\begin{equation}
  \label{eq:time-shift}
  \psi^{(\nu)}(t_0, \anon) = \mr{\psi} \And t_* = 0.
\end{equation}
We study our solutions $\psi^{(\nu)}$ on the spacetime strip $[t_0, 0] \times \R$.

Let $\euc$ denote the spacetime Euclidean distance to the origin:
\begin{equation}
  \label{eq:Euclidean}
  \euc(t, x) \coloneqq (\abs{t}^2 + \abs{x}^2)^{1/2}.
\end{equation}
As described in Section~\ref{sec:expansion}, the viscosity is perturbative in the ``outer region'' where $\euc \geq \nu^\beta$ for a certain exponent $\beta \in (0, 1/2)$ to be chosen in the course of the proof.
In this region, we approximate $\psi^{(\nu)}$ with the conventional outer expansion \eqref{eq:outer-intro} in integer powers of $\nu$.
The leading term $\out{\psi}{0}$ coincides with the inviscid solution $\psi^{(0)}$, and the higher correctors $\out{\psi}{k}$ solve linearizations of the the inviscid equation about $\out{\psi}{0}$ (Proposition~\ref{prop:outer-eq}).
Thus, the nature of the expansion is determined by the behavior of its leading part.

In \ref{hyp:expansion-informal}, we assume $\out{\psi}{0}$ is dominated near the origin by the inverse cubic
\begin{equation}
  \label{eq:fu}
  x = a \abs{t} \fu + b \fu^3.
\end{equation}
We now define a ``cubic distance''
\begin{equation}
  \label{eq:fd}
  \fd \coloneqq (\abs{t} + 3 a^{-1} b \fu^2)^{1/2},
\end{equation}
which measures proximity to the origin in the intrinsic scaling of the nondegenerate cubic preshock $\cub$.
The coefficients are chosen so that $\fd = \cubder^{-1/2}$ for
\begin{equation}
  \label{eq:fm}
  \fm \coloneqq a \partial_x \fu = \frac{1}{\abs{t} + 3 a^{-1} b \fu^2}.
\end{equation}
For future use, we record some relationships between these functions.
\begin{lemma}
  \label{lem:fuest}
  We have $t = - \fm^{-1} + 3 a^{-1}b \fu^2$,  $\rd_x \fu = a^{-1} \fm$, and $\rd_t \fu = a \fu \rd_x \fu = \fu \fm$.
  Moreover, in $\R_{\leq 0} \times \R$,
  \begin{enumerate}[label = \textup{(\roman*)}, itemsep = 2pt]
  \item $|\fu| \asymp |x|^{\f 1 3}$ for $\abs{t} \ls |x|^{\f 2 3}$ and $|\fu| \asymp |x| |t|^{-1}$ for $|x|^{\f 2 3} \ls \abs{t}$;
  \item $\fm \ls |\fu| |x|^{-1}$;
  \item $\fd \asymp |t|^{\f 1 2}$ for $|x|^{\f 2 3} \ls \abs{t}$ and $\fd \asymp |x|^{\f 1 3}$ for $|x|^{\f 2 3} \gs \abs{t}$.
  \end{enumerate}
\end{lemma}
\noindent
This consolidates Lemmas~8.4 and 8.5 in~\cite{AndChaGra25a}.

The functions $\cub,\cubder,$ and $\dist$ are self-similar under an anisotropic scaling of time and space.
We introduce a nonstandard notion of homogeneity to capture this structure:
\begin{definition}
  \label{def:rhom}
  Given $r \in \R$, a function $f \colon \R_{\le 0} \times \R \rightarrow \R$ is \emph{$r$-homogeneous} if
  \begin{equation*}
    f(\lambda^2 t,\lambda^3 x) = \lambda^r f(t,x) \ForAll \lambda > 0, t \leq 0, \text{ and } x \in \R.
  \end{equation*}
\end{definition}
\noindent
Under this definition, $\cub$ and $\dist$ are 1-homogeneous while $\cubder$ is $-2$-homogeneous.

If an $r$-homogeneous function $f$ is continuous away from the origin, it satisfies $\abs{f} \leq C(f) \dist^r$ for $C(f) \coloneqq \sup_{\{\dist = 1\}} \abs{f}$.
Thus $\dist$ and its powers conveniently control homogeneous expressions.
With this in mind, we state the precise form of \ref{hyp:expansion-informal}:
\begin{enumerate}[label = (H6)]
\item
  There exists a sequence $(\bar{\psi}_\ell)_{\ell \in \N_0}$ of $\m{V}$-valued $(\ell + 1)$-homogeneous functions such that $\bar{\psi}_\ell$ is a polynomial in $(t, \cub, \cubder)$, $\psi_0 = \cub e_1$, and for all $L,m,n \in \N_0$,
  \begin{equation*}
    \partial_t^m \partial_x^n \psi^{(0)} = \partial_t^m \partial_x^n \sum_{\ell = 0}^L \bar{\psi}_\ell + \m{O}_{L,m,n}(\dist^{L - 2m - 3n + 2}).
  \end{equation*}
  \label{hyp:expansion}
\end{enumerate}
\noindent
Note that $\partial_t$ and $\partial_x$ reduce homogeneity by $2$ and $3$, respectively.
This explains the exponent in the error term above.

We now consider some consequences of \ref{hyp:expansion}.
Recalling the notation $\psi = (\sigma, \omega)$ for the shocking and nonshocking components, \ref{hyp:expansion} asserts that $\sigma$ is much larger that $\omega$ near the origin.
We therefore project \eqref{eq:main} onto $e_1$ and extract the leading part of each term.
Because $A_1^1(0) = 0$, we linearize and keep only the shocking portion: $A_1^1(\psi) \approx \partial_1 A_1^1(0) \sigma$, where $\partial_1$ denotes the derivative in $\m{V}$ in the direction $e_1$.
Taking $\nu = 0,$ \eqref{eq:main} approximately reduces to the scalar Burgers equation:
\begin{equation}
  \label{eq:shocking-Burgers}
  \partial_t \sigma + \partial_1 A_1^1(0) \sigma \partial_x \sigma \approx 0.
\end{equation}
Substituting the leading part $\cub$ of $\sigma$, we find
\begin{equation}
  \label{eq:cubic-Burgers}
  \partial_t \cub + \partial_1 A_1^1(0) \cub \partial_x \cub = 0.
\end{equation}
Substituting \eqref{eq:fu}, we find $a = -\partial_1 A_1^1(0)$.
The former is nonzero by hypothesis, so $\partial_1 A_1^1(0) \neq 0$.

Now if $\lambda_{(1)}$ denotes the eigenvalue of $A$ with eigenvector $e_1$ at $0 \in \m{V}$, we can check that
\begin{equation}
  \label{eq:coeff}
  \partial_1 A_1^1(0) = \partial_1 \lambda_{(1)}(0) = -a.
\end{equation}
That is, $\partial_1 A_1^1(0)$ is the derivative of the shocking eigenvalue in the direction of the shocking eigenvector.
Thus $\partial_1 A_1^1(0) \neq 0$ is equivalent to Lax's \emph{genuine nonlinearity}~\cite{Lax64} for $\lambda_{(1)}$.
Hence the genuine nonlinearity of $\lambda_{(1)}$ follows from the leading cubic behavior in \ref{hyp:expansion}.

\subsection{Inner scaling}

To determine where the viscosity is non-perturbative, we reduce \eqref{eq:main} to its shocking component as in \eqref{eq:shocking-Burgers}, but now keep $\nu > 0$ to obtain the \emph{viscous} Burgers equation:
\begin{equation}
  \label{eq:viscous-Burgers}
  \partial_t \sigma + \partial_1 A_1^1(0) \sigma \partial_x \sigma \approx \nu B_1^1(0) \partial_x^2 \sigma.
\end{equation}
We observe that $B_1^1$ is the only entry of $B$ appearing at leading order, and we have assumed $B_1^1(0) > 0$ in \ref{hyp:shocking-diffusion}.
Thus at leading order we have a scalar conservation law with nondegenerate viscosity.
Using $\cub$ as an ansatz for $\sigma$, we can formally check that the diffusion $\nu \partial_x^2 \sigma$ begins to dominate the advection $\sigma \partial_x \sigma$ where $\abs{t} \lesssim \nu^{1/2}$ and $\abs{x} \lesssim \nu^{3/4}$.
There $\abs{\cub} \lesssim \nu^{1/4}$, which motivates our definition of the inner coordinates
\begin{equation}
  \label{eq:inner-coord}
  T \coloneqq \nu^{-1/2} t, \quad X \coloneqq \nu^{-3/4} x, \And \Psi \coloneqq \nu^{-1/4} \psi.
\end{equation}
Writing $\Psi = (\Sigma, \Omega)$ for the scaled shocking a nonshocking components, \eqref{eq:viscous-Burgers} becomes
\begin{equation}
  \label{eq:viscous-shocking-Burgers}
  \partial_T \Sigma + \partial_1 A_1^1(0) \Sigma \partial_X \Sigma \approx B_1^1(0) \partial_X^2 \Sigma.
\end{equation}
This illustrates the utility of the inner coordinates: in these variables, the non-perturbative nature of the diffusion becomes evident.

Precisely, the diffusion is leading order in the ``diffusive region'' where $\abs{t} \lesssim \nu^{1/2}$ and $\abs{x} \lesssim \nu^{3/4}$, or equivalently, $\dist \lesssim \nu^{1/4}$.
However, the system \eqref{eq:main} presents an important complication relative to a scalar conservation law like Burgers.
The diffusion acts primarily in the anisotropic region $\{\abs{t} \lesssim \nu^{1/2}, \abs{x} \lesssim \nu^{3/4}\} = \{\dist \lesssim \nu^{1/4}\}$, but transverse nonshocking characteristics sweep its effects to the larger isotropic region $\{\abs{t},\abs{x} \lesssim \nu^{1/2}\} = \{\euc \lesssim \nu^{1/2}\}$, as in Figure~\ref{fig:sweep}.
In hyperbolic terminology, the diffusive zone $\{\dist \lesssim \nu^{1/4}\}$ is not causally closed, so its effects are felt throughout its much larger causal closure.
In particular, the outer expansion \eqref{eq:outer-intro} is invalid where $\euc \lesssim \nu^{1/2}$, although the diffusion itself remains relatively small in much of this region.
For this reason, we use the outer expansion where $\euc \geq \nu^\beta$ and the inner where $\euc \leq \nu^\beta$ for some $\beta < 1/2.$
\begin{figure}[t]
  \centering
  \includegraphics[width = 0.85\linewidth]{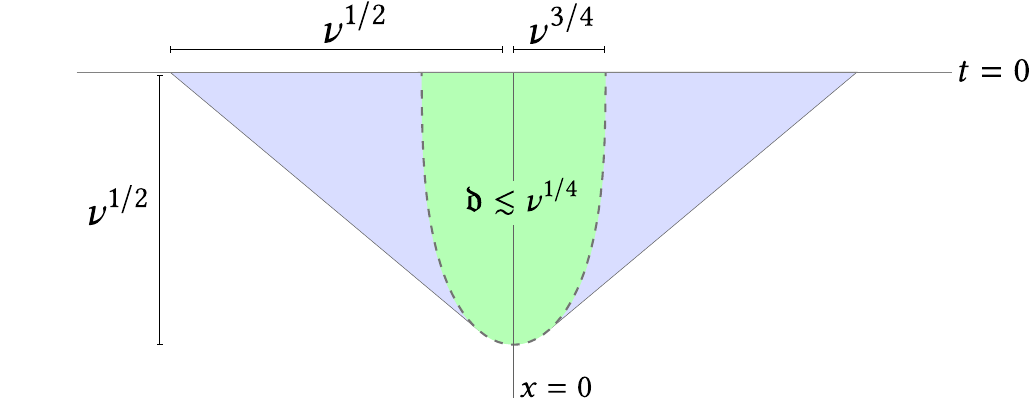}
  \caption[The diffusive zone and its causal closure.]{Due to transverse advection, the anisotropic diffusive zone $\{\dist \lesssim \nu^{1/4}\}$ in green influences a much larger region, indicated here as two blue triangular ``shadows'' cast to either side.}
  \label{fig:sweep}
\end{figure}
The interplay between anisotropic diffusive activity and the nonshocking characteristics poses a consistent technical challenge throughout this work.

Motivated by \eqref{eq:viscous-shocking-Burgers}, we let the leading part $\inn{\Sigma}{0}$ of $\Sigma$ in the inner expansion \eqref{eq:inner-intro} satisfy the viscous Burgers equation.
The nonshocking components are of lower order, so we set the corresponding nonshocking component $\inn{\Omega}{0}$ to zero.
The higher-order correctors $\inn{\Psi}{\ell} = (\inn{\Sigma}{\ell}, \inn{\Omega}{\ell})$ satisfy constant-coefficient transport equations and linearizations of viscous Burgers (Propositions~\ref{prop:Omega-inner} and \ref{prop:Sigma-inner}):
\begin{align*}
  \nu^{1/4}\partial_T \inn{\Omega}{\ell} + A(0) \partial_X \inn{\Omega}{\ell} &= \cdots\\
  \partial_T \inn{\Sigma}{\ell} + \partial_1 A_1^1(0) \partial_X(\inn{\Sigma}{0} \inn{\Sigma}{\ell}) &= B_1^1(0) \partial_X^2 \inn{\Sigma}{\ell} + \cdots
\end{align*}
Ellipses denote forcing that depends on earlier terms in the inner expansion.
These equations largely decouple, so the inner analysis is effectively scalar.

\subsection{Matching}

The most delicate task in our approximation scheme is the matching between the inner and outer expansions.
We need to choose data for the inner equations so that partial sums of the two expansions agree to high order, and we need a method to prove this agreement.
As a direct approach, we might attempt to analyze an equation satisfied by the difference of the outer and inner partial sums.
This proves challenging, however.
Each expansions has a strength and a drawback: the outer is hyperbolic but fully coupled, while the inner is decoupled but parabolic.
The difference between the two expansions would suffer the drawbacks of both, solving a fully coupled degenerate parabolic equation.

To circumvent this difficulty, we introduce an auxiliary doubly-indexed expansion $(\psi_{k,\ell})_{k,\ell \in \N_0}$ that we term the ``grid.''
Let $\inn{\psi}{\ell}$ denote a blow-down of $\inn{\Psi}{\ell}$ to the original coordinates $(t, x)$.
Inspired by the homogeneous expansion $(\bar{\psi}_\ell)$ for $\out{\psi}{0},$ we design $(\psi_{k,\ell})$ so that in the ``matching zone'' $\euc \asymp \nu^\beta$,
\begin{equation}
  \label{eq:grid-informal}
  \out{\psi}{k} \approx \psi_{k,0} + \psi_{k,1} + \ldots \And \inn{\psi}{\ell} \approx \psi_{0,\ell} + \nu \psi_{1,\ell} + \ldots
\end{equation}
(See Propositions~\ref{prop:outer-matching} and \ref{prop:inner-matching}.)
As a result, the outer and inner expansions are composed of the same underlying building blocks, and they therefore match to a high degree (Proposition~\ref{prop:outer-inner-matching}).
Thus the grid reduces the fine cancelation between the outer and inner expansions to the increasing smallness of terms in \eqref{eq:grid-informal}.
We can accomplish this while letting $\psi_{k,\ell}$ satisfy simple \emph{decoupled hyperbolic} equations (Propositions~\ref{prop:omega-grid} and \ref{prop:sigma-grid}).
The grid thus enjoys the advantages of both expansions and the disadvantages of neither.

Bounding the grid terms $(\psi_{k,\ell})$ proves to be the heart of our analysis (Proposition~\ref{prop:grid-est}).
We carefully track the interplay between homogeneous and isotropic effects, and employ elaborate bookkeeping to treat the double index.
Once we have controlled the grid, matching between the outer and inner expansions follows from \eqref{eq:grid-informal}.
This allows us to smoothly switch from one to the other in the intermediate region $\{\euc \asymp \nu^\beta\}$ in our approximate solution $\app{\psi}{K}$.
If we include enough terms $(K \gg 1)$, this approximate solution is highly accurate, in the sense that it solves \eqref{eq:main} with minuscule residue (Proposition~\ref{prop:diffeo-approx}).

\subsection{Closure}

To complete our analysis, we show that our approximate solution does in fact approximate the true solution.
This is similar to Cauchy stability from standard well-posedness theory, with the main difficulty coming from the degenerate parabolic nature of the equations.
We draw on energy estimates from Serre's approach to local well-posedness~\cite{Serre10b} to close \eqref{eq:main} in arbitrarily high Sobolev norms (Theorem~\ref{thm:small}).
Theorem~\ref{thm:approx} follows.

Corollaries~\ref{cor:rate}--\ref{cor:universal} reflect the behavior of $\psi^{(\nu)}$ in the diffusive zone $\{\dist \lesssim \nu^{1/4}\}$.
After the scaling \eqref{eq:inner-coord}, $\Psi^{(\nu)}(T,X) \sim \inn{\Psi}{0}(T,X)$ while $\Psi^{(0)}(T,X) \sim \cub(T,X)$.
The difference $\inn{\Psi}{0} - \cub$ is order 1, which corresponds to $\psi^{(\nu)} - \psi^{(0)} \asymp \nu^{1/4}$.
Thus the $L_x^\infty$-rate in Corollary~\ref{cor:rate} is sharp.

By \eqref{eq:inner-coord}, this $\nu^{1/4}$ difference persists over a spatial region of scale $\nu^{3/4}$.
As a result, the difference is small in $\m{C}_x^\al$ if and only if $\al < 1/3$.
The nonshocking component $\omega^{(0)}$ has leading homogeneity 2 rather than 1, so similar reasoning yields a H\"older threshold of $2/3$ for $\omega$, as in Corollary~\ref{cor:sharp-conv}.

Finally, $\Psi^{(\nu)} \sim \inn{\Psi}{0}$.
As noted above, we let $\inn{\Psi}{0}$ solve viscous Burgers.
In fact, we choose the unique solution of viscous Burgers constructed in~\cite{CG23} that matches $\cub$ to a high degree at infinity.
Thus, the limit of $\Psi^{(\nu)}$ is unique, modulo scaling symmetries corresponding to the constants $a$ and $b$ in \eqref{eq:fu} and $B_1^1(0)$ in \eqref{eq:viscous-shocking-Burgers}.
This explains Corollary~\ref{cor:universal}.

\section{The outer expansion}
\label{sec:outer}

In this section, we begin to study the inviscid limit $\nu \to 0$ in \eqref{eq:main}.

\subsection{A standard expansion}
As described in Section~\ref{sec:proof-over}, we expect $\psi^{(\nu)}$ to admit a standard asymptotic expansion in $\nu$ sufficiently far from shock formation:
\begin{equation}
  \label{eq:outer-expansion}
  \psi^{(\nu)} \sim \sum_{k \geq 0} \nu^k\out{\psi}{k}.
\end{equation}
We term this the outer expansion.
Because the data $\mr{\psi}$ is independent of $\nu$, we take
\begin{equation*}
  \out{\psi}{0}(t_0, \anon) = \mr{\psi} \And \out{\psi}{k}(t_0, \anon) = 0 \ForAll k \geq 1.
\end{equation*}
We can derive evolution equations for $\out{\psi}{k}$ by substituting \eqref{eq:outer-expansion} in \eqref{eq:main} and formally collecting terms of like order in $\nu$.
The first outer term $\out{\psi}{0}$ coincides with the inviscid solution $\psi^{(0)}$:
\begin{equation}
  \label{eq:inviscid}
  \partial_t \out{\psi}{0} + A(\out{\psi}{0}) \partial_x \out{\psi}{0} = 0, \quad \out{\psi}{0}(t_0, \anon) = \mr{\psi}.
\end{equation}
By \ref{hyp:gauge} and \eqref{eq:time-shift}, the earliest shock forms at the spacetime origin, $\out{\psi}{0}(0,0) = 0$, and $\lambda_{(I_0)}(0) = 0$.

To express the evolution equations for $k \geq 1$, we introduce some convenient notation.
Let $\Der A$ denote the 3-tensor field defined by
\begin{equation*}
  [(\mathrm{D} A)(\psi) v w]^I \coloneqq (\partial_K A_J^I)(\psi) v^K w^J.
\end{equation*}
We extend this in a natural fashion to higher derivatives $\Der^r A$.
Given functions $g_1,\ldots,g_k$, let $\func[g_1,\ldots,g_k]$ denote a function of $g_1,\ldots,g_k$ and derivatives thereof.
We use brackets to emphasize the presence of derivatives, and we allow the function to change from instance to instance.
We write $\out{\psi}{<k}$ for $(\out{\psi}{0},\ldots,\out{\psi}{k-1})$ and trust that similar expressions will be clear from context.
Then we can write
\begin{equation*}
  \partial_t \out{\psi}{k} + A(\out{\psi}{0}) \partial_x \out{\psi}{k} + (\Der A)(\out{\psi}{0}) \out{\psi}{k}\partial_x \out{\psi}{0} = \func[\out{\psi}{<k}], \quad \out{\psi}{k}(t_0, \anon) = 0.
\end{equation*}
In full:
\begin{proposition}
  \label{prop:outer-eq}
  For all $k \in \N$, $\out{\psi}{k}$ satisfies
  \begin{equation}
    \label{eq:outer-eq}
    \begin{aligned}
    \partial_t \out{\psi}{k} + \partial_x[A(\out{\psi}{0}) \out{\psi}{k}] = - \partial_x &\sum_\bigsub{r \geq 1, \, \kappa_j \geq 1\\ \kappa_0 + \ldots + \kappa_r = k} \frac{\Der^rA(\out{\psi}{0})}{(r + 1)!} \prod_{j=0}^r \out{\psi}{\kappa_j}
    \\ &+ \partial_x \sum_\bigsub{r \geq 0, \, \kappa_{\geq 1} \hspace{0.5pt}\geq 1\\ \kappa_0 + \ldots + \kappa_r = k-1} \frac{\Der^rB(\out{\psi}{0})}{r!} (\partial_x \out{\psi}{\kappa_0}) \prod_{j=1}^r \out{\psi}{\kappa_j}.
    \end{aligned}
  \end{equation}
\end{proposition}
Precisely, we contract the $(r + 2)$-tensor $\Der^r A$ against the $r + 1$ tensor $\prod_{j=0}^r \out{\psi}{\kappa_j}$ to obtain a vector (and similarly in the $B$ term).
We de-emphasize this tensorial structure in the sequel.
\begin{proof}
  Recalling that $A(\psi) \partial_x \psi = \partial_x[f(\psi)]$ for the flux function $f,$ we formally assume \eqref{eq:outer-expansion} and expand a Taylor series for $f$:
  \begin{align*}
    \partial_x\Big[f\Big(\out{\psi}{0} + \sum_{k \geq 1} \nu^k \out{\psi}{k}\Big) - f(\out{\psi}{0})\Big] &= \partial_x \sum_{r \geq 0} \frac{\Der^{r + 1}f(\out{\psi}{0})}{(r + 1)!} \Big(\sum_{k \geq 1} \nu^k \out{\psi}{k}\Big)^{r + 1}\\
                                                                                                          &= \partial_x \sum_\bigsub{r \geq 0\\ \kappa_0, \ldots, \kappa_r \geq 1} \nu^{\kappa_0 + \ldots + \kappa_r} \frac{\Der^{r + 1}f(\out{\psi}{0})}{(r + 1)!} \prod_{j=0}^r \out{\psi}{\kappa_j}.
  \end{align*}
  Using $\Der^{r + 1}f = \Der^r A$ and extracting the coefficient of $\nu^k$, we obtain the expression
  \begin{equation*}
    \partial_x \sum_\bigsub{r \geq 0, \, \kappa_j \geq 1\\ \kappa_0 + \ldots + \kappa_r = k} \frac{\Der^rA(\out{\psi}{0})}{(r + 1)!} \prod_{j=0}^r \out{\psi}{\kappa_j} = \partial_x[A(\out{\psi}{0}) \out{\psi}{k}] +
    \partial_x \sum_\bigsub{r \geq 1, \, \kappa_j \geq 1\\ \kappa_0 + \ldots + \kappa_r = k} \frac{\Der^rA(\out{\psi}{0})}{(r + 1)!} \prod_{j=0}^r \out{\psi}{\kappa_j}.
  \end{equation*}
  Recalling \eqref{eq:main}, this yields all but the last term in \eqref{eq:outer-eq}.
  We obtain the final term from the viscous action $\nu\partial[B(\psi) \partial_x \psi]$ in a similar fashion, and omit the details.
\end{proof}
The system \eqref{eq:outer-eq} is a linear heterogeneous transport equation and is well-posed on $[t_0, 0) \times \R$.
Hence the outer terms $\out{\psi}{k}$ are defined unambiguously.

\subsection{Shocking and nonshocking subspaces}
In subsequent work, it will be important to distinguish between shocking and nonshocking components.
For simplicity, we perform a linear change of basis in state space $\m{V}$ so that $r_I(0) = e_I$.
That is, the right eigenvectors of $A(0)$ coincide with the standard basis.
We also relabel our indices so that $I_0 = 1$.
That is, $e_1$ corresponds to the shocking characteristic.
(Hence in general, the eigenvalues are no longer ordered.)
As a consequence, $A(0)$ is diagonal and $A_1^1(0) = \lambda_{(1)}(0) = 0$.
We now introduce distinct notation for the shocking and nonshocking components:
\begin{equation}
  \label{eq:shocking-decomp}
  \sigma \coloneqq \psi^\Is \And \omega \coloneqq (\psi^2,\ldots,\psi^N).
\end{equation}
In particular, $\psi = (\sigma, \omega)^\top$.
For convenience, we let $\perp$ denote restriction to nonshocking indices, so $\omega = \psi^\sperp$.
We have decomposed the range $\m{V}$ into the one-dimensional shocking eigenspace $\R e_1$, denoted by $\m{V}^1$, and the complementary nonshocking subspace $\m{V}^\sperp \coloneqq \Span(e_2,\ldots,e_N)$.
We highlight that we ``freeze coefficients'' by decomposing relative to the matrix $A(0)$ rather than $A(\psi)$.
This choice prevents complications when derivatives land on the eigenbasis $(\lambda_{(I)}, r_I)_{I \in [N]}$.

We extend this notation to mappings from $\m{V}$.
\begin{notation}
  Let $\pi \colon \m{V} \to \m{V}^\sperp$ and $\iota \colon \m{V}^\perp \to \m{V}$ denote orthogonal projection and inclusion.
  If $f \colon \m{V} \to \m{V}$, we write $f^\sperp \coloneqq \pi f$.
  If $M \colon \m{V} \to \m{V} \otimes \m{V}^*$, we write $M_\sperp \coloneqq M \iota$, meaning $M_\sperp(p)v = M(p) \iota \, v$ for $p \in \m{V}$, $v \in \m{V}^\sperp$.
  Similarly, $M^\sperp \coloneqq \pi A$ and $M_\sperp^\sperp \coloneqq \pi M \iota$.
  If $g \colon \m{V} \to \m{W}$ for some vector space $\m{W}$, we interpret $\dn f \colon \m{V} \to \m{W} \otimes \m{V}^*$ and write $\partial_\sperp g \coloneqq (\dn g)_\sperp$.
  We use analogous notation with $1$ in place of $\perp$.
\end{notation}
In our basis, projections to $\m{V}^1$ and $\m{V}^\sperp$ represent projections onto the first component and remaining components, respectively.
Hence for a matrix $M \in \m{V} \otimes \m{V}^*$ acting on $\m{V}$, the notation $M_\sperp$ and $M^\sperp$ restricts to nonshocking indices of columns and rows, respectively (and similarly for $M_1,M^1$).

\subsection{Outer estimates}

The quality of the approximation \eqref{eq:outer-expansion} is determined by the size of $\out{\psi}{k}$ and its derivatives, particularly near the preshock.
By hypothesis \ref{hyp:expansion}, the first term $\out{\psi}{0}$ is well-approximated by $\cub$ near the origin.
We therefore bound $\out{\psi}{0}$ and its derivatives by suitable powers of the distance $\dist$ from \eqref{eq:fd}.

We take a similar approach to the higher-order outer correctors $\out{\psi}{k}$.
Each appears homogeneous at leading order, but the remainder takes a more complex form due to the interplay between the viscosity and the nonshocking characteristics.
The homogeneous distance function $\dist$ is highly anisotropic due to the different scaling exponents in $t$ and $x$.
Ultimately, this reflects our convention that $\lambda(0) = 0$, which lends the shocking dynamics a distinguished $t$ direction.
In contrast, nonshocking characteristics scale similarly in $t$ and $x$.
To capture their effects, we use the isotropic Euclidean distance $\euc$ from \eqref{eq:Euclidean}.
The two distance functions are related in the following manner:
\begin{equation}
  \label{eq:distance-relation}
  \dist^{-2} \lesssim \euc^{-1} \And \dist \lesssim \euc^{1/3}.
\end{equation}
We express most of the coming estimates in term of these distinct measures of proximity to the origin.
For example,
\begin{proposition}
  \label{prop:outersize}
  For all $k, m, n \in \N_0$, where $\euc \leq 1/2$ we have
  \begin{equation}
    \label{eq:outerest1}
    |\partial_t^m \partial_x^n \out{\psi}{k}| \lesssim \dist^{-4k - 2 m - 3 n + 1} + \abs{\log \euc} \euc^{-2 k - m + 1} \dist^{-3 n}.
  \end{equation}
\end{proposition}
\noindent
As our problem is fundamentally local, we always restrict statements such as this to a fixed neighborhood of the origin.
We take this as a convention in the sequel: all estimates apply in the region $\{\euc \leq 1/2\}$, as the solution is uniformly smooth outside this region.

The first term in Proposition~\ref{prop:outersize} represents the leading homogeneous behavior of $\out{\psi}{k}$.
However, in some regimes this term is overwhelmed by nonshocking effects, which produce the second term in \eqref{eq:outerest1}.
We defer the proof to Section~\ref{sec:zeroth-column}, which clarifies this decomposition through Proposition~\ref{prop:out-bound}.
Using Proposition~\ref{prop:outersize}, we show that the partial sum
\begin{equation*}
  \out{\psi}{[K]} \coloneqq \sum_{k=0}^K \nu^k \out{\psi}{k}
\end{equation*}
is a good approximate solution of \eqref{eq:main} provided $K \gg 1$ and we are not too close to the singularity at the origin.
\begin{proposition}
  \label{prop:outer-approx}
  For all $K \in \N$, $n \in \N_0$, and $\beta \in (3/8, 1/2)$, we can write
  \begin{equation*}
    \partial_t \out{\psi}{[K]} + A(\out{\psi}{[K]}) \partial_x \out{\psi}{[K]} = \nu \partial_x\big[B(\out{\psi}{[K]} \partial_x \out{\psi}{[K]})\big] + \out{F}{K}
  \end{equation*}
  for forcing $\out{F}{K}$ satisfying $\partial_x^n \out{F}{K} \lesssim_{K,n} \nu^{(1 - 2 \beta)K - 3n/4 - 3/4}$ where $\euc \geq \nu^\beta$.
\end{proposition}
To express the proof concisely, we introduce some multi-index notation.
Given $r \in \N_0$ and ${\kappa = (\kappa_0, \ldots, \kappa_r) \in \N_0^{r + 1}}$, we let $\abs{\kappa} \coloneqq \kappa_0 + \ldots + \kappa_r$ and $\# \kappa \coloneqq r$.
\begin{proof}
  Let $\out{\psi}{[1,K]} \coloneqq \out{\psi}{[K]} - \out{\psi}{0}$ and suppose $\euc \geq \nu^\beta$.
  Proposition~\ref{prop:outersize} implies that $\out{\psi}{[1,K]} \lesssim_K \nu \dist^{-3} \lesssim_K \nu^{1/4}$ in this region.
  After all, \eqref{eq:distance-relation} yields
  \begin{equation}
    \label{eq:dist-est}
    \dist^{-1} \lesssim \euc^{-1/2} \leq \nu^{-\beta/2} < \nu^{-1/4}.
  \end{equation}
  Recall that $A$ is the derivative of the flux function $f$.
  Using Taylor's theorem and Proposition~\ref{prop:outersize}, we write
  \begin{equation*}
    f(\out{\psi}{[K]}) - f(\out{\psi}{0}) = \sum_{r = 0}^K \frac{\Der^r A(\out{\psi}{0})}{(r + 1)!} \big(\out{\psi}{[1,K]})^{r+1} + \m{O}_K\big(\nu^{(K+2)/4}\big).
  \end{equation*}
  Moreover,
  \begin{align*}
    \sum_{r = 0}^K \frac{\Der^r A(\out{\psi}{0})}{(r + 1)!} \big(\out{\psi}{[1,K]})^{r+1} = \sum_\bigsub{1 \leq \kappa_j \leq K\\ \abs{\kappa} \leq K} \frac{\Der^{\#\kappa} A(\out{\psi}{0})}{(\#\kappa + 1)!} &\nu^{\abs{\kappa}} \prod_j \out{\psi}{\kappa_j}\\[-8pt]
    &+ \sum_\bigsub{1 \leq \kappa_j \leq K\\ \#\kappa \leq K\\ \abs{\kappa} > K} \frac{\Der^{\#\kappa} A(\out{\psi}{0})}{(\#\kappa + 1)!} \nu^{\abs{\kappa}} \prod_j \out{\psi}{\kappa_j}.
  \end{align*}
  The first sum on the right agrees with the sum of the $A$-terms in \eqref{eq:outer-eq} over $1 \leq k \leq K$, provided we weight $\out{\psi}{k}$ by $\nu^k$.
  For the second sum on the right, Proposition~\ref{prop:outersize} implies that
  \begin{equation*}
    \nu^{\abs{\kappa}} \prod_j \out{\psi}{\kappa_j} \lesssim_{\kappa} (\nu \euc^{-2})^{\abs{\kappa}} \dist^{\#\kappa + 1} \lesssim_\kappa \nu^{(1 - 2\beta)\abs{\kappa}} \dist^{\#\kappa + 1}.
  \end{equation*}
  In this sum we have $\#\kappa \geq 1$ and $\abs{\kappa} \geq K + 2$, so the sum is bounded by $\nu^{(1 - 2\beta)(K + 2)} \dist^2$.
  When we differentiate the flux in $x$, we pick up a factor of $\dist^{-3}$ (see Proposition~\ref{prop:outersize}) and obtain error bounded by $\nu^{(1 - 2\beta)K - 1/4}$, due to \eqref{eq:dist-est}.
  Noting that $1 - 2\beta < 1/4$, this is larger than the $\nu^{(K + 2)/4}$ Taylor-error from earlier.
  Performing similar computations for the diffusive terms involving $B$, we find slightly smaller error of order $\nu^{(1 - 2\beta)K}$.
  Thus $\out{F}{K} \lesssim_K \nu^{(1 - 2\beta)K - 1/4}$ where $\euc \geq \nu^\beta$.
  Estimates on derivatives follow in a similar manner, using $\dist^{-3} \lesssim \nu^{-3/4}$.
\end{proof}

\section{Inner and grid expansions}
\label{sec:inner}

We have now constructed an outer expansion that should closely approximate the full solution $\psi^{(\nu)}$ a certain distance from the preshock.
In this section, we show how to treat $\psi^{(\nu)}$ near the preshock, where viscous effects are nonperturbative.

By \ref{hyp:expansion}, $\out{\psi}{0} \approx (\cub, 0)$ near the origin, where $\cub$ satisfies the cubic equation \eqref{eq:fu}.
We thus expect viscous effects to matter where $\nu \abss{\partial_x^2 \cub} \gtrsim \abss{\cub \partial_x \cub}$ and to primarily act on the shocking component.
Unraveling the structure of $\cub$ from \eqref{eq:fu}, we can check that viscosity matters when $\abs{t} \lesssim \nu^{1/2}$ and $\abs{x} \lesssim \nu^{3/4}$, and there $\abss{\cub} \lesssim \nu^{1/4}$.
To study $\psi$ in this small region, we perform the blow-up from \eqref{eq:inner-coord}:
\begin{equation*}
  T \coloneqq \nu^{-1/2} t, \quad X \coloneqq \nu^{-3/4} x, \And \Psi \coloneqq \nu^{-1/4} \psi.
\end{equation*}
In these inner coordinates, \eqref{eq:main} becomes
\begin{equation}
  \label{eq:inner}
  \partial_T \Psi + \nu^{-1/4} A(\nu^{1/4} \Psi) \partial_X \Psi = \partial_X\big[B(\nu^{1/4}\Psi) \partial_X \Psi\big].
\end{equation}
Due to the powers of $\nu^{1/4}$ appearing in \eqref{eq:inner}, we posit the inner ansatz
\begin{equation}
  \label{eq:inner-ansatz}
  \Psi \sim \sum_{\ell \geq 0} \nu^{\ell/4} \inn{\Psi}{\ell}.
\end{equation}
Unlike the outer expansion \eqref{eq:outer-expansion}, our ansatz will be ``criminal''~\cite{ALKR22} in the sense that the terms $\inn{\Psi}{\ell}$ themselves depend on $\nu$.
Thus \eqref{eq:inner-ansatz} is not a true expansion in $\nu$.
This concession allows us to overcome a difficulty with the nonshocking components that we explain below.

\subsection{Inner equations}

Following \eqref{eq:shocking-decomp}, we set
\begin{equation*}
  \Sigma \coloneqq \Psi^\Is \And \Omega \coloneqq \Psi^\sperp = (\Psi^2, \ldots, \Psi^N).
\end{equation*}
As in \eqref{eq:inner}, we develop inner ansatzes for $\Sigma$ and $\Omega$ of the form
\begin{equation}
  \label{eq:inner-components}
  \Sigma \sim \sum_{\ell \geq 0} \nu^{\ell/4} \inn{\Sigma}{\ell} \And \Omega \sim \sum_{\ell \geq 0} \nu^{\ell/4} \inn{\Omega}{\ell}.
\end{equation}
Due to our choice $\lambda_{(\Is)}(0) = 0$ in \ref{hyp:gauge}, the shocking component lies in the nullspace of $A(0)$.
Thus formally
\begin{equation}
  \label{eq:small-advection}
  \nu^{-1/4} A^\Is(\nu^{1/4} \Psi) \partial_X \Psi = (\Der A^\Is)(0) \Psi \partial_X \Psi + \m{O}(\nu^{1/4}),
\end{equation}
while
\begin{equation}
  \label{eq:large}
  \nu^{-1/4} A^\sperp(\nu^{1/4} \Psi) \partial_X \Psi = \nu^{-1/4} A_\sperp^\sperp(0) \partial_X \Omega + \m{O}(1).
\end{equation}
Strict hyperbolicity implies that the restriction $A^\sperp_\sperp$ of $A$ to the nonshocking subspace is invertible.
It follows that the advection term in \eqref{eq:inner} is of a different order in the shocking and nonshocking components.
This term dominates all others in the nonshocking direction, while it balances others in the shocking equation.
This asymmetry leads to a triangular structure: $\inn{\Omega}{\ell}$ does not depend on $\inn{\Sigma}{\ell}$, while $\inn{\Sigma}{\ell}$ does depend on $\inn{\Omega}{\ell}$.
We thus first estimate $\inn{\Omega}{\ell}$ and then $\inn{\Sigma}{\ell}$.
We use this structure heavily.

We take a ``lawful'' approach to the evolution equations for the terms $\inn{\Sigma}{\ell}$ in \eqref{eq:inner-components}.
We project \eqref{eq:inner} onto the span of $e_\Is$, Taylor expand the nonlinearities, and collect terms of order $\nu^{\ell/4}$.
Our criminality manifests in our treatment of the nonshocking components.
To account for the $\nu^{-1/4}$ prefactor on the right side of \eqref{eq:large}, we project onto the nonshocking subspace and multiply \eqref{eq:inner} by $\nu^{1/4}$:
\begin{equation}
  \label{eq:inner-small}
  \nu^{1/4} \partial_T \Omega + A^\sperp(\nu^{1/4} \Psi) \partial_X \Psi = \nu^{1/4}\partial_X\big[B^\sperp(\nu^{1/4}\Psi) \partial_X \Psi\big].
\end{equation}
The evolution equation for $\inn{\Omega}{\ell}$ consists of the order-$\nu^{\ell/4}$ terms in \eqref{eq:inner-small} \emph{and} the time derivative $\nu^{(\ell + 1)/4} \partial_T \inn{\Omega}{\ell}$.
It is our inclusion of this ``off-order'' derivative that makes the inner ansatz criminal.
It amounts to a slight tilt of the characteristics in the $\inn{\Omega}{\ell}$-equation, which maintains a causal link between the inner and outer expansions.
This avoids technical difficulties that otherwise appear.

Because the nonshocking component of $\out{\psi}{0}$ is much smaller than the shocking component near the origin, it is natural to take $\inn{\Omega}{0} = 0$.
So in fact the $\Omega$-sum in \eqref{eq:inner-components} begins from $\ell = 1$.
For $\Sigma$, we use \eqref{eq:small-advection} to extract the leading terms in \eqref{eq:inner}:
\begin{equation}
  \label{eq:s0-pre}
  \partial_T \inn{\Sigma}{0} + (\Der A^\Is)(0) \inn{\Psi}{0} \partial_X \inn{\Psi}{0} = \partial_X[B^\Is(0)\partial_X \inn{\Psi}{0}].
\end{equation}
Because $\inn{\Omega}{0} = 0$, we have $\inn{\Psi}{0} = \inn{\Sigma}{0} e_\Is$ and \eqref{eq:s0-pre} simplifies to the viscous Burgers equation:
\begin{equation}
  \label{eq:s0}
  \partial_T \inn{\Sigma}{0} + (\partial_\Is A_\Is^\Is)(0) \inn{\Sigma}{0} \partial_X \inn{\Sigma}{0} = B_\Is^\Is(0) \partial_X^2 \inn{\Sigma}{0}.
\end{equation}
As noted after \eqref{eq:cubic-Burgers}, $\partial_\Is A_\Is^\Is(0) \neq 0$.
Moreover, by \ref{hyp:shocking-diffusion}, $B_1^1 (0) > 0$.
This is the only nondegeneracy we require of $B$.

Due to these hypotheses, \eqref{eq:s0} is equivalent modulo scaling to the viscous Burgers equation satisfied by the leading inner term in~\cite{CG23}.
There, we constructed a unique solution of \eqref{eq:s0} that agrees with the cubic \eqref{eq:fu}, and thus $\inn{\psi}{0}$, sufficiently well at infinity.
We use the same solution here, which is precisely constructed in \cite[Proposition~4.2]{CG23}.
\medskip

We next describe the evolution equations for $\inn{\Omega}{\ell}$ and $\inn{\Sigma}{\ell}$ with $\ell \geq 1$.
We first write the nonshocking equation.
Because the viscous term in \eqref{eq:inner} is smaller than the nonshocking advection by a factor of $\nu^{1/4}$, the $\inn{\Omega}{\ell}$ equation does not directly involve $\partial_X^2 \inn{\Omega}{\ell}$.
Indeed, we can write
\begin{equation}
  \label{eq:nsl}
  \nu^{1/4} \partial_T \inn{\Omega}{\ell} + A_\sperp^\sperp(0) \partial_X \inn{\Omega}{\ell} = \func[\inn{\Psi}{<\ell}].
\end{equation}
We recall that $\func[\inn{\Psi}{<\ell}]$ denotes a function of $\inn{\Psi}{0},\ldots,\inn{\Psi}{\ell-1}$ and their derivatives.
Assuming we have already determined the terms $\inn{\Psi}{\ell'}$ for $\ell' < \ell$, we can treat the expression on the right as forcing.
So \eqref{eq:nsl} is a linear constant-coefficient transport equation.
For example,
\begin{equation*}
  \nu^{1/4} \partial_T \inn{\Omega}{1} + A_\sperp^\sperp(0) \partial_X \inn{\Omega}{1} = -(\partial_\Is A_\Is^\sperp)(0) \inn{\Sigma}{0} \partial_X \inn{\Sigma}{0} + B_\Is^\sperp(0) \partial_X^2 \inn{\Sigma}{0} = \func[\inn{\Psi}{0}].
\end{equation*}
We now turn to $\inn{\Sigma}{\ell}$.
Its equation is somewhat more complicated, as it includes contributions from $\inn{\Omega}{\ell}$ and involves the viscosity directly.
\begin{equation}
  \label{eq:sl-pre}
  \begin{aligned}
    \hspace{-3pt}\partial_T \inn{\Sigma}{\ell} + (\partial_\Is A_\Is^\Is)(0) \partial_X(\inn{\Sigma}{0}\inn{\Sigma}{\ell}) + (\partial_\sperp& A_\Is^\Is)(0) \inn{\Omega}{\ell} \partial_X \inn{\Sigma}{0} + (\partial_\Is A_\Is^\sperp)(0)\inn{\Sigma}{0} \partial_X \inn{\Omega}{\ell}\\
                                                                                                                                             &= B_\Is^\Is(0) \partial_X^2 \inn{\Sigma}{\ell} + B_\sperp^\Is(0) \partial_X^2 \inn{\Omega}{\ell} + \func[\inn{\Psi}{<\ell}].
  \end{aligned}
\end{equation}
Provided we determine the nonshocking component $\inn{\Omega}{\ell}$ first, we can fold its terms into our known forcing.
So \eqref{eq:sl-pre} becomes a linear advection-diffusion equation:
\begin{equation}
  \label{eq:sl}
  \partial_T \inn{\Sigma}{\ell} + (\partial_\Is A_\Is^\Is)(0) \partial_X(\inn{\Sigma}{0}\inn{\Sigma}{\ell}) - B_\Is^\Is(0) \partial_X^2 \inn{\Sigma}{\ell} = \func[\inn{\Psi}{<\ell}, \inn{\Omega}{\ell}].
\end{equation}
Comparing \eqref{eq:nsl} and \eqref{eq:sl}, we see that $\inn{\Omega}{\ell}$ does not depend on $\inn{\Sigma}{\ell}$, while the latter depends on the former.
This is the triangular structure mentioned above.

In the following propositions, we record the precise nature of the forcing in \eqref{eq:nsl} and \eqref{eq:sl}.
Given a collection of functions $G_\al \coloneqq G_{\al_0},\ldots,G_{\al_r}$ taking values in $\m{V}$, we frequently treat Taylor terms of $A$ and $B$ of the form
\begin{equation}
  \label{eq:multilinear}
  \Aop G_\al \coloneqq \frac{\Der^rA(0)}{(r + 1)!} \prod_{j=0}^r G_{\al_j} \And \Bop G_\al \coloneqq \frac{\Der^rB(0)}{r!} \partial_X G_{\al_0} \prod_{j = 1}^r G_{\al_j}.
\end{equation}
The product in the first expression denotes contractions of the vectors $G_{\alpha_j}$ with derivatives of $A$.
We let $\Aop^1$ and $\Aop^\sperp$ denote the projections of the image of the multilinear operator $\Aop$ onto the shocking and nonshocking subspaces, respectively.
These involve derivatives of $A^1$ or $A^\sperp$ rather than $A$.
We treat $\Bop$ similarly.
\begin{proposition}
  \label{prop:Omega-inner}
  For all $\ell \in \N$, $\inn{\Omega}{\ell}$ satisfies
  \begin{equation}
    \label{eq:Omega-inner}
    [\nu^{1/4} \partial_T + A_\sperp^\sperp(0)] \partial_X \inn{\Omega}{\ell} = -\partial_X \sum_\bigsub{1 \leq r \leq \ell\\ \al_0 + \ldots + \al_r = \ell - r} \Aop^\sperp\inn{\Psi}{\al}
    + \partial_X \sum_\bigsub{0 \leq r \leq \ell - 1\\ \al_0 + \ldots + \al_r = \ell - r - 1} \Bop^\sperp\inn{\Psi}{\al}.
  \end{equation}
\end{proposition}
\begin{proof}
  We multiply \eqref{eq:inner} by $\nu^{1/4}$ and collect terms of order $\nu^{\ell/4}$, omitting the time derivative.
  Using the fact that $A(\Psi)\partial_X \Psi = \partial_X[f(\Psi)]$ for the flux $f$, we can write
  \begin{align*}
    A(\nu^{\frac{1}{4}}\Psi) \partial_X \Psi &= \nu^{-\frac{1}{4}} \partial_X[f(\nu^{\frac{1}{4}}\Psi)] = \nu^{-\frac{1}{4}} \partial_X f\Big(\sum_{\ell} \nu^{\frac{\ell + 1}{4}} \inn{\Psi}{\ell}\Big)
    \\ &= \partial_X \sum_\al \nu^{\frac{\abs{\al} + \# \al}{4}} \frac{\Der^{\# \al + 1}f(0)}{(\# \al + 1)!}  \prod_{j} \inn{\Psi}{\al_j} = \partial_X \sum_\al \nu^{\frac{\abs{\al} + \# \al}{4}} \frac{\Der^{\# \al}A(0)}{(\# \al + 1)!} \prod_{j} \inn{\Psi}{\al_j}.
  \end{align*}
  Grouping by $r = \# \al$, we see that terms of order $\nu^{\ell/4}$ satisfy $\al_0 + \ldots + \al_r = \ell - r$.
  After projection onto $\m{V}^\sperp$, the singleton index $\al = (\ell)$ yields $A_\sperp^\sperp(0) \partial_X \inn{\Omega}{\ell}$ in \eqref{eq:Omega-inner}, and the higher-order terms take the form $\Aop^\sperp$ from \eqref{eq:multilinear}.
  The computations for $B$ are similar, save that they include an extra factor of $\nu^{1/4}$, so their indices are shifted by one.
  Finally, we insert $\nu^{1/4} \partial_T \inn{\Omega}{\ell}$ as in the left side of \eqref{eq:Omega-inner} to match our criminal convention for $\inn{\Omega}{\ell}$.
\end{proof}
In the remainder of the paper, the sum notation $\sum\limits^*$ indicates the omission of terms that directly involve the function of interest ($\inn{\Sigma}{\ell}$ below).
\begin{proposition}
  \label{prop:Sigma-inner}
  For all $\ell \in \N$, $\inn{\Sigma}{\ell}$ satisfies
  \begin{equation}
   \label{eq:Sigma-inner}
    \partial_T \inn{\Sigma}{\ell} + \partial_\Is A_\Is^\Is(0) \partial_x(\inn{\Sigma}{0}\inn{\Sigma}{\ell}) - B_\Is^\Is(0)\partial_X^2 \inn{\Sigma}{\ell}
    = -\partial_X \sum^*_\bigsub{1 \leq r \leq \ell + 1\\ \al_0 + \ldots + \al_r = \ell - r + 1} \Aop^\Is\inn{\Psi}{\al}
    + \partial_X \sum^*_\bigsub{0 \leq r \leq \ell\\ \al_0 + \ldots + \al_r = \ell - r} \Bop^1\inn{\Psi}{\al}.
  \end{equation}
\end{proposition}
\begin{proof}
  We collect all terms of order $\nu^{\ell/4}$ in \eqref{eq:inner} and project onto $\m{V}^1$, observing that $A_1^1(0) = 0$.
  The computations are analogous to the proof of Proposition~\ref{prop:Omega-inner}.
\end{proof}
We again emphasize again that at level $\ell$ in the inner expansion, the nonshocking component $\inn{\Omega}{\ell}$ is decoupled from the shocking component $\inn{\Sigma}{\ell}$.
In contrast, the right side of \eqref{eq:Sigma-inner} includes terms of the form $\partial_x(\inn{\Sigma}{0}\inn{\Omega}{\ell})$ and $\partial_X^2 \inn{\Omega}{\ell}$, so $\inn{\Omega}{\ell}$ influences $\inn{\Sigma}{\ell}$.
We can thus construct $\inn{\Omega}{\ell}$ knowing only $\inn{\Psi}{<\ell}$, and then construct $\inn{\Sigma}{\ell}$ from both $\inn{\Psi}{<\ell}$ and $\inn{\Omega}{\ell}.$
This triangular structure means we do not have to study potentially complicated hyperbolic-parabolic systems.

\subsection{Grid equations}

To control the size of $\inn{\Psi}{\ell}$, we must explore the interplay between the outer expansion \eqref{eq:outer-expansion} and the inner ansatz \eqref{eq:inner-ansatz}.
For this purpose, we return to the original coordinates.
Set
\begin{equation*}
  \inn{\psi}{\ell}(t, x) \coloneqq \nu^{(\ell + 1)/4} \inn{\Psi}{\ell}\big(\nu^{-1/2}t, \nu^{-3/4}x\big).
\end{equation*}
We use the evolution equations for $\inn{\Psi}{\ell}$ to derive analogous equations for $\inn{\psi}{\ell}$.
At leading order, we have weakly-viscous Burgers
\begin{equation*}
  \partial_t \inn{\sigma}{0} + (\partial_\Is A_\Is^\Is)(0) \inn{\sigma}{0} \partial_x \inn{\sigma}{0} = \nu B_\Is^\Is(0) \partial_x^2 \inn{\sigma}{0},
\end{equation*}
while $\inn{\omega}{0}$ vanishes.
Using \eqref{eq:nsl} and \eqref{eq:sl}, the equations for $\ell \ge 1$ assume the following form:
\begin{equation*}
  \begin{split}
    \partial_t \inn{\omega}{\ell} + A_\sperp^\sperp(0) \partial_x \inn{\omega}{\ell} &= \func[\inn{\psi}{<\ell}],\\
    \partial_t \inn{\sigma}{\ell} + (\partial_\Is A_\Is^\Is)(0) \partial_x(\inn{\sigma}{0} \inn{\sigma}{\ell}) - \nu B_\Is^\Is(0) \partial_x^2 \inn{\sigma}{\ell} &= \func[\inn{\psi}{<\ell}, \inn{\omega}{\ell}].
  \end{split}
\end{equation*}
The forces $\func$ may depend on $\nu$ as well.

To relate the outer and inner expansions, we examine the contributions to $\inn{\psi}{\ell}$ from terms in the outer expansion.
The outer expansion arises from a Taylor expansion in $\nu$, so we expand $\inn{\psi}{\ell}$ in integer powers of $\nu$:
\begin{equation}
  \label{eq:in-grid}
  \inn{\psi}{\ell} \sim \sum_{k \geq 0} \nu^k \psi_{k,\ell},
\end{equation}
with similar decompositions for $\sigma$ and $\omega$.
Scaling the inner ansatz \eqref{eq:inner-ansatz}, we should have $\psi \sim \sum_{\ell \geq 0} \inn{\psi}{\ell}$ near the origin.
Comparing with the outer expansion \eqref{eq:outer-expansion}, we formally conclude that
\begin{equation*}
  \out{\psi}{k} \sim \sum_{\ell \geq 0} \psi_{k,\ell}.
\end{equation*}
If we arrange $(\psi_{k,\ell})_{k,\ell \geq 0}$ in a grid, then the outer and inner ansatzes amount to horizontal or vertical summation, respectively; see Figure~\ref{fig:grid}.
We therefore refer to $\psi_{k,\ell}$ as a ``grid term.''
The second and third authors employed a similar structure in an analysis of the Burgers equation~\cite{CG23}.
In the present setup, the criminal nature of the inner ansatz implies that the grid terms $\psi_{k,\ell}$ do not enjoy the homogeneity of their brethren in~\cite{CG23}.
\begin{figure}[t]
  \centering
  \begin{tikzpicture}[scale = 1]
    \tikzstyle{every node}=[font=\small]

    \def\k{3}
    \def\Xeps{0.6}
    \def\Yeps{0.6}

    \fill [green, opacity = 0.2, rounded corners=5pt] (- 0.7, -2 + 0.5) rectangle (\k + 1 + 0.5, -2 - 0.5);
    \fill [blue, opacity = 0.1, rounded corners=5pt] (1 - 0.5, 0.5) rectangle (1 + 0.5, -\k - 1 - 0.5);

    \foreach \x in {0,1,...,\k}
    {\node at (\x, 0) {$ \psi_{0, \x}$};
      \node at (\x, -1) {$\nu \psi_{1, \x}$};
      \foreach \y in {2,...,\k}
      {\node at (\x, -\y) {$\nu^{\y} \psi_{\y, \x}$};}
      \node at (\x, -\k-0.9) {$\vdots$};}
    
    \foreach \y in {0,...,\k}
    {\node at (\k + 1, -\y) {$\cdots$};}

    \node at (\k + 1, -\k - 0.8) {$\ddots$};

    \node[green!40!black] at (-1.4, -2) {$\nu^2 \out{\psi}{2}$};
    \node[blue!60!black] at (1, 1) {$\inn{\psi}{1}$};
  \end{tikzpicture}
  \caption{
    \tbf{Asymptotic structure of the inner and outer expansions:} $\nu^2\out{\psi}{2}$ is asymptotically equivalent to the sum of the green row, while $\inn{\psi}{1}$ is equivalent to the sum of the blue column.
  }
  \label{fig:grid}
\end{figure}

We find an evolution equation for $\psi_{k,\ell}$ by substituting \eqref{eq:in-grid} into the equation for $\inn{\psi}{\ell}$ and collecting terms of order $\nu^k$.
For example, the first two contributions to $\inn{\sigma}{0}$ solve
\begin{align*}
  \partial_t \sigma_{0,0} + (\partial_\Is A_\Is^\Is)(0) \sigma_{0,0} \partial_x \sigma_{0,0} &= 0,\\
  \partial_t \sigma_{1,0} + (\partial_\Is A_\Is^\Is)(0)\partial_x(\sigma_{0,0}\sigma_{1,0}) &= B_\Is^\Is(0) \partial_x^2 \sigma_{0,0}.
\end{align*}
The first equation is inviscid Burgers, and in fact the entire sequence of equations for $(\sigma_{k,0})_{k \geq 0}$ coincides with that in~\cite{CG23} for $(u_{k,0})_{k}$.
In accordance with \ref{hyp:expansion}, we naturally take $\sigma_{0,0} = \cub$.

When $\ell \geq 1$, we make one modification: the leading row $(\psi_{0,\ell})_\ell$ only approximately solves the expected equations.
This leaves a residue $F_\ell$ that we cancel in the next row $(\psi_{1,\ell})$, where it acts with strength $\nu^{-1}$.
We describe this twist in Section~\ref{sec:inviscid-row}.

In a minor abuse of notation, we use the $x$-derivative $\partial_x$ in $\Bop$ when it acts on expressions in the outer variables.
Extending the notation from \eqref{eq:multilinear} to doubly-indexed functions, we thus define
\begin{equation*}
  \Aop g_{\kappa,\al} \coloneqq \frac{\Der^rA(0)}{(r + 1)!} \prod_{j=0}^r g_{\kappa_j,\al_j} \And \Bop g_{\kappa,\al} \coloneqq \frac{\Der^rB(0)}{r!} \partial_x g_{\kappa_0,\al_0} \prod_{j = 1}^r g_{\kappa_j,\al_j}.
\end{equation*}
Whenever multi-indices $\kappa$ and $\al$ appear together in this fashion, we implicitly assume they have the same length.
With these conventions, we have:
\begin{proposition}
  \label{prop:omega-grid}
  For all $k, \ell \in \N$, there is a force $F_\ell^\omega$ defined in Proposition~\textnormal{\ref{prop:inviscid-row}} below such that
  \begin{equation}
    \label{eq:omega-grid}
    [\partial_t + A_\sperp^\sperp(0) \partial_x] \omega_{k, \ell} = -\partial_x \sum_\bigsub{1 \leq r \leq \ell\\ \kappa_0 + \ldots + \kappa_r = k\\ \al_0 + \ldots + \al_r = \ell - r} \Aop^\sperp\psi_{\kappa,\al}
    +\partial_x \sum_\bigsub{0 \leq r \leq \ell-1\\ \kappa_0 + \ldots + \kappa_r = k-1\\ \al_0 + \ldots + \al_r = \ell - r - 1} \Bop^\sperp\psi_{\kappa,\al}
    - \nu^{-1} \delta_{k,1} F_\ell^\omega.
  \end{equation}
\end{proposition}
\begin{proposition}
  \label{prop:sigma-grid}
  For $k \in \N$ and $\ell \in \N_0$, there is a force $F_\ell^\sigma$ defined in Proposition~\textnormal{\ref{prop:inviscid-row}} such that
  \begin{equation}
    \label{eq:sigma-grid}
    \partial_t \sigma_{k,\ell} + \partial_\Is A_\Is^\Is(0)\partial_x(\sigma_{0,0}\sigma_{k,\ell}) =
    -\partial_x \sum^*_\bigsub{1 \leq r \leq \ell + 1\\ \kappa_0 + \ldots + \kappa_r = k\\ \al_0 + \ldots + \al_r = \ell - r + 1} \Aop^1\psi_{\kappa,\al}
    + \partial_x \sum_\bigsub{0 \leq r \leq \ell\\ \kappa_0 + \ldots + \kappa_r = k-1\\ \al_0 + \ldots + \al_r = \ell - r} \Bop^1\psi_{\kappa,\al}
    - \nu^{-1} \delta_{k,1} F_\ell^\sigma.
  \end{equation}
\end{proposition}
\noindent
We recall that $\sum\limits^*$ indicates the omission of terms involving the function $\sigma_{k,\ell}$ of interest.
The Kronecker delta $\delta_{k,1}$ confines the force $-F_\ell$ to the row $k = 1$.
The proofs of these results are similar to those of Propositions~\ref{prop:Omega-inner} and \ref{prop:Sigma-inner}.

\subsection{Data}

We now select a domain and data for the grid equations.
Let
\begin{equation*}
  \mathsf{S} \coloneqq \max_I \abss{\lambda_{(I)}(0)} + 1,
\end{equation*}
which exceeds the maximal characteristic speed at the origin.
Then we define a spacetime contour $\Gamma$ and region $\m{I}$ by
\begin{equation}
  \label{eq:data-contour}
  \begin{aligned}
    \Gamma &\coloneqq \big\{(t, x) \in \R_- \times (-\pi/2, \pi/2) \mid t = - \mathsf{S}^{-1} \cos x\big\},\\
    \m{I} &\coloneqq \big\{(t, x) \in \R_- \times (-\pi/2, \pi/2) \mid t > - \mathsf{S}^{-1} \cos x\big\}.
  \end{aligned}
\end{equation}
For future use, we also define a constant $\mathsf{R} \asymp \mathsf{S}^{-1} \vee 1$ by
\begin{equation}
  \label{eq:fit}
  \mathsf{R} \coloneqq \sup\{R > 0 \colon (t, x) \in \m{I} \text{ if } \euc(t, x) < R\}.
\end{equation}
By the definition of $\mathsf{S}$, $\Gamma$ is uniformly spacelike with respect to $A(0)$: the characteristics $\partial_t + \lambda_{(I)}(0) \partial_x$ flow through $\Gamma$ in the same direction and are uniformly nontangential.

Given $p > 0$, we write $\Gamma_p \coloneqq \nu^p \Gamma$ and $\m{I}_p \coloneqq \nu^p \m{I}$.
We select $\beta \in (0, 1/2)$ satisfying certain conditions and solve the inner and grid equations in $\m{I}_\beta$.
The contour $\Gamma_\beta$ represents the lower part of $\partial \m{I}_\beta$, and the spacelike nature of $\Gamma$ ensures that $\m{I}_\beta$ is causally closed.
To solve our transport equations in $\m{I}_\beta$, it thus suffices to pose data on $\Gamma_\beta$.
The curve $\Gamma_\beta$ roughly denotes the transition region between the outer and inner expansions.
The grid is defined near this curve, and the joint decomposition for the inner and outer expansions in its vicinity allow us to compare the two expansions and show that they match.

We construct the zeroth row $(\psi_{0,\ell})_{\ell \geq 0}$ and column $(\psi_{k,0})_{k \geq 0}$ separately in Sections~\ref{sec:inviscid-row} and \ref{sec:zeroth-column} below.
For the remaining grid terms with $k \geq 1$, we set
\begin{equation}
  \label{eq:grid-data}
  \psi_{k,1} = \out{\psi}{k} - \psi_{k,0} \And \psi_{k, \ell} = 0 \enspace \text{on } \Gamma_\beta \text{ for } \ell \geq 2.
\end{equation}
Informally, we put all data on the first row $(\psi_{1,\ell})$ and initialize the remaining rows from zero.
The data for $\psi_{k,1}$ links the grid to the outer expansion.

To relate the inner expansion to the grid, we scale it to the inner coordinates:
\begin{equation}
  \label{eq:inner-conversion}
  \Psi_{k, \ell}(T, X) \coloneqq \nu^{k - (\ell + 1)/4} \psi_{k, \ell}\big(\nu^{1/2} T, \nu^{3/4}X\big).
\end{equation}
Given $K \in \N_0$, let
\begin{equation*}
  \Psi_{[K],\ell} \coloneqq \sum_{k \leq K} \Psi_{k, \ell}
\end{equation*}
denote the corresponding partial sum up to row $K$.
Then we construct solutions of the inner equations \eqref{eq:Omega-inner} and \eqref{eq:Sigma-inner} such that $\inn{\Psi}{\ell} \approx \Psi_{[K],\ell}$ in a precise sense where $\euc \asymp \nu^\beta$.
We defer the details to Section~\ref{sec:inner-matching}.

\section{Grid estimates}
\label{sec:grid-est}
We now use the equations and data for the grid terms $\psi_{k,\ell}$ to control their size in $\m{I}_\beta$.
We begin with bespoke constructions of the zeroth row $(\psi_{0,\ell})_{\ell \in \N_0}$ and zeroth column $(\psi_{0,k})_{k \in \N_0}$.

\subsection{The zeroth row}
\label{sec:inviscid-row}

The zeroth row $(\psi_{0,\ell})$ satisfies a hierarchy of equations decomposing the inviscid conservation law \eqref{eq:inviscid}.
To state this hierarchy, we use $\m{V}^\ell$ to denote the Cartesian $\ell$-power of $\m{V}$ and define polynomials $\poly_\ell^\omega \colon \m{V}^\ell \to \m{V}^\sperp$ and $\poly_\ell^\sigma \colon \m{V}^\ell \times \m{V}^\sperp \to \m{V}^{\Is}$ by
\vspace*{-2mm}
\begin{equation}
  \label{eq:poly}
  \poly_\ell^\omega(g_{<\ell}) \coloneqq -\sum_\bigsub{1 \leq r \leq \ell\\ \al_0 + \ldots + \al_r = \ell - r} \Aop^\sperp g_\al
  \And
  \poly_\ell^\sigma(g_{<\ell},(g_\ell)_\sperp) \coloneqq  - \sum^*_\bigsub{1 \leq r \leq \ell + 1\\ \al_0 + \ldots + \al_r = \ell - r + 1} \Aop^1g_\al.
  \vspace*{-2mm}
\end{equation}
Here $\sum\limits^*$ omits terms involving the shocking component $(g_\ell)_1$, which is not provided.

These polynomials represent certain terms in the Taylor series of $A$ projected onto the nonshocking and shocking subspaces, respectively.
Ideally, we would like $\psi_{0,\ell}$ to satisfy
\begin{equation}
  \label{eq:inviscid-row-ideal}
  \begin{aligned}
    \partial_t \omega_{0,\ell} + A_\sperp^\sperp(0) \partial_x \omega_{0,\ell} &= \partial_x[\poly_\ell^\omega(\psi_{0,<\ell})],\\
    \partial_t \sigma_{0,\ell} + \partial_\Is A_\Is^\Is(0) \partial_x(\sigma_{0,0}\sigma_{0,\ell}) &= \partial_x[\poly_\ell^\sigma(\psi_{0,<\ell}, \, \omega_{0,\ell})].
  \end{aligned}
\end{equation}
However, we do not exactly solve these equations.
Rather, we exhibit approximate solutions that solve the equations with a residue $F_\ell$:
\begin{equation}
  \label{eq:inviscid-row}
  \begin{aligned}
    \partial_t \omega_{0,\ell} + A_\sperp^\sperp(0) \partial_x \omega_{0,\ell} &= \partial_x[\poly_\ell^\omega(\psi_{0,<\ell})] + F_\ell^\omega\\
    \partial_t \sigma_{0,\ell} + \partial_\Is A_{\Is}^{\Is}(0) \partial_x(\sigma_{0,0}\sigma_{0,\ell}) &= \partial_x[\poly_\ell^\sigma(\psi_{0,<\ell},\omega_{0,\ell})] + F_\ell^\sigma.
  \end{aligned}
\end{equation}
In the following, we use the notation $[L]$ to indicate partial summation in $\ell$ up to $\ell = L$.
We choose $\psi_{0,\ell}$ so that the partial sum $\psi_{0,[L]} \coloneqq \sum_{\ell = 0}^L \psi_{0,\ell}$ closely approximates $\out{\psi}{0}$ and $\psi_{0,\ell}$ and $F_\ell$ vanish at suitable rates at the origin.
\begin{proposition}
  \label{prop:inviscid-row}
  There exist solutions $(\psi_{0,\ell})$ of \eqref{eq:inviscid-row} with residues $(F_\ell)$ such that for all $\ell,L,m,n \in \N_0$,
  \begin{equation}
    \label{eq:inviscid-row-size}
    \begin{gathered}
      \absb{\partial_t^m \partial_x^n \psi_{0,\ell}} \lesssim_{\ell,m,n} \abs{\log \dist}^{(\ell-1)_+} \dist^{\ell - 2m - 3n + 1},\\[2mm]
      \absb{\partial_t^m \partial_x^n F_\ell^\omega} \lesssim_{\ell,m,n} \abs{\log \dist}^{\ell} \dist^{\ell - 2m - 3n - 1},
      \quad \absb{\partial_t^m \partial_x^n F_\ell^\sigma} \lesssim_{\ell,m,n} \abs{\log \dist}^{\ell} \dist^{\ell - 2m - 3n},
    \end{gathered}
  \end{equation}
  and
  \begin{equation}
    \label{eq:inviscid-row-matching}
    \absb{\partial_t^m \partial_x^n (\out{\psi}{0} - \psi_{0,[L]})} \lesssim_{L,m,n} \abs{\log \dist}^{L} \dist^{L - 2m - 3n + 2}.
  \end{equation}
  Moreover, $\psi_{0,0} = (\sigma_{0,0}, 0)$ and $F_0 = 0.$
\end{proposition}
To construct $\psi_{0,\ell}$, we turn to the homogeneous components $\bar\psi_\ell$ of $\out{\psi}{0}$ from \ref{hyp:expansion}.
These are polynomials in the functions $t$, $\cub$, and $\cubder$ (from \eqref{eq:fu} and \eqref{eq:fm}) that are homogeneous in the sense of Definition~\ref{def:rhom}.
Expanding on \ref{hyp:expansion}, we obtain:
\begin{proposition}
  There exists a sequence $(\bar\psi_{\ell})$ of $(\ell + 1)$-homogeneous functions $\bar\psi_\ell$ such that for all $L,m,n \in \N_0$,
  \begin{equation}
    \label{eq:homog-matching}
    \absb{\partial_t^m \partial_x^n (\out{\psi}{0} - \bar\psi_{[L]})} \lesssim_{L,m,n} \dist^{L - 2m - 3n + 2}.
  \end{equation}
  Moreover, $\bar\psi_{0,\ell}$ is a polynomial in $(t, \cub, \cubder)$ and $\psi_0 = (\cub, 0)$.
  It solves
  \vspace*{-2mm}
  \begin{equation}
    \label{eq:homog}
    \begin{aligned}
      A_\sperp^\sperp(0) \partial_x \bar\omega_{\ell} &= \partial_x[\poly_\ell^\omega(\bar\psi_{<\ell})] - \partial_t \bar\omega_{\ell-1},\\
      \partial_t \bar\sigma_{\ell} + \partial_\Is A_{\Is}^{\Is}(0) \partial_x(\bar\sigma_{0}\bar\sigma_{\ell}) &= \partial_x[\poly_\ell^\sigma(\bar\psi_{<\ell},\bar\omega_{\ell})].
    \end{aligned}
  \end{equation}
\end{proposition}
\begin{proof}
  The first two statements restate \ref{hyp:expansion}.
  To find \eqref{eq:homog}, we substitute the homogeneous expansion into the inviscid equation \eqref{eq:inviscid} and collect terms with like homogeneity.
  These homogeneous equations can be compared with the criminal equations \eqref{eq:inviscid-row-ideal}.
\end{proof}
The homogeneous equations \eqref{eq:homog} only differ from \eqref{eq:inviscid-row-ideal} in their treatment of the nonshocking component: it is advected horizontally by $A_\sperp^\sperp(0)\partial_x$ in \eqref{eq:homog} and diagonally by $\partial_t + A_\sperp^\sperp(0) \partial_x$ in \eqref{eq:inviscid-row-ideal} and \eqref{eq:inviscid-row}.
The tilted characteristics of the latter are essential for our subsequent viscous analysis.
\begin{proof}[Proof of Proposition~\textup{\ref{prop:inviscid-row}}]
  To construct $\psi_{0,\ell}$, we first identify its leading part $\h\psi_{\ell}$, which will be a polynomial in $(t, \cub, \cubder,\log \dist)$ of degree at most $\ell - 1$ in $\log \dist$.
  We have already fixed $\psi_{0,0} = \bar\psi_{0} = (\cub, 0)$, so we set $\h\psi_0 = \bar\psi_0$.
  Because $\bar{\psi}_1$ is the leading part of $\out{\psi}{0} - \psi_{0,0}$, we must set $\h\psi_{1} = \bar\psi_{1}$.
  However, this identity need not hold for $\ell \geq 2.$

  We say a polynomial in $(t, \cub, \cubder, \log \dist)$ has \emph{polyhomogeneity} $h \in \Z$ if it has homogeneity $h$ when we drop the factors of $\log \dist$.
  Given $h \in \Z$ and $d \in \N_0$, let $\m{Q}_{h;d}$ denote the set of polynomials in $(t, \cub, \cubder, \log \dist)$ whose terms have polyhomogeneity at least $h$ and those of polyhomogeneity exactly $h$ have degree at most $d$ in $\log \dist$.
  We construct $\h\psi_\ell$ and $\psi_{0, \ell}$ inductively in the class $\m{Q}_{\ell+1;\ell-1}$.

  For the nonshocking component, we set
  \begin{equation*}
    \h\omega_\ell \coloneqq \poly_\ell^\omega(\h\psi_{<\ell}).
  \end{equation*}
  Using the inductive hypothesis and \eqref{eq:poly}, this lies in $\m{Q}_{\ell + 1; \ell - 2}$.
  For the shocking component, we wish to solve
  \begin{equation}
    \label{eq:sigma-hat-pre}
    \partial_t \h\sigma_\ell + \partial_\Is A_{\Is}^{\Is}(0) \partial_x(\sigma_{0,0} \h\sigma_{\ell}) = \partial_x[\poly_\ell^\sigma(\h\psi_{<\ell},\h\omega_{\ell})] \in \m{Q}_{\ell - 1; \ell - 2}.
  \end{equation}
  Let $\bar{\partial}_t = \partial_t + \partial_\Is A_{\Is}^{\Is}(0) \sigma_{0,0} \partial_x$ denote the $(t, \cub)$-coordinate vector field for $t$, so $\bar{\partial}_tt = 1$ and $\bar{\partial}_t \cub = 0$.
  Using Lemma~\ref{lem:fuest}, we can write the left side of \eqref{eq:sigma-hat-pre} as
  \begin{equation*}
    \bar{\partial}_t \h \sigma_\ell + \partial_1 A_1^1(0) (\partial_x \sigma_{0,0}) \h \sigma_\ell = \bar{\partial}_t \h \sigma_\ell - \cubder \h \sigma_\ell.
  \end{equation*}
  By the same lemma,
  \begin{equation}
    \label{eq:m-inverse}
    \cubder^{-1} = \abs{t} + 3 a^{-1} b \cub^2,
  \end{equation}
  so $\bar{\partial}_t\cubder^{-1} = -1$.
  Multiplying \eqref{eq:sigma-hat-pre} by the integrating factor $\cubder^{-1}$, the left side simplifies to $\bar\partial_t(\cubder^{-1}\h\sigma_\ell)$, so
  \begin{equation}
    \label{eq:sigma-hat-ev}
    \bar\partial_t(\cubder^{-1}\h\sigma_\ell) = \cubder^{-1}\partial_x[\poly_\ell^\sigma(\h\psi_{<\ell},\h\omega_{\ell})] \in \m{Q}_{\ell + 1; \ell - 2}.
  \end{equation}
  By \eqref{eq:m-inverse}, $\m{Q}_{\ell + 1; \ell - 2}$ is spanned by terms of the form $\cubder^p \cub^q \log^r \dist$ for $p \in \Z$ and $q,r \in \N_0$ satisfying $q-2p \geq \ell+1$ and $r \leq \ell - 2$ if $q-2p = \ell + 1$.

  We wish to find a  $\bar{\partial}_t$-antiderivative of $\cubder^p \cub^q \log^r \dist$.
  Factoring out constants, it suffices to antidifferentiate $(\abs{t} + c \cub^2)^{-p} \cub^q \log^r(\abs{t} + c \cub^2)$ in the class $\m{Q}_{\ell + 3; \ell - 1}$.
  When we integrate along $\bar\partial_t$, $\cub$ is constant.
  It follows that
  \begin{equation*}
    \int_{-\infty}^t \cub^q \frac{\log^r(\abs{s} + c \cub^2)}{(\abs{s} + c \cub^2)^p} \, \bar{\text{d}} s \in \m{Q}_{\ell + 3; \ell - 2} \quad \text{when } p \geq 2.
  \end{equation*}
  Similarly,
  \begin{equation*}
    -\int_t^0 \cub^q \frac{\log^r(\abs{s} + c \cub^2)}{(\abs{s} + c \cub^2)^p} \, \bar{\text{d}} s \in \m{Q}_{\ell + 3; \ell - 2} \quad \text{when } p \leq 0.
  \end{equation*}
  Finally, when $p = 1$,
  \begin{equation*}
    \frac{\cub^q}{r+1} \log^{r+1}(\abs{t} + c \cub^2) \in \m{Q}_{\ell + 3, \ell - 1}
  \end{equation*}
  is a suitable antiderivative.
  This is the only case in which the logarithmic degree increases.
  Given a $\bar{\partial}_t$-antiderivative $P \in \m{Q}_{\ell + 3; \ell - 1}$ of the right side of \eqref{eq:sigma-hat-ev}, we set $\h\sigma_\ell = \cubder P \in \m{Q}_{\ell + 1; \ell - 1}$.
  
  This completes the inductive construction of the leading parts $(\h\psi_\ell)$, which satisfy
  \begin{equation*}
    \begin{aligned}
      A_\sperp^\sperp(0) \partial_x \hat\omega_{\ell} &= \partial_x[\poly_\ell^\omega(\hat\psi_{<\ell})]\\
      \partial_t \hat\sigma_{\ell} + \partial_\Is A_{\Is}^{\Is}(0) \partial_x(\bar\sigma_{0}\hat\sigma_{\ell}) &= \partial_x[\poly_\ell^\sigma(\hat\psi_{<\ell},\hat\omega_{\ell})].
    \end{aligned}
  \end{equation*}
  This is similar to \eqref{eq:homog}, but the nonshocking equation is missing a $\partial_t$ term on the right hand side.

  Now define
  \begin{equation}
    \label{eq:correction}
    \psi_{0,\ell} \coloneqq \h \psi_\ell + (\bar \psi_{\ell + 1} - \hat\psi_{\ell + 1}) \in \m{Q}_{\ell + 1; \ell - 1}.
  \end{equation}
  Because $\bar \psi_{\ell + 1} - \hat\psi_{\ell + 1} \in \m{Q}_{\ell + 2; \ell}$, we can readily verify \eqref{eq:inviscid-row-size}.
  Moreover, because $\bar\psi_0 = \h\psi_0$, the telescoping structure in \eqref{eq:correction} ensures that
  \begin{equation*}
    \psi_{0,[L]} = \bar\psi_{[L]} - \bar\psi_0 + \h \psi_0 + \bar\psi_{L+1} - \hat\psi_{L+1} = \bar\psi_{[L]} + (\bar\psi_{L+1} - \hat\psi_{L+1}).
  \end{equation*}
  Thus \eqref{eq:inviscid-row-matching} follows from \eqref{eq:homog-matching}.
\end{proof}
Ultimately, we must account for the residues $F_\ell$.
We do so in the first row, where we include them in the forcing for $\psi_{1,\ell}$.
To show that $\psi_{1,[L]}$ approximates $\out{\psi}{1}$, we must show that the cumulative residue largely cancels.
\begin{lemma}
  \label{lem:force-cancels}
  For all $L,m,n \in \N_0$, we have
  \begin{equation*}
    \absb{\partial_t^m \partial_x^n F_{[L]}^\omega} \lesssim_L \abs{\log \dist}^{L + 1} \dist^{L - 2m - 3n} \And \absb{\partial_t^m \partial_x^n F_{[L]}^\sigma} \lesssim_L \abs{\log \dist}^{L + 1} \dist^{L - 2m - 3n + 1}.
  \end{equation*}
\end{lemma}
\begin{proof}
  The polynomials $\poly_\ell$ are chosen so that the sum of the grid equations \eqref{eq:inviscid-row-ideal} closely approximates the true inviscid conservation law.
  For example, if we sum the equations for $\omega_{0,\ell}$ up to $\ell = L$ using \eqref{eq:multilinear}, we obtain
  \begin{equation*}
    \partial_t \omega_{0,[L]} + \partial_x \sum_\bigsub{0 \leq r \leq L\\ \al_0 + \ldots + \al_r \leq L - r} \frac{\Der^r A^\sperp(0)}{(r + 1)!} \prod_{j=0}^r \psi_{0,\al_j} = F_{[L]}^\omega.
  \end{equation*}
  On the other hand, recalling the flux $f$ satisfying $\Der f = A$, we have
  \begin{equation*}
    f(\psi_{0,[L]}) - f(0) = \sum_\bigsub{0 \leq r \leq L\\ \al_j \leq L} \frac{\Der^r A (0)}{(r + 1)!} \prod_{j=0}^r \psi_{0,\al_j} + \m{O}_L\big(\abss{\psi_{0,[L]}}^{L + 2}\big).
  \end{equation*}
  Therefore Proposition~\ref{prop:inviscid-row} yields
  \begin{align*}
    \partial_t \omega_{0,[L]} + \partial_x f^\sperp(\psi_{0,[L]}) &= F_{[L]}^\omega + \partial_x \sum_\bigsub{0 \leq r \leq L\\\al_j \leq L\\\al_0 + \ldots + \al_r > L - r} \frac{\Der^r A^\sperp(0)}{(r + 1)!} \prod_{j=0}^r \psi_{0,\al_j} + \m{O}_L(\dist^{L - 1})\\
                                                                  &= F_{[L]}^\omega + \m{O}_L(\abs{\log \dist}^L \dist^{L-1}).
  \end{align*}
  Subtracting the conservation law \eqref{eq:inviscid}, we find
  \begin{equation*}
    \partial_t (\out{\omega}{0} - \omega_{0,[L]}) + \partial_x [f^\sperp(\out{\psi}{0}) - f^\sperp(\psi_{0,[L]})] = F_{[L]}^\omega + \m{O}_L\big(\abs{\log \dist}^L \dist^{L-1}\big).
  \end{equation*}
  Now using \eqref{eq:inviscid-row-matching} and rearranging, we see that $\absb{F_{[L]}^\omega} \lesssim_L \abs{\log \dist}^L \dist^{L - 1}.$

  In fact, we handled the above equations somewhat sloppily (as detailed in the next subsection) and can improve the result \emph{a posteriori}.
  Indeed, Proposition~\ref{prop:inviscid-row} states that $F_{L + 1}^\omega \lesssim_L \abs{\log \dist}^{L + 1} \dist^L$.
  Bounding $F_{[L+1]}^\omega$ as above, we therefore obtain the improved bound
  \begin{equation*}
    F_{[L]}^\omega = F_{[L+1]}^\omega - F_{L + 1}^\omega \lesssim_{L} \abs{\log \dist}^{L + 1} \dist^L.
  \end{equation*}

  The $\sigma$-equations involve an approximation of $F$ that is one-degree higher order, so similar reasoning yields $\absb{F_{[L]}^\omega} \lesssim_L \abs{\log \dist}^{L + 1} \dist^{L + 1}.$
  Finally, these estimates behave as expected under differentiation because $F_\ell$ is a polynomial in $(t, \cub, \cubder, \log \dist)$, so the lemma follows.
\end{proof}

\subsection{The zeroth column}
\label{sec:zeroth-column}
We next analyze the zeroth column $(\psi_{k,0})_{k \in \N_0}$.
Because $\inn{\Omega}{0} = 0$, we take $\omega_{k,0} = 0$.
The resulting equations for the shocking components $\sigma_{k,0}$ are identical to those encountered in the Burgers equation in~\cite{CG23}.
There, the second and third authors constructed solutions $\sigma_{k,0}$ as polynomials in $(\cubder, \cub)$ of homogeneity $-4k + 1$.
This corresponds to solving the grid equation for $\sigma_{k,0}$ from zero data at past infinity.
We work with the same solutions here; the following restates part of \cite[Proposition~3.9]{CG23}.
\begin{proposition}
  \label{prop:zeroth-col}
  For all $k,m,n \in \N_0$, $\omega_{k,0} = 0$ and $\abss{\partial_t^m \partial_x^n \sigma_{k,0}} \lesssim_{k,n,m} \dist^{-4k - 2m - 3n + 1}.$
\end{proposition}
Because the difference $\out{\psi}{k} - \psi_{k,0}$ serves as data for the next column $\psi_{k,1}$, we must control it before constructing the remaining grid terms.
The outer expansion \eqref{eq:outer-eq} is governed by the linear operator $\partial_t + \partial_x[A(\out{\psi}{0})\anon]$.
We therefore state a general estimate for this operator, which we prove in Section~\ref{sec:integrate-residue}.
\begin{proposition}
  \label{prop:outer-op}
  Fix $k \in \N$ and let $\partial_t p + \partial_x[A(\out{\psi}{0}) p] = \partial_x h + F$ in a  causally closed domain $\m{D}$ with smooth Cauchy hypersurface $\Sigma$.
  Let $a_k = \tbf{1}(k \geq 2)$ and suppose
  \begin{gather*}
    \abss{\nab^n(p |_{\Sigma})} \lesssim_{n} \euc^{-2k + 1} \dist^{-3n}, \quad \abss{\partial_t^m \partial_x^n h_\sperp} \lesssim_{m,n} \euc^{-2k - m + 1} \dist^{-3n},\\
    \abss{\partial_t^m \partial_x^n F_\sperp} \lesssim_{m,n} \euc^{-2k - m} \dist^{-3n}, \quad \abss{\partial_t^m \partial_x^n (\partial_x h_1 + F_1)} \lesssim_{m,n} \abs{\log \euc}^{a_k} \euc^{-2k - m + 1} \dist^{-3n-2}
  \end{gather*}
  for all $m,n \in \N_0$.
  Then $\abss{\partial_t^m \partial_x^n p} \lesssim_{m,n} \abs{\log \euc} \euc^{-2k - m + 1} \dist^{-3n}$ in $\m{D}$.
\end{proposition}
With this, we can control the data $\out{\psi}{k} - \psi_{k,0}$ of $\psi_{k,1}$.
\begin{proposition}
  \label{prop:out-bound}
  For all $k \in \N$ and $m,n \in \N_0$,
  \begin{equation}
    \label{eq:out-bound}
    \abss{\partial_t^m \partial_x^n (\out{\psi}{k} - \psi_{k,0})} \lesssim_{k,m,n} \abs{\log \euc} \euc^{-2k - m + 1} \dist^{-3n}.
  \end{equation}
\end{proposition}
\noindent
Together, Propositions~\ref{prop:zeroth-col} and \ref{prop:out-bound} imply Proposition~\ref{prop:outersize}.
\begin{proof}
  Let $\ti{\m{A}}$ denote a variation on $\m{A}$ from \eqref{eq:multilinear} defined with $\Der^r A(\out{\psi}{0})$ in place of $\Der^r A(0)$, and similarly for $\ti{\m{B}}$.
  Then by \eqref{eq:outer-eq}, $p_k \coloneqq \out{\psi}{k} - \psi_{k,0}$ satisfies
  \begin{equation}
    \label{eq:diff-eq-long}
    \begin{aligned}
        \partial_t(\psi_{k,0} + p_k) + \partial_x[A(\out{\psi}{0})(\psi_{k,0} + p_k)] &=
  - \partial_x \sum_\bigsub{\# \kappa \geq 1, \, \kappa_j \geq 1\\ \abs{\kappa} = k} \ti{\m{A}}(\psi_{\kappa,0} + p_\kappa)\\[-4pt]
                                                                                &\hspace{1.5cm}+ \partial_x \sum_\bigsub{\kappa_{\geq 1} \hspace{0.5pt}\geq 1\\ \abs{\kappa} = k-1} \ti{\m{B}}(\psi_{\kappa,0} + p_\kappa).
    \end{aligned}
  \end{equation}
  We expand the the multilinear terms on the right and group them into terms with no factors of $p$ and terms with at least one factor of $p$, writing:
  \begin{equation}
    \label{eq:diff-eq}
    \partial_t(\psi_{k,0} + p_k) + \partial_x[A(\out{\psi}{0})(\psi_{k,0} + p_k)] = \partial_x(h_0 + h_{\geq 1}).
  \end{equation}
  Proceeding by induction, suppose we have proved the proposition for $k' < k$.
  Because $\mr{\psi}$ is smooth and $\out{\psi}{0}$ is smooth away from the origin, finite speed of propagation implies that the desired estimate holds on any closed set disjoint from the origin.
  It thus suffices to prove the estimate in $\m{D} \coloneqq c \m{I}$ for the region $\m{I}$ defined in \eqref{eq:data-contour} and a constant $c > 0$ chosen so that $c \m{I} \subset \{\euc \leq 1/2\}$.
  By the construction of $\m{I}$, $\m{D}$ is causally closed with Cauchy hypersurface $c\Gamma$, also from \eqref{eq:data-contour}.

  First consider $h_{\geq 1}$, which only involves differences $p_{k'}$ with $k' < k$.
  If $k' = 0$, \eqref{eq:inviscid-row-matching} from Proposition~\ref{prop:inviscid-row} implies that $p_0 \lesssim \dist^2$, while $\psi_{0,0} \lesssim \dist$.
  Note that there is no log-factor in $p_0$.
  If $1 \leq k' < k$, \eqref{eq:distance-relation} and Proposition~\ref{prop:zeroth-col} yield
  \begin{equation*}
    \abss{\psi_{k',0}} \lesssim \dist^{-4k' + 1} \lesssim \euc^{-2k'} \dist \gg \abs{\log \euc} \euc^{-2k' + 1}.
  \end{equation*}
  In each case, our bound on $\psi_{k',0}$ exceeds that on $p_{k'}$, so $h_{\geq 1}$ is dominated by terms with a single factor of $p$.

  These terms of $h_{\geq 1}$ involve products of the form $p_{\kappa_i} \prod_{j \neq i} \psi_{\kappa_j,0}$.
  Using \eqref{eq:distance-relation}, the inductive hypothesis, and Proposition~\ref{prop:zeroth-col}, we find
  \begin{align*}
    p_{\kappa_i} \prod_{j \neq i} \psi_{\kappa_j,0} \lesssim \abs{\log \dist} \euc^{-2\kappa_i + 1} \dist^{\sum_{j \neq i}(-4\kappa_j + 1)} &= \abs{\log \dist} \euc^{-2\kappa_i + 1} \dist^{-4(\abs{\kappa} - \kappa_i) + \# \kappa}\\[-2pt]
                                                                                                                                            &\lesssim \abs{\log \dist} \euc^{-2 \abs{\kappa} + 1} \dist^{\# \kappa}.
  \end{align*}
  The $\ti{\m{A}}$-terms in \eqref{eq:diff-eq-long} satisfy $\abs{\kappa} = k$ and $\#\kappa \geq 1$, so we conclude that these contribute no more than $\abs{\log \euc} \euc^{-2 k + 1} \dist$ to $h_{\geq 1}$.

  For $\ti{\m{B}}$, we note that $\abs{\kappa} = k - 1$ and $\# \kappa \geq 0$, while the additional $\partial_x$ causes a loss of $\dist^{-3}$ by the inductive hypothesis.
  Hence a similar calculation yields
  \begin{equation*}
   \max\Big\{p_{\kappa_i} \partial_x \psi_{\kappa_0,0} \prod_{j \neq i,0} \psi_{\kappa_j,0}, \,\partial_x p_{\kappa_0} \prod_{j > 0} \psi_{\kappa_j,0}\Big\} \lesssim \abs{\log \euc} \euc^{-2k + 3} \dist^{-3} \lesssim \abs{\log \euc} \euc^{-2k + 1} \dist.
  \end{equation*}
  Therefore $\abs{h_{\geq 1}} \lesssim_k \abs{\log \euc} \euc^{-2 k + 1} \dist.$
  Moreover, as noted earlier, $p_0$ has no log-loss.
  It follows that when $k = 1$, we in fact have $\abs{h_{\geq 1}} \lesssim \euc^{-2 k + 1} \dist.$
  Recalling the indicator function $a_k = \tbf{1}(k \geq 2)$, we can thus write
  \begin{equation*}
    \abs{h_{\geq 1}} \lesssim \abs{\log \euc}^{a_k} \euc^{-2 k + 1} \dist \And \abs{\partial_x h_{\geq 1}} \lesssim \abs{\log \euc}^{a_k} \euc^{-2 k + 1} \dist^{-2}.
  \end{equation*}

  Next, we show that $h_0$ nearly balances the terms on the left of \eqref{eq:diff-eq} that do not involve $p_k$.
  We wish to define $h^\res$ by
  \begin{equation*}
    \partial_x h^\res = \partial_x h_0 - \partial_t \psi_{k,0} - \partial_x[A(\out{\psi}{0}) \psi_{k,0}] = h_0 - \partial_t \sigma_{k,0} e_1 - \partial_x[A_1(\out{\psi}{0})\sigma_{k,0}].
  \end{equation*}
  To select a solution of this equation, we note that \eqref{eq:sigma-grid} yields
  \begin{equation*}
    \partial_t \sigma_{k,0} = -\frac{1}{2} \partial_1 A_1^1(0) \partial_x \sum_\bigsub{0 \leq k' \leq k} \sigma_{k',0} \sigma_{k - k',0} + B_1^1(0) \partial_x^2 \sigma_{k-1,0}.
  \end{equation*}
  We therefore set
  \begin{equation*}
    h^\res \coloneqq h_0 + \frac{1}{2} \partial_1 A_1^1(0) \sum_\bigsub{0 \leq k' \leq k} \sigma_{k',0} \sigma_{k - k',0} e_1 - B_1^1(0) \partial_x \sigma_{k-1,0} e_1 - A_1(\out{\psi}{0}) \sigma_{k,0}.
  \end{equation*}
  Projecting on the nonshocking directions, we find $h_\sperp^\res = (h_0)_\sperp - A_1^\sperp(\out{\psi}{0}) \sigma_{k,0}.$
  Note that \eqref{eq:homog-matching} yields $\out{\psi}{0} \lesssim \dist$.
  Because $A_1^\sperp(0) = 0$, the mean value theorem implies that $A_1^\sperp(\out{\psi}{0}) \lesssim \dist$, so $A_1^\sperp(\out{\psi}{0}) \sigma_{k,0} \lesssim \dist^{-4k + 2}$.
  Next, \eqref{eq:diff-eq-long} shows that the $\ti{\m{A}}$-terms in $h_0$ are controlled by
  \begin{equation*}
    \prod_j \psi_{\kappa_j,0} \lesssim \dist^{-4\abs{\kappa} + \# \kappa + 1} \lesssim \dist^{-4k + 2}.
  \end{equation*}
  One can readily check the same bound for the $\ti{\m{B}}$-terms, so in total \eqref{eq:distance-relation} yields
  \begin{equation*}
    \abss{h_\sperp^\res} \lesssim \dist^{-4k + 2} \lesssim \euc^{-2k + 1}.
  \end{equation*}
  Now consider the shocking component.
  We view $\ti{\m{A}}$-terms with $\#\kappa \geq 2$ and $\ti{\m{B}}$-terms with $\#\kappa \geq 1$ as error.
  Using Proposition~\ref{prop:zeroth-col}, these are bounded by $\dist^{-4k + 3}$.
  For the low-order terms, we exploit cancellation with $\partial_t \sigma_{k,0}$, writing
  \begin{align*}
    h_1^\res = - \frac{1}{2} \sum_{0 \leq k' \leq k} [\partial_1 A_1^1(\out{\psi}{0}) - \partial_1 A_1^1(0)] \sigma_{k',0} \sigma_{k - k',0} + [B_1^1(\out{\psi}{0}) - B_1^1&(0)] \partial_x \sigma_{k-1,0}\\[-2pt]
    &+ \m{O}(\dist^{-4k + 3}).
  \end{align*}
  Because $\out{\psi}{0} \lesssim \dist$, the differences involving $A$ and $B$ are order $\dist$, so Proposition~\ref{prop:zeroth-col} yields $h_1^\res \lesssim \dist^{-4k + 3}$.
  Taking a derivative, we find
  \begin{equation*}
    \abss{\partial_x h_1^\res} \lesssim \dist^{-4k} = \dist^{-4k + 2} \dist^{-2} \lesssim \euc^{-2k + 1} \dist^{-2}.
  \end{equation*}
  Now let $h \coloneqq h^\res + h_{\geq 1}$, so $\partial_t p_k + \partial_x[A(\out{\psi}{0})p_k] = \partial_x h$.
  Gathering the above results, we have
  \begin{equation*}
    \abs{h_\sperp} \lesssim \euc^{-2k +1} \And \abs{\partial_x h_1} \lesssim \abs{\log \euc}^{a_k} \euc^{-2k + 1} \dist^{-2}.
  \end{equation*}
  We therefore satisfy the hypotheses of Proposition~\ref{prop:outer-op} with $F \equiv 0$ and $m = n = 0$.
  Throughout our estimates, $\partial_t$ and $\partial_x$ cause $\dist^{-2} \lesssim \euc^{-1}$ and $\dist^{-3}$ losses, respectively, so the hypotheses in fact hold for all $m,n \in \N_0$.
  Thus \eqref{eq:out-bound} follows from Proposition~\ref{prop:outer-op}.
\end{proof}

\subsection{Renormalization}

To treat the remaining grid terms, we require a few more technical tools.
We will analyze hyperbolic equations of the form
\begin{equation}
  \label{eq:model}
  \partial_t \phi + M \partial_x \phi = \partial_x F,
\end{equation}
where $M$ is a constant full-rank matrix and $\phi$ is a vector-valued unknown.
For example, from \eqref{eq:omega-grid}, we might have $\phi = \omega_{k,\ell}$, $M = A_\sperp^\sperp(0)$, and $F$ a spatial derivative of a function of $\psi_{\leq k, < \ell}$.

Near the origin, dynamics are largely driven by the shocking component $\sigma$.
Its characteristic vector field is parallel to $\partial_t$ but transverse to $\partial_x$ at the origin, so $\partial_t\sigma$ is better behaved than $\partial_x \sigma$.
For this region, it is often advantageous to rearrange or ``renormalize'' \eqref{eq:model} so that it is driven by time derivatives rather than space derivatives.
We formalize this procedure in the following lemma.
\begin{lemma}
  \label{lem:renorm}
  Let $n \in \N_0$, $M$ be a constant full-rank matrix, and $F$ and $G$ be sufficiently smooth.
  If $\phi$ satisfies $\rd_t \phi + M \rd_x \phi = \rd_x^{n + 1} F + G,$ then
  \begin{equation}
    \label{eq:renorm}
    (\rd_t + M \rd_x) (\phi - \bR_n [M, F]) = (-1)^{n+1} M^{-(n+1)} \rd_t^{n+1} F + G,
  \end{equation}
  where $\bR_n[M, \anon] \colon \m{C}^k \to \m{C}^{k - n}$ is the linear map
  \begin{equation*}
    \bR_n [M, F] \coloneqq \sum_{j = 0}^n (-1)^j M^{-(j+1)} \rd_t^j \rd_x^{n - j} F.
  \end{equation*}
\end{lemma}
\begin{proof}
  Proceeding by induction, we first consider $n = 0$.
  Direct computation yields
  \begin{equation*}
    (\rd_t + M \rd_x) (\phi - M^{-1}  F) = -M^{-1} \rd_t F + G.
  \end{equation*}
  Next, assuming we have shown \eqref{eq:renorm} for some $n \in \N_0$, we show it for $n + 1$.
  Using $\rd_x^{n+2}F = \rd_x^{n+1} (\rd_x F)$, the inductive hypothesis yields
  \begin{equation*}
    (\rd_t + M \rd_x) (\phi - \bR_n [M, \rd_x F]) = \rd_x[(-1)^{n+1} M^{-(n+1)} \rd_t^{n+1} F] + G.
  \end{equation*}
  Applying the $n = 0$ case, we find
  \begin{equation*}
    (\rd_t + M \rd_x) (\phi - \bR_n [M, \rd_x F] - (-1)^{n+1} M^{-(n + 2)} \rd_t^{n+1} F) = (-1)^n M^{-(n + 2)} \rd_t^{n + 2} F + G.
  \end{equation*}
  Noting that $\bR_{n + 1} [M, F] = \bR_n [M, \rd_x F] + (-1)^{n+1} M^{-(n + 2)} \rd_t^{n+1} F$, the result follows from induction.
\end{proof}

\subsection{Integral bounds}

After renormalization, we project \eqref{eq:model} onto an eigenvector of $A$ and integrate along its characteristic.
Here, we record bounds on the resulting integrals.
In the following, let $\m{D} \subset \{\euc \leq 1/2\}$ be a causally closed domain with Cauchy hypersurface $\Sigma$.
Let $\ti\gamma_I$ denote the characteristic of $\partial_t + \lambda_{(I)}(\out{\psi}{0})\partial_x$ beginning at $\Sigma$ and ending at some $(t, x) \in \m{D}$.
Let $\gamma_{I'}$ and $\gamma_1$ denote ``frozen'' characteristics corresponding to $\partial_t + \lambda_{(I')}(0) \partial_x$ and $\partial_t + \partial_1 A_1^1(0) \sigma_{0,0} \partial_x$, respectively, with the same initial and terminal behavior.
All are affinely parameterized, meaning the parameter increases at rate one along the defining vector field.
\begin{lemma}
  \label{lem:integrals}
  Fix $2 \leq I' \leq N$ and $(t, x) \in \m{D}$.
  For all $r > 1$, and $a,q \geq 0$,
  \begin{align}
    \int \euc^{-r} \circ \ti\gamma_{I'} &\lesssim_r \euc^{-r + 1}(t, x),\label{eq:nonshock}\\
    \int \dist^{-2}\euc^{-r} \circ \ti\gamma_{I'} &\lesssim_r \euc^{-r+1/3}(t,x),\label{eq:nonshock-e-and-d}\\
    \int \abs{\log \euc}^a \euc^{-r} \dist^{-q}  \circ \ti\gamma_1 &\lesssim_{r,a,q} \abs{\log \euc(t,x)}^a \euc^{-r + 1}(t,x) \dist^{-q}(t, x),\label{eq:shock}\\
    \int \euc^{-1} \dist^{-q} \circ \ti\gamma_1 &\lesssim_q \abs{\log \euc(t,x)} \dist^{-q}(t,x).\label{eq:shock-log}
  \end{align}
  Moreover, the same estimates hold with $\gamma$ in place of $\ti\gamma$.
\end{lemma}
To treat the true shocking characteristic $\ti \gamma_1$, it is helpful to introduce a corresponding Lagrangian coordinate $u$, termed the eikonal function.
It satisfies
\begin{equation}
  \label{eq:eikonal}
  \partial_t u + \lambda_{(1)}(\out{\psi}{0}) \partial_x u = 0, \quad u(t_0, x) = x - c,
\end{equation}
where $c \in \R$ is chosen so that $u(0,0) = 0$.
We expect $u$ to resemble a multiple of $\cub$ near the origin.
We only require a weaker result, that the ``cubic distance'' associated to $u$ is similar to $\dist$.
\begin{lemma}
  \label{lem:eikonal-distance}
  In $\m{D}$, $\dist^2 \asymp \abs{t} + u^2$.
\end{lemma}
\begin{proof}
  We first control $\dist^{2}\partial_x u$.
  By \ref{hyp:expansion}, \eqref{eq:fm}, and \eqref{eq:coeff},
  \begin{equation}
    \label{eq:lambda-expansion}
    \lambda_{(1)}(\out{\psi}{0}) = -a\cub + \m{O}(\dist^2) \And \partial_x[\lambda_{(1)}(\out{\psi}{0})] = -\cubder + \m{O}(\dist^{-1}).
  \end{equation}
  Also, \eqref{eq:cubic-Burgers} yields $(\partial_x - a \cub\partial_x)\cub = 0$.
  Writing $L \coloneqq \partial_t + \lambda_{(1)}(\out{\psi}{0}) \partial_x$, \eqref{eq:fd} thus implies
  \begin{equation*}
    L(\dist^2) = [\partial_t - a \cub \partial_x + \m{O}(\dist^2) \partial_x](\dist^2) = -1 + \m{O}(\dist).
  \end{equation*}

  Using these observations, we commute \eqref{eq:eikonal} with $\dist^2 \partial_x$ to find
  \begin{equation*}
    L(\dist^2 \partial_x u) = \m{O}(\dist^{-1}) \dist^2 \partial_x u.
  \end{equation*}
  Now $\dist^{-1}$ is integrable along the characteristic $\ti{\gamma}_1$, which resembles $\partial_t$ near the origin.
  It follows from Gr\"onwall that $\dist^2 \partial_x u$ is continuous.
  Since it is initially positive, it must remain so.
  Hence
  \begin{equation}
    \label{eq:eikonal-deriv}
    \partial_x u \sim \const \dist^{-2} = \const \partial_x \cub
  \end{equation}
  near the origin, for some constant $\const > 0$.

  Next, let $\gamma(t)$ denote the $x$-coordinate of the unique root of $u$ at time $t$, so $u\big(t,\gamma(t)\big) = 0$.
  By \eqref{eq:eikonal}, this satisfies $\dot\gamma = \lambda_{(1)}\circ\out{\psi}{0}\big(t, \gamma(t)\big)$ with $\gamma(0) = 0$.
  Now \eqref{eq:fu} implies $\abs{a\cub(t, x)} \leq \abs{x}\abs{t}^{-1}$.
  Thus \eqref{eq:lambda-expansion} yields
  \begin{equation*}
    \abs{\dot\gamma} \leq \abs{t}^{-1}\abs{\gamma} + F
  \end{equation*}
  with $F \lesssim \dist^2$.
  By Lemma~\ref{lem:fuest}, there exists $C > 0$ such that $F \leq C \abs{t}$ where $\abs{\gamma} \leq \abs{t}^{3/2}$.
  For the moment, cut off $F$ by $0$ where $\abs{t} > \abs{t}^{3/2}$ and let $\h\gamma$ denote the corresponding solution.
  Then Gr\"onwall yields $\h\gamma \lesssim \abs{t}^2$.
  Because $\abs{t} \leq \abs{t}^{3/2}$ for $\abs{t}$ small, $\h\gamma$ does not experience the $F$-cutoff.
  It follows that $\h\gamma = \gamma$ near the origin and $\gamma \lesssim \abs{t}^2$.

  We now fix $t < 0$ and integrate \eqref{eq:eikonal-deriv} in $x$ from $\gamma(t)$ to find
  \begin{equation*}
    u \lesssim \cub + \cub(t,\gamma) \lesssim \cub + \abs{t}.
  \end{equation*}
  The second bound follows from Lemma~\ref{lem:fuest}.
  Then
  \begin{equation*}
    \abs{t} + u^2 \lesssim \abs{t} + (\cub + \abs{t})^2 \lesssim \dist^2.
  \end{equation*}

  To show the other direction, we invert \eqref{eq:eikonal-deriv} to find
  \begin{equation*}
    \partial_u x \lesssim \dist^2 \lesssim \max\{\abs{t}, \abs{x}^{2/3}\}.
  \end{equation*}
  Holding $t$ constant and integrating in $u$ from $u = 0$, this yields
  \begin{equation*}
    x - \gamma(t) \lesssim \abs{t u} + \abs{u}^3.
  \end{equation*}
  By Young's inequality,
  \begin{equation*}
    x \lesssim \abs{tu} + \abs{u}^3 + \abs{t}^2 \lesssim \abs{t}^{3/2} + \abs{u}^3.
  \end{equation*}
  Therefore Lemma~\ref{lem:fuest} yields
  \begin{equation*}
    \dist^2 \lesssim \abs{t} + \abs{x}^{2/3} \lesssim \abs{t} + u^2.
    \qedhere
  \end{equation*}
\end{proof}
With this bound in place, we can control the integrals \eqref{eq:nonshock}--\eqref{eq:shock-log}.
\begin{proof}[Proof of Lemma~\ref{lem:integrals}]
  We only treat the unfrozen characteristics $\ti\gamma$, as the proofs in the frozen case $\gamma$ are similar but simpler.
  Let $s \in [s_0, s_1]$ denote the affine parameter, which is only determined up to a shift, and let the upright $\big(\sft(s), \sfx(s)\big)$ denote the coordinates of $\ti\gamma$.

  We first consider the nonshocking bounds.
  Set $s_1 = -\euc(t, x)$.
  Modulo a constant factor, the time component of $\partial_t + \lambda_{(I')} \partial_x$ is no smaller than its space component.
  One can thus readily check that $\abs{s} \asymp \euc \circ \ti\gamma_{I'}(s)$ for all $s \in [s_0, s_1]$.
  This immediately implies \eqref{eq:nonshock}, as
  \begin{equation*}
    \int \euc^{-r} \circ \ti\gamma_{I'} \asymp \int_{s_0}^{s_1} \abs{s}^{-r} \d s \lesssim \abs{s_1}^{-r + 1} \asymp \euc^{-r+1}(t, x).
  \end{equation*}
  
  Next, it suffices to focus on a neighborhood of the origin on which $\lambda_{(I')}(\out{\psi}{0})$ is nonzero.
  By symmetry, we may assume $\lambda_{(I')}(\out{\psi}{0}) \geq c > 0$ for some $c > 0$ that may change from line to line.
  We then claim it suffices to treat $x \geq 0$.
  Indeed, suppose we parametrize by time instead, which changes our integrals by no more than a constant factor.
  Given $x > 0$, let $\ti{\gamma}_{I'}^{\pm}$ denote the characteristics ending at $(t, \pm x)$.
  Then $\sft^- = \sft^+$, so by ODE uniqueness, $\sfx^-$ can never coincide with $\sfx^+$.
  It follows that $\abs{\sfx^-} \geq \abs{\sfx^+}$, so the integrand of \eqref{eq:nonshock-e-and-d} is larger for $\ti{\gamma}_{I'}^+$ than for $\ti{\gamma}_{I'}^-$.
  We are thus free to assume that $x \geq 0$ in $(t, x)$.

  To show \eqref{eq:nonshock-e-and-d}, we return to an affine parametrization ending at $s_1 = -\euc(t, x)$, so $\euc \asymp \abs{s}$ in the integral.
  We thus need to control $\dist$.
  It suffices to work within a small ball $B$ about the origin, as contributions from outside $B$ are order $1$.
  Assume $(t, x) \in B$ with $x \geq 0$.
  Within $B$, $\ti{\gamma}_{I'}$ is roughly a straight line with positive slope, while the curves $\sfx = \pm \abs{\sft}^{3/2}$ are nearly vertical.
  The characteristic thus crosses these at two locations, with parameters $s_- < s_+$.
  We can easily check that
  \begin{equation}
    \label{eq:inside-horn}
    \abs{\sft}\big|_{[s_-,s_+]} \asymp \euc(t, x) \And s_+ - s_- \asymp \abs{\sft}^{3/2} \asymp \euc(t, x)^{3/2}.
  \end{equation}
  
  If $x < \abs{t}^{3/2}$, then $s_+ > s_1$, and we instead set $s_+ = s_1$.
  The characteristic leaves $B$ at some $s_B < s_-$.
  Let $I_1$, $I_2$, and $I_3$ denote the contributions to the integral in \eqref{eq:nonshock-e-and-d} from $s$ in $[s_B, s_-],[s_-,s_+],$ and $[s_+,s_1]$, respectively.
  In $I_1$, $\dist \asymp \abs{\sfx}^{1/3}$ and $\abs{\sfx} \asymp s_+ - s$.
  Using Lemma~\ref{lem:fuest}, we find
  \begin{equation*}
    I_1 = \int_{s_B}^{s_-} \dist^{-2} \euc^{-r} \circ \ti{\gamma}_{I'} \lesssim \int_{-\infty}^{s_-} \abs{s}^{-r} (s_+ - s)^{-2/3} \d s \asymp \abs{s_-}^{-r + 1/3} \asymp \euc(t, x)^{-r + 1/3}.
  \end{equation*}
  For $I_2$, $\dist \asymp \abs{\sft}^{1/2}$, so \eqref{eq:inside-horn} yields
  \begin{equation*}
    I_2 = \int_{s_-}^{s_+} \dist^{-2} \euc^{-r} \circ \ti{\gamma}_{I'} \asymp \euc(t, x)^{-r-1}(s_+ - s_-) \asymp \euc(t, x)^{-r+1/2}.
  \end{equation*}
  Finally, in $I_3$, $\dist \asymp \abs{\sfx}^{1/3}$ and $\abs{\sfx} \asymp s - s_-$, so mimicking $I_1$,
  \begin{equation*}
    I_3 = \int_{s_+}^{s_1} \dist^{-2} \euc^{-r} \circ \ti{\gamma}_{I'} \lesssim \int_{s_+}^{s_1} \abs{s}^{-r} (s - s_-)^{-2/3} \d s \lesssim \euc(t, x)^{-r + 1/3}.
  \end{equation*}
  In combination, we obtain \eqref{eq:nonshock-e-and-d}.

  Turning to the shocking characteristic $\ti{\gamma}_1$, we recall the eikonal function $u$ defined in \eqref{eq:eikonal} and switch to coordinates $(\sft, u)$.
  Then $\ti{\gamma}_1$ is constant in $u$, so throughout we let $u = u(t, x)$ be the value of the eikonal function at the endpoint.
  Because $\ti{\gamma}_1$ is constant in $u$, we can use $s = \sft$ as our parameter.

  For simplicity, we focus on \eqref{eq:shock} with $a = 0$.
  When $a > 0$, the result will follow from analogous computations.
  Combining Lemmas~\ref{lem:fuest} and \ref{lem:eikonal-distance}, we can check that
  \begin{equation*}
    \euc \asymp \abs{t} + \dist^3 \asymp \abs{t} + \abs{u}^3.
  \end{equation*}
  It thus suffices to control
    \begin{equation*}
    \int \euc^{-r} \dist^{-q} \circ \ti{\gamma_1} \asymp \int_{t_0}^t \frac{\dn s}{(\abs{s}^r + \abs{u}^{3r})(\abs{s}^{q/2} + \abs{u}^q)},
  \end{equation*}
  where $t_0$ is the earliest time in our integral.
  We generally assume $r > 1$, which corresponds to \eqref{eq:shock}; we make brief asides to treat the $r = 1$ case \eqref{eq:shock-log}.
  \smallskip
  
  \noindent
  \textit{Case 1}: Suppose $\abs{t} < \abs{u}^3$, in which case $\dist(t, x) \asymp \abs{u}$ and $\euc(t, x) \asymp \abs{u}^3$.
  We break the integral into $I_1,I_2$, and $I_3$ over the intervals $[-\abs{u}^3,t]$, $[-\abs{u}^2, -\abs{u}^3],$ $[t_0, -\abs{u}^2]$, respectively.
  We may assume that these three intervals are nontrivial: otherwise we set their contribution to zero, and the calculation simplifies.

  In $I_1$, we have $\dist(s, u) \asymp \abs{u}$ and $\euc(s, u) \asymp \abs{u}^3$, so
  \begin{equation*}
    I_1 \asymp \abs{u}^{-3r - q} \int_{-\abs{u}^3}^{t} \d s \lesssim \abs{u}^{-3(r-1) - q} \asymp \euc(t, x)^{-r+1} \dist(t, x)^{-q}.
  \end{equation*}
  In $I_2$, $\dist \asymp \abs{u}$ but $\euc \asymp \abs{s}$.
  Hence
  \begin{equation*}
    I_2 \asymp \abs{u}^{-q} \int_{-\abs{u}^2}^{-\abs{u}^3} \abs{s}^{-r} \d s \lesssim u^{-3(r-1) - q} \asymp \euc(t, x)^{-r+1} \dist(t, x)^{-q}.
  \end{equation*}
  Additionally, if $r = 1$, we pick up an additional factor of $\abs{\log \abs{u}} \asymp \abs{\log \euc(t, x)}$, as in \eqref{eq:shock-log}.
  Finally, in $I_3$, $\dist \asymp \abs{s}^{1/2}$ and $\euc \asymp \abs{s}$, so
  \begin{equation*}
    I_3 \asymp \int_{t_0}^{-\abs{u}^2} \abs{s}^{-r - \frac{q}{2}} \d s \lesssim u^{-2(r - 1) - q} \lesssim \euc(t, x)^{-\frac{2(r-1)}{3}} \dist(t, x)^{-q}.
  \end{equation*}
  Similarly, we generate a log if $r = 1$ and $q =  0$.
  Collecting these bounds, we find \eqref{eq:shock} and \eqref{eq:shock-log}.
  \smallskip
  
  \noindent
  \textit{Case 2}: Next suppose $\abs{u}^3 \leq \abs{t} < \abs{u}^2$, so $\dist(t, x) \asymp \abs{u}$ but $\euc(t, x) \asymp \abs{t}$.
  We break the integral into $I_2$ and $I_3$ over $[-\abs{u}^2, t]$ and $[t_0, -\abs{u}^2]$.
  By the bounds used in Case 1,
  \begin{equation*}
    I_2 \asymp \abs{u}^{-q} \int_{-\abs{u}^2}^t \abs{s}^{-r} \d s \lesssim u^{-q} \abs{t}^{-r + 1} \asymp \euc(t, x)^{-r+1} \dist(t, x)^{-q}.
  \end{equation*}
  Again, we pick up $\abs{\log \abs{t}} \asymp \abs{\log \euc(t, x)}$ if $r = 1$.
  Next,
  \begin{equation*}
    I_3 \asymp \int_{t_0}^{-\abs{u}^2} \abs{s}^{-r - \frac{q}{2}} \d s \lesssim u^{-2(r - 1) - q} \lesssim \euc(t, x)^{-r+1} \dist(t, x)^{-q}.
  \end{equation*}
  And again, $r = 1$ and $q = 0$ produces a log.
  Together, these yield \eqref{eq:shock} and \eqref{eq:shock-log}.
  \smallskip

  \noindent
  \textit{Case 3}: Finally, suppose $\abs{t} \geq \abs{u}^2$, so $\dist(t, x) \asymp \abs{t}^{1/2}$ and $\euc(t, x) \asymp \abs{t}$.
  Let $I_3$ denote the entire integral.
  Using the bounds above,
  \begin{equation*}
    I_3 \lesssim \int_{t_0}^t \abs{s}^{-r - \frac{q}{2}} \d s \lesssim \abs{t}^{-r - \frac{q}{2} + 1} \asymp \euc(t, x)^{-r+1} \dist(t, x)^{-q}.
  \end{equation*}
  If $r = 1$ and $q = 0$, we accrue $\abs{\log \euc}$.
  This completes the proof of \eqref{eq:shock} and \eqref{eq:shock-log}, and the lemma.
\end{proof}

\subsection{Bulk grid estimates}
We now control the ``bulk'' grid terms $\psi_{k,\ell} = (\sigma_{k,\ell}, \omega_{k,\ell})^\top$ with $k,\ell \geq 1$.
Recall that $\omega_{k,\ell}$ and $\sigma_{k,\ell}$ satisfy the linear advection equations \eqref{eq:omega-grid} and \eqref{eq:sigma-grid} in $\m{I}_\beta$ with data \eqref{eq:grid-data} on $\Gamma_\beta$.
\begin{proposition}
  \label{prop:grid-est}
  Fix $\eps \in \big(0, \tfrac{1}{6}\big]$ and $\beta \in \big(\tfrac{1}{2 + \eps},\tfrac{1}{2}\big)$.
  Then for all $k, \ell \in \N$, $m, n \in \N_0$,
  \begin{equation}
    \label{eq:grid-est}
    \absb{\partial_t^m\partial_x^n \psi_{k,\ell}} \lesssim_{\eps,\beta,k,\ell,m,n} \nu^{(\ell - 1 - 3\eps)\beta/3} \euc^{-2k - m + 1}\dist^{-3n} \quad \text{in } \m{I}_\beta \setminus \m{I}_{1/2}.
  \end{equation}
\end{proposition}
\noindent
We often take $\eps = 1/6$ and suppress the dependence of the implied constant on $\beta \in (6/13, 1/2)$.

This estimate differs from its predecessors, Propositions~\ref{prop:inviscid-row} and \ref{prop:zeroth-col}, in important respects.
First, as $\ell$ increases, we do not show a higher order of vanishing at the origin \emph{per se}.
Rather, we show that we gain a positive power of $\nu$.
In other words, we do not describe the germ of $\out{u}{k}$ at the origin, but content ourselves with uniform improvement in the region $\m{I}_\beta \setminus \{\euc \leq \nu^{1/2}\}$.
The latter is simpler to show and suffices for our purposes.
Note that we cut out the region $\m{I}_{1/2}$ where the viscosity is nonperturbative.
The approximations underlying the grid break down in $\m{I}_{1/2}$, and we make no use of bulk grid terms there.

Second, the estimate \eqref{eq:grid-est} involves the Euclidean distance $\euc$ in addition to the cubic distance $\dist$.
This reflects the transverse non-shocking characteristics, which are not adapted to $\dist$.
Integrating in non-shocking directions naturally leads to expressions controlled by $\euc$ (\emph{Cf.} Lemma~\ref{lem:integrals}), and much of the coming analysis involves the interplay between $\dist$ and $\euc$ along various characteristics.
\begin{proof}
  Given $\eps \in \big(0, \frac{1}{6}\big]$ and $\beta \in \big(\tfrac{1}{2 + \eps},\tfrac{1}{2}\big)$, let $\eps'$ be the mean of $\eps$ and $\beta^{-1} - 2 < \eps$, so that $\eps' < \eps$ but $\beta > \frac{1}{2 + \eps'}$.
  Let $b \coloneqq \frac{2}{3} + \eps'$, so $b \in \big(\frac{2}{3}, \frac{5}{6}\big)$.
  We show
  \begin{equation}
    \label{eq:grid-tight}
    \absb{\partial_t^m\partial_x^n \psi_{k,\ell}} \lesssim_{b,\beta,k,\ell,m,n}  \nu^{(\ell + 1)\beta/3 - b\beta - o_\ell(1)} \euc^{-2k - m + 1}\dist^{-3n} \quad \text{in } \m{I}_\beta \setminus \m{I}_{1/2}\,.
  \end{equation}
  The proposition follows after relaxing $b$ to $\frac{2}{3} + \eps$.
  
  We induct on $(k,\ell)$ according to the lexicographic ordering on $\N \times \N$, so $(k', \ell') < (k, \ell)$ if $k' < k$ or $k' = k$ and $\ell' < \ell$.
  Fix $(k, \ell) \in \N \times \N$ and suppose we have shown \eqref{eq:grid-tight} for all $(k',\ell') \in \N \times \N$ such that $(k',\ell') < (k,\ell)$ and all $m, n \in \N_0$.
  We show \eqref{eq:grid-tight} for $(k, \ell)$ with $m = n = 0$.
  Derivative estimates with $m > 0$ or $n > 0$ follow from nearly identical calculations, so we only sketch their proofs.
  Recall from Propositions~\ref{prop:omega-grid} and \ref{prop:sigma-grid} that $\omega_{k,\ell}$ depends only on $\psi_{<(k,\ell)}$ while $\sigma_{k,\ell}$ depends also on $\omega_{k,\ell}$.
  We therefore first analyze $\omega$ and then $\sigma$.
  \smallskip

  \noindent\emph{Step 1: Nonshocking estimates.}
  Recalling \eqref{eq:omega-grid}, define $h = h_A + h_B$ for
  \begin{equation}
    \label{eq:omega-driver}
    h_A \coloneqq -\sum_\bigsub{1 \leq r \leq \ell\\ \kappa_0 + \ldots + \kappa_r = k\\ \al_0 + \ldots + \al_r = \ell - r} \Aop_\sperp\psi_{\kappa,\al}
    \And
    h_B \coloneqq \sum_\bigsub{0 \leq r \leq \ell-1\\ \kappa_0 + \ldots + \kappa_r = k-1\\ \al_0 + \ldots + \al_r = \ell - r - 1} \Bop_\sperp\psi_{\kappa,\al}.
  \end{equation}
  Using Lemma~\ref{lem:renorm}, let $R \coloneqq \textbf{R}_0[A_\sperp^\sperp(0), h]$ and $F \coloneqq -A_\sperp^\sperp(0)^{-1} \partial_th - \nu^{-1} \delta_{k,1} F_\ell^\omega,$ so 
  \begin{equation}
    \label{eq:omega-simple}
    [\partial_t \omega_{k,\ell} + A_\sperp^\sperp(0) \partial_x](\omega_{k,\ell} - R) = F.
  \end{equation}
  We first control the renormalization term $R$, which is simply a linear transformation of $h$.
  We want to show that $\abs{h} \lesssim_{k,\ell} \nu^{(\ell + 1)\beta/3 - b\beta - o_\ell(1)} \euc^{-2k + 1}.$
  We begin with $h_A$.

  Our approach depends on the number of terms in $\psi_{\kappa_j,\alpha_i}$ in the zeroth column or zeroth row, as we use different estimates in these cases.
  Let $\s{C}$ denote the indices $j$ such that $\al_j = 0$ (the zeroth column), $\s{R}$ those such that $\al_j \geq 1$ and $\kappa_j = 0$ (the inviscid row), and $\s{B}$ the remainder (the bulk).
  Using Propositions~\ref{prop:inviscid-row} and \ref{prop:zeroth-col} as well as the inductive hypothesis, we can estimate the product in $\Aop \psi_{\kappa,\al}$ by
  \begin{equation}
    \label{eq:prod-A-star}
    \begin{aligned}
      \prod_{j=0}^r &\psi_{\kappa_j,\al_j} =  \prod_{\s{C}} \psi_{\kappa_j,0}\cdot \prod_{\s{R}} \psi_{0,\al_j} \cdot \prod_{\s{B}} \psi_{\kappa_j,\al_j}\\
                                                           &\lesssim_{k,\ell} \dist^{-4 \sum_{\s{C}} \kappa_j} \dist^{\abs{\s{C}}} (\abs{\log \dist}\dist)^{\sum_{\s{R}}(\al_j + 1)} \nu^{\frac{\beta}{3}\sum_{\s{B}} (\al_j + 1)} \nu^{-b\beta \abs{\s{B}}} \euc^{-2 \sum_{\s{B}} \kappa_j} \euc^{\abs{\s{B}}}.
    \end{aligned}
  \end{equation}
  On $\m{I}_\beta \setminus \m{I}_{1/2}$, we have $\abs{\log \dist} \asymp \abs{\log \nu} = \nu^{-o(1)}$.
  We can thus absorb logarithmic terms in the factor $\nu^{-o_\ell(1)}$ in \eqref{eq:grid-tight}, and we ignore them in the remainder of the proof.

  To simplify \eqref{eq:prod-A-star}, we extract $2k-1$ factors of $\euc^{-1}$ and show that the remaining terms involve only nonnegative powers of $\dist$ or $\euc$, which we bound by suitable powers of $\nu$ on $\m{I}_\beta$.
  We sometimes convert factors of $\dist^{-2}$ to $\euc^{-1}$, so the precise calculation depends on the relative powers of $\dist$ and $\euc$ in \eqref{eq:prod-A-star}.
  We therefore modulate our argument depending on the balance between $\s{C} \cup \s{R}$ to $\s{B}$.
  This theme runs throughout the proof.
  
  The terms in $h_A$ satisfy $\abs{\s{C} \cup \s{R} \cup \s{B}} = r + 1$, $\sum_j (\al_j + 1) = \ell + 1$, and $\sum_j \kappa_j = k$.
  If $\s{B} = \emptyset$, \eqref{eq:prod-A-star} simplifies.
  Noting that $\al_j = 0$ for $j \in \s{C}$ and $\kappa_j = 0$ for $j \in \s{R}$, \eqref{eq:distance-relation} and $\dist \lesssim \nu^{\beta/3}$ yield
  \begin{equation} \label{eq:no-bulk}
    \begin{aligned}
    \prod_{j=0}^r \psi_{\kappa_j,\al_j} &\lesssim_{k,\ell} \nu^{-o_\ell(1)} \dist^{-4k} \dist^{\ell + 1} = \nu^{-o_\ell(1)} \dist^{\ell - 1} \dist^{-4k+2}
    \\ &\lesssim_{k,\ell} \nu^{(\ell - 1)\beta/3 - o_\ell(1)} \euc^{-2k + 1}.
    \end{aligned}
  \end{equation}
  This is sufficiently small because $b > 2/3$.
  Now suppose $\abs{\s{B}} \geq 1$.
  Using $\euc \lesssim \nu^\beta$ and $\beta < 1$, we find
  \begin{align*}
    \prod_{j=0}^r \psi_{\kappa_j,\al_j} &\lesssim_{k,\ell} \nu^{-o_\ell(1)} \euc^{-2k + 1} \dist^{\sum_{\s{C}\cup\s{R}} (\al_j + 1)} \nu^{\frac{\beta}{3} \sum_j (\al_j + 1)} \nu^{-b\beta\abs{\s{B}}} \euc^{\abs{\s{B}}-1}\\
                                        &\lesssim_{k,\ell} \nu^{-o_\ell(1)} \euc^{-2k + 1} \nu^{(\ell + 1)\beta/3} \nu^{(1 - b)\beta\abs{\s{B}}-\beta} \lesssim \nu^{(\ell + 1)\beta/3 - b\beta -o_\ell(1)} \euc^{-2k + 1}.
  \end{align*}
  
  We can treat $h_B$ in \eqref{eq:omega-driver} similarly, but we get an additional factor of $\dist^{-3}$ from $\partial_x$, while $\sum_j (\al_j + 1) = \ell$ and $\sum_j \kappa_j = k - 1$.
  When $\s{B} = \emptyset$, these changes cancel and we obtain an identical bound to \eqref{eq:no-bulk}.
  If $\abs{\s{B}} \geq 1$, we find
  \begin{align}
    (\partial_x \psi_{\kappa_0,\al_0}) \prod_{j=1}^r\psi_{\kappa_j, \al_j} &\lesssim_{k, \ell} \dist^{-3} \nu^{-o_\ell(1)} \euc^{-2k + 3} \dist^{\sum_{\s{C}\cup\s{R}} (\al_j + 1)} \nu^{\frac{\beta}{3} \sum_j (\al_j + 1)} \nu^{-b\beta\abs{\s{B}}} \euc^{\abs{\s{B}}-1}\nonumber\\
                                                                           &\lesssim_{k, \ell} \dist^{-3} \euc^2 \nu^{\beta\ell/3} \nu^{-b\beta - o_\ell(1)} \euc^{-2k + 1}.\label{eq:prod-B-star}
  \end{align}
  Now \eqref{eq:distance-relation} implies $\dist^{-3}\euc^2 \lesssim \euc^{1/2} \lesssim \nu^{\beta/2} < \nu^{\beta/3}$, so the same bound holds.
  That is:
  \begin{equation}
    \label{eq:omega-renorm-fine}
    \abs{R} \lesssim \abs{h} \lesssim_{k,\ell} \nu^{(\ell + 1)\beta/3 - b\beta -o_\ell(1)} \euc^{-2k + 1}.
  \end{equation}

  Now consider $F \coloneqq -A_\sperp^\sperp(0)^{-1} \partial_th - \nu^{-1} \delta_{k,1} F_\ell^\omega$ in \eqref{eq:omega-simple}.
  Throughout Propositions~\ref{prop:inviscid-row} and \ref{prop:zeroth-col} and the inductive hypothesis, $\partial_t$ contributes $\euc^{-1}$, so
  \begin{equation*}
    \abs{\partial_t h} \lesssim_{k, \ell} \nu^{(\ell + 1)\beta/3 - b\beta -o_\ell(1)} \euc^{-2k}.
  \end{equation*}
  Meanwhile, when $k = 1$, Proposition~\ref{prop:inviscid-row} states that the residue $\nu^{-1} F_\ell^\omega$ contributes $\nu^{-1-o(1)} \dist^{\ell - 1}$.
  Using $1 \leq (2 + b - 2/3)\beta$, we find
  \begin{equation*}
    \nu^{-1-o(1)} \dist^{\ell - 1} \lesssim_\ell \nu^{-(2 + b - 2/3)\beta+(\ell-1)\beta/3-o(1)} \lesssim \nu^{(\ell + 1)\beta/3 - b\beta - o(1)} \euc^{-2}.
  \end{equation*}
  So this contribution is no worse, and we obtain
  \begin{equation}
    \label{eq:omega-force}
    \abs{F} \lesssim_{k, \ell} \nu^{(\ell + 1)\beta/3 - b\beta - o_\ell(1)} \euc^{-2k}.
  \end{equation}
  The above control of the inviscid residue $F_\ell$ is the only step requiring $\beta \geq \frac{1}{2 + \eps}$.
  It therefore determines the lower bound on $\beta$ throughout the paper.
  
  Fix $2 \leq I' \leq N$.
  Using the transport equation \eqref{eq:omega-simple}, we can write
  \begin{equation}
    \label{eq:omega-integral}
    \omega_{k,\ell}^{I'} = \omega_{k,\ell}^{I'}|_{\Gamma_\beta} + R^{I'} - R^{I'}|_{\Gamma_\beta} + \int_{\Gamma_\beta}^\cdot F^{I'} \circ \gamma_{I'},
  \end{equation}
  where $\gamma_{I'}$ denotes the characteristic of $\partial_t + \lambda_{(I')}(0) \partial_x$ and the integral limits indicate we integrate from $\Gamma_\beta$ to the implied argument $(t, x)$ of $\omega_{k,\ell}$.
  Using \eqref{eq:grid-data}, Proposition~\ref{prop:out-bound}, and $b > 2/3$,
  \begin{equation*}
    \abs{\omega_{k,\ell}|_{\Gamma_\beta}} \lesssim_k \nu^{-(2k-1)\beta} \lesssim \nu^{(1 + 1)\beta/3 - b\beta - o(1)} \euc^{-2k + 1} \quad \text{if } \ell = 1
  \end{equation*}
  and $\omega_{k,\ell}|_{\Gamma_\beta} = 0$ otherwise.
  Next, \eqref{eq:omega-renorm-fine} yields
  \begin{equation*}
    \abs{R - R|_{\Gamma_\beta}} \lesssim_{k,\ell} \nu^{(\ell + 1)\beta/3 - b\beta - o_\ell(1)} \euc^{-2k+1}.
  \end{equation*}
  Finally, we apply \eqref{eq:nonshock} from Lemma~\ref{lem:integrals} to \eqref{eq:omega-force} to find
  \begin{equation*}
    \Big|\int_{\Gamma_\beta}^\cdot F^{I'} \circ \gamma_{I'}\Big| \lesssim_{k, \ell} \nu^{(\ell + 1)\beta/3 - b\beta - o_\ell(1)} \int_{\Gamma_\beta}^\cdot \euc^{-2k} \circ \gamma_{I'} \lesssim_{k, \ell} \nu^{(\ell + 1)\beta/3 - b\beta - o_\ell(1)} \euc^{-2k + 1}.
  \end{equation*}
  Now \eqref{eq:omega-integral} yields $\abss{\omega_{k,\ell}} \lesssim_{k, \ell} \nu^{(\ell + 1)\beta/3 - b\beta - o_\ell(1)} \euc^{-2k + 1}.$

  We briefly discuss the corresponding derivative bounds.
  Fix $m,n \in \N_0$ and apply $\partial_t^m\partial_n^x$ to \eqref{eq:omega-grid}.
  Using Lemma~\ref{lem:renorm}, we can write
  \begin{equation*}
    [\partial_t + A_\sperp^\sperp(0)](\partial_t^m\partial_x^n \omega_{k,\ell} - R_{m,n}) = F_{m,n}
  \end{equation*}
  for
  \begin{align*}
    R_{m,n} &\coloneqq \textbf{R}_n[A_\sperp^\sperp(0), \partial_t^m h] - \textbf{R}_{n-1}[A_\sperp^\sperp(0), \nu^{-1} \delta_{k,1} \partial_t^m F_\ell^\omega],\\
    F_{m,n} &\coloneqq (-1)^{n+1} A_\sperp^\sperp(0)^{-(n+1)} \partial_t^{m + n + 1} h + (-1)^{n+1} \nu^{-1} \delta_{k,1} A_\sperp^\sperp(0)^{-n} \partial_t^{m + n} F_\ell^\omega.
  \end{align*}
  By Propositions~\ref{prop:inviscid-row} and \ref{prop:zeroth-col} and the inductive hypothesis, $\partial_t$ and $\partial_x$ contribute factors of $\euc^{-1}$ and $\dist^{-3}$, respectively.
  Repeating our analysis with these additional factors, we find
  \begin{align*}
    \abss{R_{m,n}} &\lesssim_{k,\ell,m,n} \nu^{(\ell + 1)\beta/3 - b\beta - o_\ell(1)} \euc^{-2k - m + 1} \dist^{-3n},\\
    \abss{F_{m,n}} &\lesssim_{k,\ell,m,n} \nu^{(\ell + 1)\beta/3 - b\beta - o_\ell(1)} \euc^{-2k - m - n}.
  \end{align*}

  To complete the picture, we must control the derivative data, namely $\partial_t^m \partial_x^n\omega_{k,\ell}|_{\Gamma_\beta}$.
  We illustrate the procedure for $\Der \omega_{k,\ell}$; higher derivatives follow from induction.
  The equation \eqref{eq:omega-simple} provides a linear combination of $\partial_x(\omega_{k,\ell} - R)$ and $\partial_t(\omega_{k,\ell} - R)$.
  Moreover, we can differentiate $\omega_{k,\ell} - R$ along $\Gamma_\beta$ to obtain a second, linearly independent combination of the same quantities.
  By construction, $\Gamma_\beta$ is uniformly timelike with respect to the characteristics of $A$.
  This ensures that we can invert the above linear relations to uniformly express $\Der (\omega_{k,\ell} - R)$ in terms of $F$ and the derivative of $\omega_{k,\ell} - R$ along $\Gamma_\beta$.
  Using \eqref{eq:grid-data}, \eqref{eq:out-bound}, and \eqref{eq:omega-force}, we obtain
  \begin{equation}
    \label{eq:deriv-data}
    \Der (\omega_{k,\ell} - R) |_{\Gamma_\beta} \lesssim_{k,\ell} \nu^{(\ell + 1)\beta/3 - b\beta - o_\ell(1)} \euc^{-2k}.
  \end{equation}
  Differentiating $R$, we obtain the desired bounds on $\partial_t \omega_{k,\ell}|_{\Gamma_\beta}$ and $\partial_x \omega_{k,\ell}|_{\Gamma_\beta}$.
  Integrating as in \eqref{eq:omega-integral}, we extend these bounds to the interior $\m{I}_\beta$, and to higher derivatives by induction.
  This completes the analysis of the nonshocking component (Step 1).
  \smallskip

  \noindent\emph{Step 2: Estimates on the shocking component.}
  We recycle notation from Step 1.
  Recalling \eqref{eq:sigma-grid}, define $F \coloneqq F_\omega + F_A + F_B - \nu^{-1} \delta_{k,1} F_\ell^\sigma$ for $F_\omega \coloneqq -\partial_\Is A_\Is^\sperp(0) \partial_X(\sigma_{0,0} \omega_{k,\ell})$,
  \begin{equation*}
    F_A \coloneqq - \partial_x\sum_\bigsub{1 \leq r \leq \ell + 1\\ \kappa_0 + \ldots + \kappa_r = k\\ \al_0 + \ldots + \al_r = \ell - r + 1\\ (\kappa_j,\al_j) \neq (k,\ell)} \Aop_1\psi_{\kappa,\al}
    \And
    F_B \coloneqq \partial_x \sum_\bigsub{0 \leq r \leq \ell\\ \kappa_0 + \ldots + \kappa_r = k-1\\ \al_0 + \ldots + \al_r = \ell - r} \Bop_1\psi_{\kappa,\al}.
  \end{equation*}
  Then we can write \eqref{eq:sigma-grid} as
  \begin{equation}
    \label{eq:sigma-simple}
    \partial_t \sigma_{k,\ell} + \partial_1 A_1^1(0)\partial_x(\sigma_{0,0} \sigma_{k,\ell}) = F.
  \end{equation}
  Recall that each $\partial_x$ contributes a factor of $\dist^{-3}$.
  Using Step 1, the first term $F_\omega$ in $h$ is simple to bound: $F_\omega \lesssim_{k,\ell} \dist^{-2} \nu^{(\ell - 2)\beta/3 - b \beta - o_\ell(1)} \euc^{-2k + 1}.$
  To treat $F_A$, we follow Step 1 and write
  \begin{equation}
    \label{eq:prod-bd}
    \begin{aligned}
    \hspace{-7pt}\prod_{j=0}^r \psi_{\al_j,\kappa_j} \!\lesssim_{\ell,k} \nu^{-o_\ell(1)} \dist^{-4 \sum_{\s{C}}\kappa_j} \dist^{\sum_{\s{C}\cup\s{R}} (\al_j + 1)} \nu^{\frac{\beta}{3}\sum_{\s{B}} (\al_j + 1)} \nu^{-b\beta\abs{\s{B}}} \euc^{-2\sum_{\s{B}} \kappa_j} \euc^{\abs{\s{B}}}.
    \end{aligned}
  \end{equation}
  In $h_A$, $\abs{\s{C} \cup \s{R} \cup \s{B}} = r + 1 \geq 2$, $\sum_j \kappa_j = k$, and $\sum_j (\al_j + 1) = \ell + 2$.
  We analyze \eqref{eq:prod-bd} by cases depending on $\abs{\s{B}}$.
  If $\s{B} = \emptyset$, we follow \eqref{eq:no-bulk} and find
  \begin{equation*}
  \begin{aligned}
    \prod_{j=0}^r \psi_{\al_j,\kappa_j} &\lesssim_{\ell,k} \nu^{-o_\ell(1)} \dist^{-4k} \dist^{\ell + 2} = \dist \nu^{-o_\ell(1)} \dist^{\ell - 1} \dist^{-4k + 2}
    \\ &\lesssim_{k,\ell} \dist \nu^{(\ell + 1)\beta/3 - b \beta - o_\ell(1)} \euc^{-2k + 1}.
    \end{aligned}
  \end{equation*}
  If $\abs{\s{B}} = 1$, then $\abs{\s{C} \cup \s{R}} = r + 1 - \abs{\s{B}} \geq 1$, so $\sum_{\s{C} \cup \s{R}} (\al_j + 1) \geq 1$ and
  \begin{equation*}
  \begin{aligned}
    \prod_{j=0}^r \psi_{\al_j,\kappa_j} &\lesssim_{\ell,k} \dist \nu^{-o_\ell(1)} \dist^{\sum_{\s{C} \cup \s{R}}(\al_j + 1) - 1} \nu^{\frac{\beta}{3} \sum_{\s{B}}(\al_j + 1)} \nu^{-b\beta} \euc^{-2k + 1}
    \\ &\lesssim_{k,\ell} \dist \nu^{(\ell + 1)\beta/3 - b\beta - o_\ell(1)} \euc^{-2k + 1}.
    \end{aligned}
  \end{equation*}
  Finally, if $\abs{\s{B}} \geq 2$, we have
  \begin{equation*}
  \begin{aligned}
    \prod_{j=0}^r \psi_{\al_j,\kappa_j}\! &\lesssim_{\ell,k} \euc^{1/2} \nu^{(\ell + 2)\beta/3 - b\beta\abs{\s{B}} - o_\ell(1)} \euc^{\abs{\s{B}} - 3/2}\! \euc^{-2k + 1}
    \\ &\lesssim_{k,\ell} \dist \nu^{(\ell + 1)\beta/3 - o_\ell(1)} \nu^{[1/3 + (1-b)\abs{\s{B}}-3/2]\beta} \euc^{-2k + 1}.
    \end{aligned}
  \end{equation*}
  Now $b < 5/6$ and $\abs{\s{B}} \geq 2$, so $(1-b)\abs{\s{B}} \geq 2(1-b)$ and
  \begin{equation*}
    b + \frac{1}{3} + 2(1 - b) - \frac{3}{2} = 5/6 - b > 0.
  \end{equation*}
  It follows that
  \begin{equation}
    \label{eq:prod-A-1}
    \prod_{j=0}^r \psi_{\al_j,\kappa_j} \lesssim_{k,\ell} \dist \nu^{(\ell + 1)\beta/3 - b\beta - o_\ell(1)} \euc^{-2k + 1}.
  \end{equation}
  Collecting these bounds and applying $\partial_x$, we have shown that
  \begin{equation*}
    F_A \lesssim_{k,\ell} \dist^{-2} \nu^{(\ell + 1)\beta/3 - b\beta - o_\ell(1)} \euc^{-2k + 1}.
  \end{equation*}
  We note that $\abs{\s{B}} \geq 2$ is the only case that demands $b \leq 5/6$, and thus $\eps \leq 1/6$.
  We cannot be too lax with $\eps$ lest repeated products of \eqref{eq:grid-est} become unmanageable.

  Now consider $F_B$, for which $r \geq 0$, $\sum_j \kappa_j = k - 1$, $\sum_j (\al_j + 1) = \ell + 1$, and we pick up an additional factor of $\dist^{-3}$ from $\partial_x$.
  We have
  \begin{equation*}
  \begin{aligned}
    (\partial_x \psi_{\kappa_0,\al_0}) \prod_{j=1}^r &\psi_{\kappa_j, \al_j}\\[-8pt]
                                                     &\lesssim_{k, \ell} \nu^{-o(1)} \dist^{-3} \dist^{-4 \sum_{\s{C}}\kappa_j} \dist^{\abs{\s{C} \cup \s{R}}} \dist^{\sum_{\s{R}} \al_j} \nu^{\frac{\beta}{4}\sum_{\s{B}} \al_j} \nu^{-\beta\abs{\s{B}}/2} \euc^{-2\sum_{\s{B}} \kappa_j} \euc^{\abs{\s{B}}}.
    \end{aligned}
  \end{equation*}
  When $\abs{B} = \emptyset$, this simplifies to powers of $\dist$ alone (and $\nu^{-o(1)}$).
  The result is identical to the corresponding case in $F_A$.
  When $\abs{\s{B}} \geq 1$, we write
  \begin{align}
    (\partial_x \psi_{\kappa_0,\al_0}) \prod_{j=1}^r &\psi_{\kappa_j, \al_j}\nonumber\\[-8pt]
                                                     &\lesssim_{k, \ell} \dist (\dist^{-4} \euc^2) \nu^{-o_\ell(1)} \dist^{\sum_{\s{C}\cup\s{R}} (\al_j + 1)} \nu^{\frac{\beta}{3}\sum_{\s{B}} (\al_j + 1)} \nu^{-b\beta\abs{\s{B}}} \euc^{\abs{\s{B}}-1} \euc^{-2k + 1}\nonumber\\
                                                     &\lesssim_{k,\ell} \dist \nu^{(\ell + 1)\beta/3 - b\beta - o_\ell(1)} \nu^{(1 - b)(\abs{\s{B}} - 1)\beta} \euc^{-2k + 1}\nonumber\\
                                                     &\lesssim_{k,\ell} \dist \nu^{(\ell + 1)\beta/3 - b\beta - o_\ell(1)} \euc^{-2k + 1}.\label{eq:prod-B-1}
  \end{align}
  Applying $\partial_x$, we see that $F_\omega + F_A + F_B \lesssim_{k, \ell} \dist^{-2} \nu^{(\ell + 1)\beta/3 - b\beta - o_\ell(1)} \euc^{-2k + 1}.$
  
  In Step 1, Proposition~\ref{prop:inviscid-row} provided $\nu^{-1}\delta_{k,1}F_\ell^\omega \lesssim \nu^{(\ell + 1)\beta/3 - b\beta - o_\ell(1)} \euc^{-2k}.$
  Because $F_\ell^\sigma$ is smaller than $F_\ell^\omega$ by a factor of $\dist$, \eqref{eq:distance-relation} yields
  \begin{equation*}
    \nu^{-1}\delta_{k,1}F_\ell^\sigma \lesssim_\ell \dist \nu^{(\ell + 1)\beta/3 - b\beta - o_\ell(1)} \euc^{-2k} \lesssim_\ell \dist^{-2} \nu^{(\ell + 1)\beta/3 - b\beta - o_\ell(1)} \euc^{-2k + 1}.
  \end{equation*}
  Hence \eqref{eq:sigma-simple} yields
  \begin{equation*}
    \partial_t \sigma_{k,\ell} + \partial_1 A_1^1(0)\partial_x(\sigma_{0,0} \sigma_{k,\ell}) = F \lesssim_{k,\ell} \dist^{-2} \nu^{(\ell + 1)\beta/3 - b\beta - o_\ell(1)} \euc^{-2k + 1}.
  \end{equation*}
  We now follow the derivation of \eqref{eq:sigma-hat-ev} to write this as
  \begin{equation}
    \label{eq:scaled-force}
    \bar\partial_t(\dist^2 \sigma_{k,\ell}) = \dist^2 F \lesssim_{k,\ell} \nu^{(\ell + 1)\beta/3 - b\beta - o_\ell(1)} \euc^{-2k + 1},
  \end{equation}
  where $\bar\partial_t = \partial_t + \partial_1 A_1^1(0)\sigma_{0,0}\partial_x$ and $\dist^2 = \cubder^{-1}$.
  Integrating along the corresponding shocking characteristic $\gamma_1$, we find
  \begin{equation*}
    \sigma_{k,\ell} = \sigma_{k, \ell}|_{\Gamma_\beta} + \dist^{-2} \int_{\Gamma_\beta}^\cdot (\dist^2 F) \circ \gamma_1.
  \end{equation*}
  Using \eqref{eq:grid-data}, Lemma~\ref{lem:integrals}, and \eqref{eq:scaled-force}, we obtain $\abss{\sigma_{k,\ell}} \lesssim_{k, \ell} \nu^{(\ell + 1)\beta/3 - b\beta - o_\ell(1)} \euc^{-2k + 1}$, as desired.
  We note from \eqref{eq:shock-log} that we pick up a factor of $\abs{\log \euc} \lesssim \abs{\log \nu}$ when $k = 1$.
  We absorb this in $\nu^{-o_\ell(1)}$.

  To control derivatives of $\sigma_{k,\ell}$, we commute \eqref{eq:sigma-simple} with $\partial_t^m\partial_x^n$.
  Inducting on $(m,n)$ and recalling \eqref{eq:fm}, we can check that
  \begin{equation*}
    \bar\partial_t(\partial_t^m\partial_x^n \sigma_{k,\ell}) - (n + 1) \dist^{-2} \partial_t^m\partial_x^n \sigma_{k,\ell} = F_{n,m} \lesssim_{k,\ell,n,m} \nu^{(\ell + 1)\beta/3 - b\beta - o_\ell(1)} \euc^{-2k - m} \dist^{-3n}.
  \end{equation*}
  Multiplying by the integrating factor $\dist^{2(n+1)}$, this yields
  \begin{equation*}
    \bar\partial_t(\dist^{2(n+1)} \partial_t^m\partial_x^n \sigma_{k,\ell}) \lesssim_{k,\ell,n,m} \nu^{(\ell + 1)\beta/3 - b\beta - o_\ell(1)} \euc^{-2k - m} \dist^{2-n}.
  \end{equation*}
  We control the data $\sigma_{k,\ell}|_{\Gamma_\beta}$ as described for the nonshocking components above.
  By Lemma~\ref{lem:integrals}, $\abss{\partial_t^m\partial_x^n \sigma_{k,\ell}} \lesssim_{k,\ell,m,n} \nu^{(\ell + 1)\beta/3 - b\beta - o_\ell(1)} \euc^{-2k - m + 1} \dist^{-3n}.$
\end{proof}
For future convenience, we record bounds on $\Aop$ and $\Bop$ evaluated on the grid.
\begin{lemma}
  \label{lem:prod}
  Fix $\eps \in (0,1/6]$, $\beta \in \big(\tfrac{1}{2 + \eps}, \tfrac{1}{2}\big)$, $k,\ell \in \N$, and $m,n \in \N_0$.
  \begin{enumerate}[label = \textnormal{(\roman*)}]
  \item
    If $\abs{\al} + \#\al = \ell$ and $\abs{\kappa} = k$,
    \begin{equation*}
      \abss{\partial_t^m \partial_x^n \Aop_\sperp\psi_{\kappa,\al}} \lesssim_{\eps,\beta,k,\ell,m,n} \nu^{(\ell - 1 - 3\eps)\beta/3} \euc^{-2k - m + 1} \dist^{-3n.}
    \end{equation*}

  \item
    If $\abs{\al} + \#\al = \ell - 1$ and $\abs{\kappa} = k - 1$,
    \begin{equation*}
      \abss{\partial_t^m \partial_x^n \Bop_\sperp\psi_{\kappa,\al}} \lesssim_{\eps,\beta,k,\ell,m,n} \nu^{(\ell - 1 - 3\eps)\beta/3} \euc^{-2k - m + 1} \dist^{-3n}.
    \end{equation*}

  \item
    If $\abs{\al} + \#\al = \ell + 1$ and $\abs{\kappa} = k$,
    \begin{equation*}
      \abss{\partial_t^m \partial_x^n \Aop_1\psi_{\kappa,\al}} \lesssim_{\eps,\beta,k,\ell,m,n} \dist \nu^{(\ell - 1 - 3\eps)\beta/3} \euc^{-2k - m + 1} \dist^{-3n}.
    \end{equation*}
    
  \item
    If $\abs{\al} + \#\al = \ell$ and $\abs{\kappa} = k - 1$,
    \begin{equation*}
      \abss{\partial_t^m \partial_x^n \Bop_1\psi_{\kappa,\al}} \lesssim_{\eps,\beta,k,\ell,m,n} \dist \nu^{(\ell - 1 - 3\eps)\beta/3} \euc^{-2k - m + 1} \dist^{-3n}.
    \end{equation*}
  \end{enumerate}
\end{lemma}
\begin{proof}
  These essentially follow from \eqref{eq:prod-A-star}, \eqref{eq:prod-B-star}, \eqref{eq:prod-A-1}, and \eqref{eq:prod-B-1}.
  There are two small cases to fill in.
  First, we assumed in \eqref{eq:prod-A-star} that $\#\al \geq 1$.
  When $\# \al = 0$, the bound follows directly from Proposition~\ref{prop:grid-est}.
  Similarly, when estimating $\Aop_1$ in \eqref{eq:prod-A-1}, we omitted terms involving $\sigma_{k,\ell}$.
  Now with Proposition~\ref{prop:grid-est} we can include these terms without changing the bound.
\end{proof}

\section{Outer-grid matching}
\label{sec:outer-matching}
We designed the grid $(\psi_{k,\ell})$ so that summation in $\ell$ yields $\out{\psi}{k}$ and summation in $k$ yields $\inn{\psi}{\ell}$.
In this section we prove the first fact: summation in $\ell$ closely matches the outer expansion in the matching zone $\m{I}_\beta$.
Given $L \in \N$, we study the difference $p_{k,L}$ between $\out{\psi}{k}$ and the partial sum
\begin{equation}
  \label{eq:grid-outer-partial}
  \psi_{k,[L]} \coloneqq \sum_{\ell = 0}^L \psi_{k,\ell}.
\end{equation}
In the inviscid $k=0$ case, we already bounded $p_{0,L}$ in \eqref{eq:inviscid-row-matching} of Proposition~\ref{prop:inviscid-row}.
Thus we need only consider $k \geq 1$.
If $\out{\psi}{k} \sim \sum_{\ell \geq 0} \psi_{k,\ell}$, then we should expect $p_{k,L}$ to be commensurate with the largest term omitted from \eqref{eq:grid-outer-partial}, namely $\psi_{k,L + 1}$.
\begin{proposition}
  \label{prop:outer-matching}
  Fix $\eps \in \big(0, \tfrac{1}{6}\big]$ and $\beta \in \big(\tfrac{1}{2 + \eps},\tfrac{1}{2}\big)$.
  For all $k, L \in \N$ and $m,n \in \N_0$, the difference $p_{k,L} \coloneqq \out{\psi}{k} - \psi_{k,[L]}$ satisfies
  \begin{equation}
    \label{eq:outer-matching}
    \absb{\partial_t^m \partial_x^n p_{k,L}} \lesssim_{k,L,m,n} \nu^{(L - 3\eps)\beta/3} \euc^{-2k - m + 1} \dist^{-3n} \quad \text{in } \m{I}_\beta \setminus \m{I}_{1/2}.
  \end{equation}
\end{proposition}
Our approach is similar to the proof of Proposition~\ref{prop:out-bound}.
Recalling the outer equation \eqref{eq:outer-eq}, we write
\begin{equation}
  \label{eq:outer-diff}
  \begin{aligned}
    \partial_t(\psi_{k,[L]} + &p_{k,L}) + \partial_x[A(\out{\psi}{0})(\psi_{k,[L]} + p_{k,L})]
                                       \\ &= - \partial_x \sum_\bigsub{r \geq 1, \, \kappa_j \geq 1\\ \kappa_0 + \ldots + \kappa_r = k} \frac{\Der^rA(\out{\psi}{0})}{(r + 1)!} \prod_{j=0}^r (\psi_{\kappa_j,[L]} + p_{\kappa_j,L})\\
                                     &\hspace{12pt}+ \partial_x \sum_\bigsub{r \geq 0, \, \kappa_{\geq 1} \hspace{0.5pt}\geq 1\\ \kappa_0 + \ldots + \kappa_r = k-1} \frac{\Der^rB(\out{\psi}{0})}{r!} \partial_x(\psi_{\kappa_0,[L]} + p_{\kappa_0,L}) \prod_{j=1}^r (\psi_{\kappa_j,[L]} + p_{\kappa_j,L})
                                     \\ &\eqqcolon \partial_x(h_0 + h_{\geq 1}).
  \end{aligned}
\end{equation}
To define $h_0$ and $h_{\geq 1}$, we expand products on the right side and obtain terms mixing factors of $\psi_{<k,[L]}$ and $p_{<k,L}$.
Let $h_0$ denote the collection of terms with no factors of $p$ and $h_{\geq 1}$ the remaining terms with at least one $p$-factor.

Each difference $p$ is relatively small, and we can readily bound $h_{\geq 1}$ through the inductive hypothesis.
In contrast, $h_0$ has no reason to be small.
However, it is nearly balanced by the terms involving $\psi_{k,[L]}$ on the left.
Equivalently, $\psi_{k,[L]}$ approximately solves \eqref{eq:outer-eq}.

Together, these observations imply that $\partial_tp_{k,L} + \partial_x[A(\out{\psi}{0})p_{k,L}] \approx 0$.
We chose grid data in \eqref{eq:grid-data} so that $p_{k,[L]} = 0$ on $\Gamma_\beta$.
Integrating, we conclude that $p_{k,L}$ is small.
Throughout, we focus on the case $m = n = 0$ in Proposition~\ref{prop:outer-matching}.
Derivative estimates follow from very similar calculations.

\subsection{Difference estimates}
We first control the terms $h_{\geq 1}$ in the right side of \eqref{eq:outer-diff} containing at least one factor of $p$.
We begin with a simple bound on the partial sums $\psi_{k,[L]}$.
\begin{lemma}
  \label{lem:partial-bound}
  For all $L \in \N$ and $k, m,n \in \N_0$,
  \begin{equation}
    \label{eq:partial-bound}
    \absb{\partial_t^m \partial_x^n \psi_{k,[L]}} \lesssim_{k,L,m,n} \euc^{-2k - m} \dist^{-3n + 1}.
  \end{equation}
\end{lemma}
\begin{proof}
  We show that the bound \eqref{eq:partial-bound} holds for each $\psi_{k,\ell}$.
  Using \eqref{eq:distance-relation} and Propositions~\ref{prop:inviscid-row} and \ref{prop:zeroth-col}, it is immediate when $k$ or $\ell$ vanishes.
  For the remaining cases, we observe that \eqref{eq:grid-outer-partial} and $\euc \lesssim \nu^\beta$ imply that $\nu^{-\beta/6} \euc \lesssim \euc^{5/6} \lesssim \euc^{1/2} \lesssim \dist$.
  Hence when $k,\ell \geq 1$, Proposition~\ref{prop:grid-est} yields
  \begin{equation*}
    \begin{aligned}
      \absb{\partial_t^m \partial_x^n \psi_{k,\ell}} &\lesssim_{k,\ell,m,n} \nu^{(\ell - 3/2)\beta3} \euc^{-2k - m + 1} \dist^{-3n} \lesssim_{k,\ell,m,n} \nu^{-\beta/6} \euc  \euc^{-2k - m} \dist^{-3n}
      \\ &\lesssim_{k,\ell,m,n} \euc^{-2k - m} \dist^{-3n + 1}.
    \end{aligned}
    \qedhere
  \end{equation*}
\end{proof}
We can now bound the terms on the right side of \eqref{eq:outer-diff} containing factors of $p$.
\begin{lemma}
  \label{lem:diff-prod}
  Fix $k, L \in \N$.
  If Proposition~\ref{prop:outer-matching} holds for all $k' < k$, then for all $m,n \in \N_0$,
  \begin{equation*}
    \absb{\partial_t^m \partial_x^n h_{\geq 1}} \lesssim_{k,L,m,n} \nu^{(L - 3\eps)\beta/3} \euc^{-2k - m + 1} \dist^{-3n + 1}.
  \end{equation*}
\end{lemma}
\begin{proof}
  Comparing \eqref{eq:outer-matching} and \eqref{eq:partial-bound}, we see that the inductive bound on $p_{k',L}$ is always tighter than that on $\psi_{k',[L]}$ for $k' < k$.
  So terms in $h_{\geq 1}$ with multiple factors of $p$ are easier to control than those with exactly one.
  We therefore only explicitly analyze terms with exactly one factor of $p_{<k,L}$.

  First consider the contribution to $h_{\geq 1}$ from $A$.
  Given $0 \leq i \leq r$, the inductive hypothesis \eqref{eq:outer-matching} and \eqref{eq:partial-bound} from Lemma~\ref{lem:partial-bound} yield
  \begin{equation*}
    p_{\kappa_i,L} \prod_\bigsub{0 \leq j \leq r\\ j \neq i} \psi_{\kappa_j,[L]} \lesssim_{k,L} \nu^{(L - 3\eps)\beta/3} \euc^{-2\kappa_i + 1} \prod_{j \neq i} \euc^{-2\kappa_j}\dist \lesssim_{k,L} \nu^{(L - 3\eps)\beta/3} \euc^{-2k + 1} \dist.
  \end{equation*}
  Here we use $\kappa_i \geq 1$ to apply \eqref{eq:outer-matching}.
  In the contribution to $h_{\geq 1}$ from $B$, we sometimes have $\kappa_0 = 0$.
  If $\kappa_0 \geq 1$, similar manipulations yield
  \begin{equation*}
  \begin{aligned}
    (\partial_x p_{\kappa_0,L}) \prod_j \psi_{\kappa_j,[L]}, \, (\partial_x \psi_{\kappa_0,[L]}) p_{k_i,L} \prod_{j \neq i} \psi_{\kappa_j,[L]} &\lesssim_{k,L} \dist^{-3} \euc^2 \nu^{(L - 3\eps)\beta/3} \euc^{-2k + 1}
    \\ &\lesssim_{k,L} \nu^{(L - 3\eps)\beta/3} \euc^{-2k + 1}\dist.
    \end{aligned}
  \end{equation*}
  Here we used \eqref{eq:distance-relation} to conclude that $\dist^{-3} \euc^2 \lesssim \dist$.
  If $\kappa_0 = 0$, we use \eqref{eq:inviscid-row-matching} to similarly write
  \begin{equation*}
  \begin{aligned}
    (\partial_x p_{0,L}) \prod_j \psi_{\kappa_j,[L]} &\lesssim_{k,L} \nu^{-o_L(1)} \dist^{L-1} \prod_j \euc^{-2\kappa_j} \lesssim_{k,L} \nu^{-o_L(1)} \dist^{L-1} \euc^{-2k + 2}
    \\ &\lesssim_{k,L} \nu^{-o_L(1)} \dist^L \euc^{-2k + 1} \dist.
    \end{aligned}
  \end{equation*}
  Now $\nu^{-o_L(1)} \dist^L \lesssim_L \nu^{L\beta/3 - o_L(1)} \lesssim_L \nu^{(L - 3\eps)\beta/3}$, so this term is bounded like the rest.

  We have now treated all terms in $h_{\geq 1}$ with exactly one factor of $p$, and those with multiple factors satisfy identical bounds.
  As usual, $\partial_t^m\partial_x^n$ adds a factor of $\euc^{-m}\dist^{-3n}$, so the lemma follows.
\end{proof}

\subsection{Partial sum equation}
We next show that $\psi_{k,[L]}$ approximately satisfies \eqref{eq:outer-eq}.
Recall that $h_0$ denotes the terms in the right side of \eqref{eq:outer-diff} with no factors of $p$.
\begin{lemma}
  \label{lem:residual}
For all $k, L \in \N$, there exist $h^{\textnormal{res}}$ and $F^{\textnormal{res}}$ such that
  \begin{equation}
    \label{eq:residual}
    \partial_t \psi_{k,[L]} + \partial_x[A(\out{\psi}{0}) \psi_{k,[L]}] = \partial_x h_0 + \partial_x h^{\textnormal{res}} + F^{\textnormal{res}}
  \end{equation}
  and for all $m,n \in \N_0$,
  \begin{equation*}
    \begin{aligned}
      \absb{\partial_t^m \partial_x^n h_\sperp^{\textnormal{res}}} &\lesssim_{k,L,m,n} \nu^{(L - 1/2)\beta/3} \euc^{-2k - m + 1} \dist^{-3n},\\
      \absb{\partial_t^m \partial_x^n F_\sperp^{\textnormal{res}}} &\lesssim_{k,L,m,n} \nu^{(L - 1/2)\beta/3} \euc^{-2k - m} \dist^{-3n},
    \end{aligned}
  \end{equation*}
  and
  \begin{equation}
    \label{eq:sigma-residual}
    \abs{\partial_t^m \partial_x^n(\partial_xh_1^{\textnormal{res}} + F_1^{\textnormal{res}})} \lesssim_{k,L,m,n} \dist^{-2} \nu^{(L - 1/2)\beta/3} \euc^{-2k - m + 1} \dist^{-3n}.
  \end{equation}
\end{lemma}
\begin{proof}
  We first analyze the nonshocking components.
  Using Proposition~\ref{prop:omega-grid}, we can write
  \begin{equation*}
    \partial_t \omega_{k,\ell} =
    -\partial_x \sum_\bigsub{0 \leq r \leq \ell\\ \kappa_0 + \ldots + \kappa_r = k\\ \al_0 + \ldots + \al_r = \ell - r} \Aop_\sperp\psi_{\kappa,\al}
    +\partial_x \sum_\bigsub{0 \leq r \leq \ell-1\\ \kappa_0 + \ldots + \kappa_r = k-1\\ \al_0 + \ldots + \al_r = \ell - r - 1} \Bop_\sperp\psi_{\kappa,\al}
    - \nu^{-1} \delta_{k,1} F_\ell^\omega.
  \end{equation*}
  Recall the notation $\abs{\al} \coloneqq \al_0 + \ldots + \al_r$ and $\# \al \coloneqq r$.
  Then summing over $0 \leq \ell \leq L$, we can compactly write
  \begin{equation*}
    \partial_t \omega_{k,[L]} =
    -\partial_x \sum_\bigsub{\abs{\kappa} = k\\ \abs{\al} + \# \al \leq L} \Aop_\sperp\psi_{\kappa,\al}
    +\partial_x \sum_\bigsub{\abs{\kappa} = k-1\\ \abs{\al} + \# \al \leq L-1} \Bop_\sperp\psi_{\kappa,\al}
    - \nu^{-1} \delta_{k,1} F_{[L]}^\omega.
  \end{equation*}
  To compare this with $h_0$ in \eqref{eq:outer-diff}, we use the smallness of $\out{\psi}{0}$ in $\m{I}_\beta$ to expand $\Der^r A_\sperp(\out{\psi}{0})$ (after projecting onto the nonshocking subspace $\m{V}_\sperp$).
  Using Proposition~\ref{prop:inviscid-row}, we have
  \begin{equation*}
  \begin{aligned}
    \frac{\Der^rA_\sperp(\out{\psi}{0})}{(r + 1)!} &= \sum_{s = 0}^{L - r} \frac{\Der^{r + s}A_\sperp(0)}{(r + 1)!\,s!} (\out{\psi}{0})^s + \m{O}_L\big(\abss{\out{\psi}{0}}^{L - r + 1}\big)
    \\ &= \sum_{s = 0}^{L - r} \frac{\Der^{r + s}A_\sperp(0)}{(r + 1)!\,s!} \psi_{0,[L]}^s + \m{O}_L\big(\abss{\psi_{0,[L]}}^{L - r + 1}\big).
    \end{aligned}
  \end{equation*}
  Now $\frac{1}{(r + 1)!\, s!} = \frac{1}{(r + s + 1)!}\binom{r + s + 1}{s}$, and the binomial coefficient counts the number of ways of inserting $s$ factors of $\psi_{0,[L]}$ into a product $\prod_{j=0}^{r + s} \psi_{\kappa_j,[L]}$.
  Replacing $r + s$ by $r$, we can therefore write
  \begin{equation*}
    \sum_\bigsub{r \geq 0, \, \kappa_j \geq 1\\ \kappa_0 + \ldots + \kappa_r = k} \frac{\Der^rA_\sperp(\out{\psi}{0})}{(r + 1)!} \prod_{j=0}^r \psi_{\kappa_j,[L]} =
    \sum_\bigsub{\# \kappa \leq L\\ \abs{\kappa} = k} \Aop_\sperp\psi_{\kappa,[L]} + \m{O}_L\Big(\sup_{\substack{\#\kappa = L + 1\\ \abs{\kappa} = k}} \prod_j \abss{\psi_{\kappa_j,[L]}}\Big).
  \end{equation*}
  Expanding each $\psi_{\kappa,[L]} = \sum_{\ell \leq L} \psi_{\kappa,\ell}$, we find
  \begin{equation*}
  \begin{aligned}
    \sum_\bigsub{r \geq 0, \, \kappa_j \geq 1\\ \kappa_0 + \ldots + \kappa_r = k} \frac{\Der^rA_\sperp(\out{\psi}{0})}{(r + 1)!} &\prod_{j=0}^r \psi_{\kappa_j,[L]}
    - \sum_\bigsub{\abs{\kappa} = k\\ \abs{\al} + \# \al \leq L} \Aop_\sperp\psi_{\kappa,\al}\\
                                                                                                                                 &\lesssim_{k,L} \sup_{\substack{\#\al \leq L,\, \al_j \leq L\\\abs{\al} + \# \al > L\\\abs{\kappa} = k}} \prod_j \abss{\psi_{\kappa_j, \alpha_j}}
    + \sup_{\substack{\#\kappa = L + 1\\ \abs{\kappa} = k}} \prod_j \abss{\psi_{\kappa_j,[L]}}.
    \end{aligned}
  \end{equation*}
  Using Lemma~\ref{lem:prod}, we find
  \begin{equation*}
    \sup_{\substack{\#\kappa = \#\al \leq L\\\abs{\kappa} = k, \, \al_j \leq L\\ \abs{\al} + \# \al > L}} \prod_j \abss{\psi_{\kappa_j, \alpha_j}}
    \lesssim_{\eps,k,L} \nu^{(L - 3\eps)\beta/3} \euc^{-2k + 1}.
  \end{equation*}
  (Though the lemma is stated for $\m{A}_\sperp$, the proof controls the product above.)
  Next, \eqref{eq:distance-relation} and \eqref{eq:partial-bound} yield
  \begin{equation*}
    \sup_{\substack{\#\kappa = L + 1\\ \abs{\kappa} = k}} \prod_j \abss{\psi_{\kappa_j,[L]}} \lesssim_{k,L} \sup_{\substack{\#\kappa = L + 1\\ \abs{\kappa} = k}} \prod_j \dist \euc^{-2\kappa_j} \lesssim_{k, L} \dist^{L + 2} \euc^{-2k} \lesssim_{k,L} \nu^{\beta L/3} \euc^{-2k + 1}.
  \end{equation*}
  Combining these bounds, we have thus shown:
  \begin{equation*}
    \sum_\bigsub{r \geq 0, \, \kappa_j \geq 1\\ \kappa_0 + \ldots + \kappa_r = k} \frac{\Der^rA_\sperp(\out{\psi}{0})}{(r + 1)!} \prod_{j=0}^r \psi_{\kappa_j,[L]}
    = \sum_\bigsub{\abs{\kappa} = k\\ \abs{\al} + \# \al \leq L} \Aop_\sperp\psi_{\kappa,\al}
    + \m{O}_{k,L}\big(\nu^{(L - 3\eps)\beta/3} \euc^{-2k + 1}\big).
  \end{equation*}
  In a similar fashion, we can check that
  \begin{equation*}
    \hspace{14pt}\sum_\bigsub{r \geq 0, \, \kappa_{\geq 1} \hspace{0.5pt}\geq 1\\ \kappa_0 + \ldots + \kappa_r = k-1} \frac{\Der^rB(\out{\psi}{0})}{r!} (\partial_x\psi_{\kappa_0,[L]}) \prod_{j=1}^r \psi_{\kappa_j,[L]}
    = \sum_\bigsub{\abs{\kappa} = k-1\\ \abs{\al} + \# \al \leq L-1} \Bop_\sperp\psi_{\kappa,\al} + \m{O}_{k,L}\big(\nu^{(L - 3\eps)\beta/3} \euc^{-2k + 1}\big).
  \end{equation*}
  The details are sufficiently similar that we do not repeat them.
  We let $h_\sperp^{\text{res}}$ denote the sum of the error terms in the above two displays.
  
  Finally, when $k = 1$, \eqref{eq:distance-relation}, Lemma~\ref{lem:force-cancels}, and $\beta > (2 + \eps)^{-1}$, imply that
  \begin{equation*}
  \begin{aligned}
    F_\sperp^{\text{res}} \coloneqq -\nu^{-1} \delta_{k,1} F_{[L]}^\omega &\lesssim_L \nu^{-1 - o_L(1)} \dist^L \lesssim_L \nu^{(2 + \eps)\beta - 1-o_L(1)} \nu^{(L - 3\eps)\beta/3}\euc^{-2}
    \\ &\lesssim_L \nu^{(L - 3 \eps)\beta/3}\euc^{-2}.
    \end{aligned}
  \end{equation*}
  Now \eqref{eq:residual} follows from Proposition~\ref{prop:omega-grid} and \eqref{eq:outer-diff}.
  Derivative bounds follow as usual.

  We take a similar approach with the shocking component.
  Using Lemma~\ref{lem:prod}, we pick up an additional factor of $\dist$ in $h_1^{\text{res}}$ relative to the nonshocking estimates.
  Applying $\partial_x$, we obtain \eqref{eq:sigma-residual}.
\end{proof}

\subsection{Integrating residue equation}
\label{sec:integrate-residue}

Now our matching estimate is a simple consequence of the hyperbolic estimate Proposition~\ref{prop:outer-op} we stated earlier.
\begin{proof}[Proof of Proposition~\ref{prop:outer-matching}]
  Fix $k, L \in \N$.
  Let $\eps'$ be the mean of $\eps$ and $\beta^{-1} - 2$, so $\eps' \in (0, \eps)$ and $\beta > (2 + \eps')^{-1}$.
  Proceeding by induction, we assume that \eqref{eq:outer-matching} holds for all $k' < k$ with $\eps'$ in place of $\eps$.
  Note that $p_{k,L} = 0$ on $\Gamma_\beta$.
  Combining Lemmas~\ref{lem:diff-prod} and \ref{lem:residual}, we see that $\nu^{-(L - 3\eps')\beta/3} p_{k,L}$ satisfies the hypotheses of Proposition~\ref{prop:outer-op} with $\m{D} = \m{I}_\beta \setminus \m{I}_{1/2}$ and $\Sigma = \Gamma_\beta$.
  By said proposition,
  \begin{equation*}
    \abss{\partial_t^m \partial_x^n p_{k,L}} \lesssim_{k,L,m,n} \abs{\log \euc} \nu^{(L - 3\eps')\beta/3} \euc^{-2k - m + 1} \dist^{-3n}.
  \end{equation*}
  In $\m{I}_\beta \setminus \m{I}_{1/2}$, $\abs{\log \euc} \asymp \abs{\log \nu}$.
  Using $\abs{\log \nu} \nu^{-\eps'/\beta} \lesssim \nu^{-\eps/\beta}$, we obtain \eqref{eq:outer-matching}.
\end{proof}
It remains to prove the hyperbolic estimate.
\begin{proof}[Proof of Proposition~\ref{prop:outer-op}]
  For the reader's convenience, we recall the setting.
  Let $p$ solve $\partial_t p + \partial_x[A(\out{\psi}{0}) p] = \partial_x h + F$ in a causally closed domain $\m{D}$ with Cauchy hypersurface $\Sigma$, and suppose
  \begin{equation}
    \label{eq:residue-data}
    \abss{\nab^n (p|_{\Sigma}}) \lesssim_{n} \euc^{-2k + 1} \dist^{-3n}.
  \end{equation}
  Setting $a_k = \tbf{1}(k \geq 2)$, assume
  \begin{equation}
    \label{eq:residue-ns}
    \abss{\partial_t^m \partial_x^n h^\sperp} \lesssim_{m,n} \euc^{-2k - m + 1} \dist^{-3n}, \quad \abss{\partial_t^m \partial_x^n F^\sperp} \lesssim_{m,n} \euc^{-2k - m} \dist^{-3n},
  \end{equation}
  and
  \begin{equation}
    \label{eq:residue-sh}
    \abss{\partial_t^m \partial_x^n (\partial_x h^1 + F^1)} \lesssim_{m,n} \abs{\log \euc}^{a_k} \euc^{-2k - m + 1} \dist^{-3n-2}.
  \end{equation}
  We wish to show that
  \begin{equation}
    \label{eq:residue-bd}
    \abss{\partial_t^m \partial_x^n p} \lesssim_{m,n} \abs{\log \euc} \euc^{-2k - m + 1} \dist^{-3n}.
  \end{equation}
  We first suppose $m = n =0$ and employ the following bootstrap assumption for a time $\tb$:
  \begin{equation}
    \label{eq:residue-boot}
    \abss{p} \lesssim_{m,n} \euc^{-2k + 9/10} \quad \text{in } \m{D} \cap \{t < t_{\mathrm{b}}\}.
  \end{equation}
  We show that the stronger estimate \eqref{eq:residue-bd} holds until $\tb$.
  By continuity, we can push $\tb$ to $0$ and complete the proof.

  To avoid derivative loss, we work with the left eigenvectors $l$ and eigenvalues $\lambda$ associated to the variable matrix $A(\out{\psi}{0})$ rather than the frozen matrix $A(0)$.
  For the sake of brevity, in the following we implicitly compose $A$, $l$, $\lambda$, and their derivatives with $\out{\psi}{0}$.
  So, for example, $A$ represents $A(\out{\psi}{0})$.
  We note that Proposition~\ref{prop:inviscid-row} implies that
  \begin{equation}
    \label{eq:inviscid-derivs}
    \partial_t \out{\psi}{0} \lesssim \dist^{-1} \And \partial_x \out{\psi}{0} \lesssim \dist^{-2}.
  \end{equation}

  First fix a nonshocking index $2 \leq I' \leq N$.
  Let $p^{I'} \coloneqq l^{I'}p$, and similarly $h^{I'}$ and $F^{I'}$.
  Commuting the $p$-equation with $l^{I'}$, we find
  \begin{equation}
    \label{eq:ns-pre-renorm}
    (\partial_t + \lambda_{(I')}\partial_x) p^{I'} = [\Der \ell_{I'} \partial_t \out{\psi}{0} + \Der \ell_{I'} \partial_x \out{\psi}{0} A - \Der \lambda_{(I')} \partial_x \out{\psi}{0}] p + l^{I'} \partial_x h + F^{I'}.
  \end{equation}
  Let $M$ denote the matrix coefficient of $p$ in brackets.
  To renormalize, we write
  \begin{equation*}
    l^{I'} \partial_x h = \lambda_{(I')} \lambda_{(I')}^{-1} l^{I'} \partial_x h = \lambda_{(I')}\partial_x(\lambda_{(I')}^{-1} h^{I'}) - \lambda_{(I')} \Der(\lambda_{(I')}^{-1} l^{I'}) \partial_x \out{\psi}{0} h.
  \end{equation*}
  Then we can write \eqref{eq:ns-pre-renorm} as
  \begin{equation*}
    (\partial_t + \lambda_{(I')} \partial_x)(p - \lambda_{(I')} h^{I'}) = Mp - \partial_t (\lambda_{(I')}^{-1} h^{I'}) - \lambda_{(I')} \Der(\lambda_{(I')}^{-1} l^{I'}) \partial_x \out{\psi}{0} h + F^{I'}.
  \end{equation*}
  Because $I'$ is a nonshocking index, strict hyperbolicity implies that $\lambda_{(I')}^{-1}$ is uniformly bounded in a neighborhood of the origin.
  Combining \eqref{eq:residue-ns} and \eqref{eq:inviscid-derivs} with the bootstrap \eqref{eq:residue-boot}, we see that
  \begin{equation*}
    (\partial_t + \lambda_{(I')} \partial_x)(p - \lambda_{(I')} h^{I'}) \lesssim \dist^{-2} \euc^{-2k + 9/10} + \euc^{-2k}
  \end{equation*}
  and $\lambda_{(I')} h^{I'} \lesssim \euc^{-2k + 1}$.
  Combining \eqref{eq:nonshock} and \eqref{eq:nonshock-e-and-d} of Lemma~\ref{lem:integrals} and \eqref{eq:residue-data}, we obtain
  \begin{equation*}
    p^{I'} \lesssim \euc^{-2k + 1/3 + 9/10} + \euc^{-2k + 1} \lesssim \euc^{-2k + 1}.
  \end{equation*}
  We have thus improved on the bootstrap in the nonshocking component.

  Similar reasoning in the shocking component yields
  \begin{equation*}
    (\partial_t + \lambda_{(1)} \partial_x) p^1 = [\Der l^1 \partial_t \out{\psi}{0} + \Der l^1 \partial_x \out{\psi}{0} A - \Der \lambda_{(1)} \partial_x \out{\psi}{0}] p + l^1 \partial_x h + F^1.
  \end{equation*}
  Recall the right eigenvectors $r_I$ of $A$.
  Using our bounds on the nonshocking components and $\lambda_{(1)} \asymp \sigma_{0,0} \lesssim \dist$, we have
  \begin{equation*}
    \Der l^1 \partial_x \out{\psi}{0} A p = \Der l^1 \partial_x \out{\psi}{0} (\lambda_{(1)} p^1 r_1 + \lambda_{(I')} p^{I'}r_{I'}) = \m{O}(\dist^{-1}) p^1 + \m{O}(\dist^{-2} \euc^{-2k + 1}).
  \end{equation*}
  Next, we use Proposition~\ref{prop:inviscid-row} to write
  \begin{equation*}
    - \Der \lambda_{(1)} \partial_x \out{\psi}{0} p = -\partial_1 A_1^1(0) \partial_x \sigma_{0,0} p^1 + \m{O}(\dist^{-1}p^1) + \m{O}(\dist^{-2} \euc^{-2k + 1}).
  \end{equation*}
  Finally, \eqref{eq:residue-ns}, \eqref{eq:residue-sh}, and \eqref{eq:inviscid-derivs} yield
  \begin{equation*}
    l^1 \partial_x h = \partial_x h^1 - \Der l^1 \partial_x \out{\psi}{0} h \lesssim \abs{\log \euc}^{a_k} \dist^{-2} \euc^{-2k + 1}.
  \end{equation*}
  Collecting these observations and using \eqref{eq:residue-sh}, we have
  \begin{equation*}
    (\partial_t + \lambda_{(1)} \partial_x) p^1 = -\partial_1 A_1^1(0) \partial_x \sigma_{0,0} p^1 + \m{O}(\dist^{-1}p^1) + \m{O}(\abs{\log \euc}^{a_k} \dist^{-2} \euc^{-2k + 1}).
  \end{equation*}
  Now Proposition~\ref{prop:inviscid-row} implies that $\lambda_{(1)} = \partial_1 A_1^1(0)\sigma_{0,0} + \m{O}(\dist^2)$.
  Following the logic leading to \eqref{eq:sigma-hat-ev}, we multiply by the integrating factor $\dist^2$ to obtain
  \begin{equation*}
    (\partial_t + \lambda_{(1)} \partial_x) (\dist^2 p^1) = \m{O}(\dist p^1) + \m{O}(\abs{\log \euc}^{a_k} \euc^{-2k + 1}).
  \end{equation*}
  The bootstrap \eqref{eq:residue-boot} implies that the second term on the right dominates the first, so in fact
  \begin{equation*}
    (\partial_t + \lambda_{(1)} \partial_x) (\dist^2 p^1) \lesssim \abs{\log \euc}^{a_k} \euc^{-2k + 1}.
  \end{equation*}
  Using \eqref{eq:shock} or \eqref{eq:shock-log} of Lemma~\ref{lem:integrals} (if $k \geq 2$ or $k = 1$) and \eqref{eq:residue-data}, we conclude that
  \begin{equation*}
    \dist^2 p^1 \lesssim \abs{\log \euc} \euc^{-2k}.
  \end{equation*}
  The desired bound follows from \eqref{eq:distance-relation}.
  Having improved on the bootstrap, we see that in fact \eqref{eq:residue-bd} holds throughout $\m{D}$.

  Derivative estimates follow from analogous calculations.
  In the nonshocking case, we use a variant of $\bR_n$ from Lemma~\ref{lem:renorm} to renormalize $\partial_t^m \partial_x^n p$.
  Bounds on the derivatives along the Cauchy hypersurface $\Sigma$ follow from the reasoning leading to \eqref{eq:deriv-data}.
  We omit these similar calculations.
\end{proof}

\section{Inner estimates and inner-grid matching}
\label{sec:inner-matching}

We now show that summing the grid $(\psi_{k,\ell})$ in $k$ asymptotically yields the inner expansion $\psi_\ell$.
We work in inner coordinates, and therefore study the inner grid $\Psi_{k,\ell}$ defined in \eqref{eq:inner-conversion}.
Recalling the partial sum $\Psi_{[K],\ell} \coloneqq \sum_{k \leq K} \Psi_{k,\ell}$ and \eqref{eq:in-grid}, we expect $\inn{\Psi}{\ell} \approx \Psi_{[K],\ell}$ in the matching region $\nu^{1/2} \ll \euc \lesssim \nu^\beta$ when $K \gg 1$.

It will be convenient to extend statements to the ``true inner zone'' where $\euc \lesssim \nu^{1/2}$.
However, the terms $\Psi_{k,\ell}$ are singular at the origin while $\inn{\Psi}{\ell}$ remains regular.
To circumvent this difficulty, we cut off $\Psi_{k, \ell}$ in the true inner zone.
Let
\begin{equation*}
\begin{aligned}
  \Dist(T, X) \coloneqq \nu^{-1/4}\dist\big(\nu^{1/2}T, \nu^{3/4}X\big) &= \dist(T, X),
  \\ \Euc(T, X) \coloneqq \nu^{-1/2} \euc\big(\nu^{1/2}T, \nu^{3/4}X\big) &= \big(\abs{T}^2 + \nu^{1/2}\abs{X}^2\big)^{1/2},
  \end{aligned}
\end{equation*}
noting that $\Euc$ depends implicitly on $\nu$ because $\euc$ is not a homogeneous function under our cubic convention.
The true inner region $\{\euc \lesssim \nu^{1/2}\}$ corresponds to $\{\Euc \lesssim 1\}$.
Let $\theta \in \m{C}^\infty(\R)$ be a smooth cutoff function satisfying $\theta|_{[0, 1/2]} \equiv 0$ and $\theta|_{[1, \infty)} \equiv 1$.
We write $\theta_\Euc \coloneqq \theta \circ \Euc$ and $\theta_\Dist \coloneqq \theta \circ \Dist$ and define the cutoff grid
\begin{equation*}
  \cut{\Psi}{k,\ell} \coloneqq 
  \begin{cases}
    \theta_\Dist \Psi_{0,\ell} & \text{if } k = 0,\\
    \theta_\Euc \Psi_{k,\ell} & \text{if } k \geq 1
  \end{cases}
  \And
  \cut{\Psi}{[K],\ell} \coloneqq \sum_{k \leq K} \cut{\Psi}{k,\ell}.
\end{equation*}
We emphasize that we use a less aggressive cutoff $\theta_\Dist$ for the inviscid terms $\Psi_{0,\ell}$, which are known quite precisely through Proposition~\ref{prop:inviscid-row}.

Given $K \in \N$, we construct solutions of the inner equations that agree closely with $\cut{\Psi}{[K],\ell}$ away from the cutoff regions.
There is one wrinkle, however.
Recalling \eqref{eq:fit}, we will only employ the inner expansion where $\euc \leq \frac{\mathsf{R}}{2} \nu^\beta$.
Our shocking inner solution $\inn{\Sigma}{\ell}$ will solve \eqref{eq:Sigma-inner} exactly where $\euc \leq \tfrac{\mathsf{R}}{2} \nu^\beta$, but with additional forcing $G_\ell^\match$ where $\tfrac{\mathsf{R}}{2} \nu^\beta < \euc < \mathsf{R} \nu^\beta$.
So our solutions will satisfy
\begin{align}
  [\nu^{1/4} \partial_T + A_\sperp^\sperp(0)&] \partial_X \inn{\Omega}{\ell} = -\partial_X \sum_\bigsub{1 \leq r \leq \ell\\ \al_0 + \ldots + \al_r = \ell - r} \Aop_\sperp\inn{\Psi}{\al}
  + \partial_X \sum_\bigsub{0 \leq r \leq \ell - 1\\ \al_0 + \ldots + \al_r = \ell - r - 1} \Bop_\sperp\inn{\Psi}{\al},\label{eq:Omega-true}\\
  \partial_T \inn{\Sigma}{\ell} + \partial_\Is A_\Is^\Is(0) \partial_x(\inn{\Sigma}{0}\inn{\Sigma}{\ell})&\nonumber\\[-3pt]
  - B_\Is^\Is(0&)\partial_X^2 \inn{\Sigma}{\ell}
  = -\partial_X \sum^*_\bigsub{1 \leq r \leq \ell + 1\\ \al_0 + \ldots + \al_r = \ell - r + 1} \Aop_\Is\inn{\Psi}{\al}
  + \partial_X \sum^*_\bigsub{0 \leq r \leq \ell\\ \al_0 + \ldots + \al_r = \ell - r} \Bop_1\inn{\Psi}{\al} + G_\ell^\match.\label{eq:Sigma-true}
\end{align}
With this notation, we can state the main result of the section.
\begin{proposition}
  \label{prop:inner-matching}
  Fix $\eps \in (0, 1/6]$ and $\beta \in (\tfrac{1}{2 + \eps}, \tfrac{1}{2})$.
  For all $K \in \N$, there exists a solution $(\inn{\Psi}{\ell})_\ell$ of \eqref{eq:Omega-true} and \eqref{eq:Sigma-true} in $\m{I}_\beta$ with $G_\ell^\match \equiv 0$ where $\euc \leq \tfrac{\mathsf{R}}{2} \nu^\beta$ such that for all $\ell,m,n \in \N_0$,
  \begin{equation}
    \label{eq:inner-matching}
    \begin{aligned}
    \absb{\partial_T^m &\partial_X^n(\inn{\Psi}{\ell} - \cut{\Psi}{[K],\ell})}
    \\ &\lesssim_{\eps,\beta,K,\ell,m,n} \nu^{-(\ell + 1)(1/4 - \beta/3) - 2\beta/3 - \eps \beta + 1/2} (\Euc \vee 1)^{-2K - 2m - 1} (\Dist \vee 1)^{-3n}.
    \end{aligned}
  \end{equation}
\end{proposition}
We follow the approach to outer matching in Section~\ref{sec:outer-matching}.
Let $P_{K,\ell} \coloneqq \inn{\Psi}{\ell} - \cut{\Psi}{[K],\ell}$ denote the difference of interest with $P = (P^\Sigma, P^\Omega)$.
Then we can write \eqref{eq:Omega-inner} as
\begin{equation}
  \label{eq:Omega-diff}
  \begin{aligned}
    [\nu^{1/4} \partial_T& + A_\sperp^\sperp(0)\partial_X](\cut{\Omega}{[K],\ell} + P_{K,\ell}^\Omega)\\
                         &= -\partial_X \sum_\bigsub{1 \leq r \leq \ell\\ \al_0 + \ldots + \al_r = \ell - r} \Aop_\sperp(\cut{\Psi}{[K],\al} + P_{K,\al})
    + \partial_X \sum_\bigsub{0 \leq r \leq \ell - 1\\ \al_0 + \ldots + \al_r = \ell - r - 1} \Bop_\sperp(\cut{\Psi}{[K],\al} + P_{K,\al})\\[-10pt]
                         &\hspace{7.5cm}\eqqcolon \partial_X h_0 + \partial_X h_{\geq 1},
  \end{aligned}
\end{equation}
where $h_0$ denotes the collection of terms on the right with no factors of $P$, and $h_{\geq 1}$ the collection with at least one factor of $P$.
Similarly, \eqref{eq:Sigma-inner} yields
\begin{equation}
  \label{eq:Sigma-diff}
  \begin{aligned}
    \partial_T (&\cut{\Sigma}{[K],\ell} + P_{K,\ell}^\Sigma) + \partial_\Is A_\Is^\Is(0) \partial_x[\inn{\Sigma}{0}(\cut{\Sigma}{[K],\ell} + P_{K,\ell}^\Sigma)] - B_\Is^\Is(0)\partial_X^2 (\cut{\Sigma}{[K],\ell} + P_{K,\ell}^\Sigma)\\
                                                                   &=-\partial_X \sum^*_\bigsub{1 \leq r \leq \ell + 1\\ \al_0 + \ldots + \al_r = \ell - r + 1} \Aop_1(\cut{\Psi}{[K],\al} + P_{K,\al}) + \partial_X \sum^*_\bigsub{1 \leq r \leq \ell\\ \al_0 + \ldots + \al_r = \ell - r} \Bop_1(\cut{\Psi}{[K],\al} + P_{K,\al}) + G_\ell^\match
                                                                   \\ &\eqqcolon G_0 + G_{\geq 1} + G_\ell^\match.
  \end{aligned}
\end{equation}
We first inductively show that $h_{\geq 1}$ and $G_{\geq 1}$ are small.
We then argue that $\partial_X h_0$ and $G_0$ nearly cancel the partial-sum terms on the left of \eqref{eq:Omega-diff} and \eqref{eq:Sigma-diff}.
Finally, we integrate \eqref{eq:Omega-diff} and use the comparison principle in \eqref{eq:Sigma-diff} to show that $P$ is small.

Above, we assume $\ell \geq 1$.
In the zeroth column ($\ell = 0$), $\Psi_{k,0} = (\Sigma_{k,0}, 0)^\top$ and we constructed both $(\Sigma_{k,0})$ and $\inn{\Sigma}{0}$ in \cite{CG23}.
\begin{proposition}
  \label{prop:Burgers-inner}
  For all $K \in \N_0$ and $m,n \in \N_0$, $\abss{\partial_T^m \partial_X^n\inn{\Sigma}{0}} \lesssim_{m,n} (\Dist \vee 1)^{1 - 2m - 3n}$ and
  \begin{equation}
    \label{eq:Burgers-inner}
    \absb{\partial_T^m \partial_X^n(\inn{\Sigma}{0} - \cut{\Sigma}{[K],0})} \lesssim_{K,m,n} (\Dist \vee 1)^{-4K - 2m - 3n - 3} \quad \text{in } \m{I}_\beta.
  \end{equation}
\end{proposition}
\noindent
This is essentially Corollary~4.3 and Proposition~4.2 in \cite{CG23}.
There we did not explicitly treat time derivatives, but the estimate \eqref{eq:Burgers-inner} follows from the evolution equation for $\inn{\Sigma}{0}$ and induction in $m$.
We note that the region $\m{I}_\beta$ looks large in the inner coordinates, for there $\Euc \lesssim \nu^{-(1/2 - \beta)}$.

We also record bounds in the outer coordinates, scaling $\cut{\psi}{0,\ell}$ in the natural way.
\begin{corollary}
  \label{cor:worst-inner-size}
  Fix $\eps \in (0, 1/6]$ and $\beta \in (\tfrac{1}{2 + \eps}, \tfrac{1}{2})$.
  For all $K,\ell \in \N$ and $m,n \in \N_0$,
  \begin{align}
    \absb{\partial_t^m \partial_x^n(\inn{\psi}{\ell} - \cut{\psi}{0,\ell})} &\lesssim_{K,\ell,m,n} \nu^{1 + (\ell - 1)\beta/3 - \eps \beta} (\euc \vee \nu^{1/2})^{-m-1} (\dist \vee \nu^{1/4})^{-3n}\label{eq:cut-diff},\\
    \absb{\partial_t^m \partial_x^n \inn{\psi}{\ell}} &\lesssim_{K,\ell,m,n} \abs{\log \nu}^{\ell - 1} \nu^{(\ell + 1)\beta/3} (\euc \vee \nu^{1/2})^{-m} (\dist \vee \nu^{1/4})^{-3n}\label{eq:inner-bounds}
  \end{align}
  in $\m{I}_\beta$.
  Moreover, in $\m{I}_\beta \setminus \big\{\euc \leq \tfrac{\mathsf{R}}{2} \nu^\beta\big\}$,
  \begin{equation}
    \label{eq:inner-matching-outer-form}
    \absb{\partial_t^m \partial_x^n(\inn{\psi}{\ell} - \psi_{[K],\ell})} \lesssim_{K,\ell,m,n} \nu^{(1 - 2\beta)K + (\ell + 1)\beta/3 - m/2 - 3n/4}.
  \end{equation}
\end{corollary}

\subsection{Difference estimates}

In this section we control the size of $h_{\geq 1}$ in \eqref{eq:Omega-diff} and $G_{\geq 1}$ in \eqref{eq:Sigma-diff}.
We use of the following bounds on the partial sums $\cut{\Psi}{[K],\ell}$.
\begin{lemma}
  \label{lem:inner-partial}
  For all $K \in \N$ and $\ell, m, n \in \N_0$, in $\m{I}_\beta$ we have
  \begin{align}
    \absb{\partial_T^m \partial_X^n \cut{\Psi}{[K],\ell}} &\lesssim_{K,\ell,m,n} \abs{\log \nu}^{\ell - 1} \nu^{-\ell(1/4 - \beta/3)} (\Euc \vee 1)^{-2m} (\Dist \vee 1)^{-3n + 1},
    \label{eq:inner-partial-fine}\\
    &\lesssim_{K,\ell,m,n} \abs{\log \nu}^{\ell - 1} \nu^{-(\ell + 1)(1/4 - \beta/3)} (\Euc \vee 1)^{-2m} (\Dist \vee 1)^{-3n}.
    \label{eq:inner-partial}
  \end{align}
\end{lemma}
\begin{proof}
  Thanks to the cutoff, it suffices to treat $\nu^{1/2} \lesssim \euc \lesssim \nu^{\beta}$.
  There, Propositions~\ref{prop:inviscid-row} and \ref{prop:grid-est} yield $\psi_{0,\ell} \lesssim \abs{\log \nu}^{\ell - 1} \nu^{\ell \beta/3} \dist$ and $\nu^k\psi_{k,\ell} \lesssim \nu^{1/2 + (\ell - 3/2)\beta/3} \lesssim \nu^{\ell\beta/3 + 1/4}$, using $\beta < 1/2$ in the last bound.
  Switching to inner coordinates, we multiply by $\nu^{-(\ell + 1)/4}$ and use $\dist = \nu^{1/4}\Dist$ to obtain \eqref{eq:inner-partial-fine} for $m = n = 0$.
  Derivative estimates follow in the usual manner.
  Because $\Dist \lesssim \nu^{-(1/4 - \beta/3)}$ in $\m{I}_\beta$, \eqref{eq:inner-partial} follows from \eqref{eq:inner-partial-fine}.
\end{proof}
We now establish an inductive bound on the terms in \eqref{eq:Omega-diff} and \eqref{eq:Sigma-diff} containing at least one factor of $P$.
Throughout this section, we use the shorthand
\begin{equation}
  \label{eq:exponent}
  \eta \coloneqq \frac{1}{2} - \frac{2 \beta}{3} - \eps \beta \in \Big(\frac{1}{12}, \frac{1}{6}\Big).
\end{equation}
The given bounds follow from the conditions $\beta > \tfrac{1}{2 + \eps}$ and $\beta < 1/2$, $\eps \leq 1/6$.
\begin{lemma}
  \label{lem:inner-diff}
  Fix $\eps \in (0, 1/6],$ $\beta \in (\tfrac{1}{2 + \eps}, \tfrac{1}{2})$, and $K,\ell \in \N$.
  If Proposition~\ref{prop:inner-matching} holds for all $\ell' < \ell$, then for all $m, n \in \N_0$,
  \begin{equation}
    \label{eq:Omega-diff-prod}
    \absb{\partial_T^m \partial_X^n h_{\geq 1}} \lesssim_{\eps,\beta,K,\ell,m,n} \nu^{-(\ell + 1)(1/4 - \beta/3) + \eta} (\Euc \vee 1)^{-2K - 2m - 1} (\Dist \vee 1)^{-3n}.
  \end{equation}
  If moreover Proposition~\ref{prop:inner-matching} holds for $P_{K,\ell}^\Omega$, then
  \begin{equation}
    \label{eq:Sigma-diff-prod}
    \absb{\partial_T^m \partial_X^n G_{\geq 1}} \lesssim_{\eps,\beta,K,\ell,m,n} \nu^{-(\ell + 1)(1/4 - \beta/3) + \eta} (\Euc \vee 1)^{-2K - 2m - 1} (\Dist \vee 1)^{-3n - 2}.
  \end{equation}
\end{lemma}
\begin{proof}
  Fix $\eps \in (0, 1/6]$ and $\beta \in \big(\tfrac{1}{2 + \eps}, \tfrac{1}{2}\big)$.
  Choose $\eps' \in (0, \eps)$ such that $\tfrac{1}{2 + \eps'}$ is midway between $\tfrac{1}{2 + \eps}$ and $\beta$.
  Define $\eta' > \eta$ with $\eps'$ in place of $\eps$ in \eqref{eq:exponent}.
  For $1 \leq \ell' < \ell$, Proposition~\ref{prop:inner-matching} yields
  \begin{equation}
    \label{eq:diff-induction}
    \abss{P_{K,\ell'}} \lesssim_{\eps,\beta,K,\ell} \nu^{-(\ell' + 1)(1/4 - \beta/3) + \eta'} (\Euc \vee 1)^{-2K - 1}.
  \end{equation}
  Comparing with \eqref{eq:inner-partial-fine} and using $\eta' - (1/4 - \beta/3) > \eta' - 1/12 > 0$, we see that the dominant contributions to $h_{\geq 1}$ and $F_{\geq 1}$ involve exactly one factor of $P$ when $\ell' \geq 1$.
  Thanks to Proposition~\ref{prop:Burgers-inner}, the same holds for $\ell' = 0$.
  We therefore only treat terms with exactly one factor of $P$.
  
  First consider the contributions to $h_{\geq 1}$ from $\Aop$ in \eqref{eq:Omega-diff}.
  Given $1 \leq i \leq r$,
  Lemma~\ref{lem:inner-partial} yields
  \begin{equation*}
    P_{K,\al_i} \prod_{j \neq i} \cut{\Psi}{[K],\al_j} \lesssim_{\eps,\beta,K,\ell} \nu^{-(\al_i + 1)(1/4 - \beta/4) + \eta'} (\Euc \vee 1)^{-2K - 1} \prod_{j \neq i} \nu^{-(\al_j + 1)(1/4 - \beta/3) - o_\ell(1)}
  \end{equation*}
  For these terms, $\sum_j(\al_j + 1) = \ell + 1$, so this simplifies to
  \begin{equation*}
    P_{K,\al_i} \prod_{j \neq i} \cut{\Psi}{[K],\al_j} \lesssim_{\eps,\beta,K,\ell} \nu^{-(\ell + 1)(1/4 - \beta/4) + \eta' - o_\ell(1)} (\Euc \vee 1)^{-2K - 1}.
  \end{equation*}
  Now $\eta' > \eta$, so the right side of \eqref{eq:Omega-diff-prod} controls these terms.
  For the contributions to $h_{\geq 1}$ from $\Bop$, the derivative $\partial_X$ is helpful in the inner coordinates (multiplying by $(\Dist \vee 1)^{-3} \lesssim 1$).
  Moreover, these terms satisfy $\sum_j(\al_j + 1) = \ell$, leaving us with a less singular $\nu$-factor.
  So these terms are likewise controlled by the right side of \eqref{eq:Omega-diff-prod}, completing the proof for $m = n = 0$.
  Derivatives multiply our estimates by factors of $(\Euc \vee 1)^{-1}$ or $(\Dist \vee 1)^{-3}$ in the usual fashion.
  
  We now turn to $G_{\geq 1}$ in \eqref{eq:Sigma-diff}.
  First consider contributions from $P_{K,\ell}^\Omega$, for which we can assume
  \begin{equation*}
    \abss{P_{K,\ell}^\Omega} \lesssim_{\eps,\beta,K,\ell} \nu^{-(\ell + 1)(1/4 - \beta/3) + \eta'} (\Euc \vee 1)^{-2K - 1}.
  \end{equation*}
  Hence by Proposition~\ref{prop:Burgers-inner}, we have
  \begin{align*}
    \partial_X(\inn{\Sigma}{0} P_{K,\ell}^\Omega) \lesssim_{\eps,\beta,K,\ell} \nu^{-(\ell + 1)(1/4 - \beta/3) + \eta'} (\Euc \vee 1)^{-2K - 1} (\Dist \vee 1)^{-2},\\
    \partial_X^2 P_{K,\ell}^\Omega \lesssim_{\eps,\beta,K,\ell} \nu^{-(\ell + 1)(1/4 - \beta/3) + \eta'} (\Euc \vee 1)^{-2K - 1} (\Dist \vee 1)^{-6}.
  \end{align*}
  Next consider the $\Aop$ contributions to $G_{\geq 1}$ in \eqref{eq:Sigma-diff}.
  First, $\partial_X$ contributes a factor of $(\Dist \vee 1)^{-3}$.
  Second, because $r \geq 1$ in that sum, these terms include at least one factor of $\cut{\Psi}{[K],\ell'}$.
  We use \eqref{eq:inner-partial-fine} for one such factor and \eqref{eq:inner-partial} for the rest.
  Combining this with \eqref{eq:diff-induction} and $\sum_j (\al_j + 1) = \ell + 2$, we find
  \begin{equation*}
    P_{K,\al_i} \prod_{j \neq i} \cut{\Psi}{[K],\al_j} \lesssim_{\eps,\beta,K,\ell} (\Dist \vee 1)^{-2} \nu^{-(\ell + 1)(1/4 - \beta/3) + \eta' - o_\ell(1)} (\Euc \vee 1)^{-1}.
  \end{equation*}
  Compared to our earlier $\Omega$-calculation, the refined estimate \eqref{eq:inner-partial-fine} allows us to trade one factor of $\nu^{-(1/4 - \beta/3)}$ for $\Dist \vee 1$ and thus overcome the larger value of $\sum_i (\al_i + 1) = \ell + 2$.
  We now use $\eta' > \eta$ to absorb $\nu^{-o_\ell(1)}$.
  
  The $\Bop$ contributions satisfy $\sum_j (\al_j + 1) = \ell + 1$, so we can bound them using \eqref{eq:inner-partial} and \eqref{eq:diff-induction}, much as in the nonshocking calculation.
  We obtain \eqref{eq:Sigma-diff-prod} for $m = n = 0$, and derivative estimates routinely follow.
\end{proof}

\subsection{Partial sum equation}
Next, we show that $\cut{\Omega}{[K],\ell}$ and $\cut{\Sigma}{[K],\ell}$ approximately solve \eqref{eq:Omega-inner} and \eqref{eq:Sigma-inner}, respectively.
Recall that $h_0$ and $G_0$ account for terms with no factors of $P$ in \eqref{eq:Omega-diff} and \eqref{eq:Sigma-diff}.
\begin{lemma}
  \label{lem:residual-inner}
  For all $K, \ell \in \N$, there exist $h^\res$, $F^\res$, and $G^\res$ such that
  \begin{gather}
    [\nu^{1/4}\partial_T + A_\sperp^\sperp(0)\partial_X]\cut{\Omega}{[K],\ell} = \partial_X(h_0 + h^\res) + F^\res,\label{eq:residual-nonshocking}\\
    \partial_T \cut{\Sigma}{[K],\ell} + \partial_1 A_1^1(0)\partial_X(\inn{\Sigma}{0} \cut{\Sigma}{[K],\ell}) - B_1^1(0) \partial_X^2 \cut{\Sigma}{[K],\ell} = G_0 + G^\res,\label{eq:residual-shocking}
  \end{gather}
  and for all $m, n \in \N_0$,
  \begin{align*}
    \absb{\partial_T^m \partial_X^n h^\res} &\lesssim_{\eps,\beta,k,L,m,n} \nu^{-(\ell + 1)(1/4 - \beta/3) + \eta} (\Euc \vee 1)^{-2K - 2m - 1} (\Dist \vee 1)^{-3n},\\
    \absb{\partial_T^m \partial_X^n F^\res} &\lesssim_{\eps,\beta,k,L,m,n} \nu^{-(\ell + 1)(1/4 - \beta/3) + \eta + 1/4} \tbf{1}_{\Euc \leq 1} (\Dist \vee 1)^{-3n},\\
    \absb{\partial_T^m \partial_X^n G^\res} &\lesssim_{\eps,\beta,k,L,m,n} \nu^{-(\ell + 1)(1/4 - \beta/3) + \eta} (\Euc \vee 1)^{-2K - 2m - 1} (\Dist \vee 1)^{-3n - 2}.
  \end{align*}
\end{lemma}
We make use of the following restatement of Lemma~\ref{lem:prod}, which is immediate from the multilinearity of $\Aop$ and $\Bop$.
\begin{lemma}
  \label{lem:inner-prod}
  Fix $\eps \in (0,1/6]$, $\beta \in \big(\tfrac{1}{2 + \eps}, \tfrac{1}{2}\big)$, $k,\ell \in \N$, and $m,n \in \N_0$.
  \begin{enumerate}[label = \textnormal{(\roman*)}]
  \item
    If $\abs{\al} + \#\al = \ell$ and $\abs{\kappa} = k$,
    \begin{equation*}
      \abss{\partial_T^m \partial_T^n \Aop_\sperp\Psi_{\kappa,\al}} \lesssim_{\eps,\beta,k,\ell,m,n} \nu^{-(\ell + 1)(1/4 - \beta/3) + \eta} \Euc^{-2k - m + 1} \Dist^{-3n}.
    \end{equation*}

  \item
    If $\abs{\al} + \#\al = \ell - 1$ and $\abs{\kappa} = k - 1$,
    \begin{equation*}
      \abss{\partial_T^m \partial_X^n \Bop_\sperp\Psi_{\kappa,\al}} \lesssim_{\eps,\beta,k,\ell,m,n} \nu^{-(\ell + 1)(1/4 - \beta/3) + \eta} \Euc^{-2k - m + 1} \Dist^{-3n}.
    \end{equation*}

  \item
    If $\abs{\al} + \#\al = \ell + 1$ and $\abs{\kappa} = k$,
    \begin{equation*}
      \abss{\partial_T^m \partial_X^n \Aop_1\Psi_{\kappa,\al}} \lesssim_{\eps,\beta,k,\ell,m,n} \Dist \nu^{-(\ell + 1)(1/4 - \beta/3) + \eta} \Euc^{-2k - m + 1} \Dist^{-3n}.
    \end{equation*}

  \item
    If $\abs{\al} + \#\al = \ell$ and $\abs{\kappa} = k - 1$,
    \begin{equation*}
      \abss{\partial_T^m \partial_X^n \Bop_1\Psi_{\kappa,\al}} \lesssim_{\eps,\beta,k,\ell,m,n} \Dist \nu^{-(\ell + 1)(1/4 - \beta/3) + \eta} \Euc^{-2k - m + 1} \Dist^{-3n}.
    \end{equation*}
  \end{enumerate}
\end{lemma}
\begin{proof}[Proof of Lemma~\textnormal{\ref{lem:residual-inner}}]
  The residues in the lemma come from two sources: the cutoff and the fact that the multilinear operators $\Aop$ and $\Bop$ do not commute with partial summation.
  We treat cutoff error first.
  
  The inner equations involve the operators $\partial_T$, $\partial_X \Aop$, and $\partial_X\Bop$.
  We multiply these equations by a cutoff $\theta$ and move it through the operators to land on $\Psi_{k,\ell}$, forming $\cut{\Psi}{k,\ell}$.
  In the process, we generate commutators of the form $[\theta, \partial_T]$, $[\theta,\partial_X] \Aop$, and $\partial_X [\theta, \Aop]$ (resp. $\Bop$).
  In the nonshocking component, we allot the first two to $F^\res$ and the last to $h^\res$ (as a derivative in $X$).
  We group all in $G^\res$ in the shocking component.

  Recall that we use different cutoffs $\theta_\Dist$ and $\theta_{\Euc}$ for $\Psi_{0,\ell}$ and $\Psi_{\geq 1, \ell}$.
  These excise the regions $\Dist \leq 1/2$ and $\Euc \leq 1/2$, respectively.
  To treat the distinction, let $\Psi_{[1,K],\ell} \coloneqq \Psi_{[K],\ell} - \Psi_{0,\ell}$ denote the sum of terms in $\Psi_{[K],\ell}$ with $k \geq 1$.
  Then $\cut{\Psi}{0,\ell} = \theta_\Dist \Psi_{0,\ell}$ and $\cut{\Psi}{[1,K],\ell} = \theta_\Euc \Psi_{[1,K],\ell}$.

  First consider the nonshocking component.
  Writing \eqref{eq:inviscid-row} in the inner coordinates and commuting with $\theta_\Dist$, we have
  \begin{equation}
    \label{eq:Omega-inviscid-commutators}
    \nu^{1/4} \partial_T \cut{\Omega}{0,\ell} = \partial_X \sum_\bigsub{\abs{\al} + \#\al = \ell} \Aop_\sperp \cut{\Psi}{0,\al} + \partial_Xh_1 + F_1 + \theta_\Dist F_\ell^\Omega,
  \end{equation}
  where $F_\ell^\Omega$ denotes the scaling of $F_\ell^\omega$ to inner coordinates,
  \begin{equation*}
    h_1 \coloneqq \sum_\bigsub{\abs{\al} + \#\al = \ell} [\theta_\Dist, \Aop_\sperp] \Psi_{0,\al}, \And F_1 \coloneqq -\nu^{1/4}[\theta_\Dist, \partial_T] \Omega_{0,\ell} + \sum_\bigsub{\abs{\al} + \#\al = \ell} [\theta_\Dist, \partial_X] \Aop_\sperp \Psi_{0,\al}.
  \end{equation*}
  We can scale Proposition~\ref{prop:inviscid-row} to find $\abss{\Psi_{0,\ell}} \lesssim_\ell \abs{\log \nu}^\ell$ where $\Dist \asymp 1$, and likewise for derivatives.
  The above commutators are supported where $1/2 \leq \Dist \leq 1$, so $\abs{h_1} + \abs{F_1} \lesssim_{\ell} \nu^{-o_\ell(1)} \tbf{1}_{\Dist \leq 1}$.
  Because $\ell \geq 1$, $\eta < (\ell + 1)(1/4 - \beta/3)$ and we can bound this by a multiple of $\nu^{-(\ell + 1)(1/4 - \beta/3) + \eta}\tbf{1}_{\Dist \leq 1}$.

  For $k \geq 1$, we write \eqref{eq:omega-grid} in inner coordinates and commute with $\theta_\Euc$ to obtain
  \begin{align}
    \nu^{1/4} \partial_T \cut{\Omega}{[1,K],\ell} &= \overbrace{-\nu^{1/4}[\theta_\Euc, \partial_T] \Omega_{[1,K],\ell}}^{F_2}\nonumber\\
    &+ \sum_\bigsub{\abs{\al} + \#\al = \ell\\ 1 \leq \abs{\kappa} \leq K} \big(\partial_X\Aop_\sperp \cut{\Psi}{\kappa,\al} + \overbrace{[\theta_\Euc,\partial_X] \Aop_\sperp \Psi_{\kappa,\al}}^{F_2} + \partial_X \overbrace{[\theta_\Euc, \Aop_\sperp] \Psi_{\kappa,\al}}^{h_2}\big)\label{eq:Omega-viscous-commutators}\\
    &+ \partial_X \sum_\bigsub{\abs{\al} + \#\al = \ell - 1\\ \abs{\kappa} \leq K - 1} \big(\partial_X\Bop_\sperp \cut{\Psi}{\kappa,\al} + \underbrace{[\theta_\Euc,\partial_X] \Bop_\sperp \Psi_{\kappa,\al}}_{F_2} + \partial_X \underbrace{[\theta_\Euc, \Bop_\sperp] \Psi_{\kappa,\al}}_{h_2}\big)
    - \theta_\Euc F_\ell^\Omega.\nonumber
  \end{align}
  We use braces to label the various contributions to $h_2$ and $F_2$.
  Now $[\theta_\Euc, \partial] = \partial \theta_\Euc$, and $\abs{\partial_T \theta_\Euc} \lesssim 1$ while $\abs{\partial_X \theta_\Euc} \lesssim \nu^{1/4}$.
  Using Lemma~\ref{lem:inner-prod}, we thus find
  \begin{equation*}
    \abs{F_2} \lesssim_{\eps,\beta,K,\ell} \nu^{-(\ell + 1)(1/4 - \beta/3) + \eta + 1/4} \tbf{1}_{\Euc \leq 1}.
  \end{equation*}
  Commuting the cutoff through the multilinear operators, we find $[\theta_\Euc, \Aop] \lesssim \abs{\Aop}$ and $[\theta_\Euc, \Bop] \lesssim \abs{\Bop} + \nu^{1/4}\abs{\Aop},$ with the second term due to the derivative in $\Bop$.
  Thus Lemma~\ref{lem:inner-prod} yields
  \begin{equation*}
    h_2 \lesssim_{\eps,\beta,K,\ell} \nu^{-(\ell + 1)(1/4 - \beta/3) + \eta} \tbf{1}_{\Euc \leq 1}.
  \end{equation*}
  When we add \eqref{eq:Omega-inviscid-commutators} and \eqref{eq:Omega-viscous-commutators}, we obtain the difference $F_3 \coloneqq (\theta_\Dist - \theta_\Euc) F_\ell^\Omega$.
  This is supported where $\Euc \lesssim 1$ and hence $\Dist \lesssim \nu^{-1/12}$.
  Using Proposition~\ref{prop:inviscid-row} and \eqref{eq:exponent}, we can therefore compute
  \begin{equation*}
    F_3 \lesssim_\ell \nu^{1/4 - o_\ell(1)} \Dist^{\ell - 1} \tbf{1}_{\Euc \lesssim 1} \lesssim_\ell \nu^{1/4 - (\ell + 1)/12 + 1/6 - o_\ell(1)} \tbf{1}_{\Euc \lesssim 1} \lesssim_\ell \nu^{-(\ell + 1)(1/4 - \beta/3) + \eta + 1/4} \tbf{1}_{\Euc \lesssim 1}.
  \end{equation*}
  Adding \eqref{eq:Omega-inviscid-commutators} and \eqref{eq:Omega-viscous-commutators}, we can thus write
  \begin{equation}
    \label{eq:Omega-commutators}
    \nu^{1/4} \partial_T \cut{\Omega}{[K],\ell} = \partial_X \sum_\bigsub{\abs{\al} + \#\al = \ell\\ \abs{\kappa} \leq K} \Aop_\sperp \cut{\Psi}{\kappa,\al}
    + \partial_X \sum_\bigsub{\abs{\al} + \#\al = \ell - 1\\ \abs{\kappa} \leq K - 1} \Bop_\sperp \cut{\Psi}{\kappa,\al}
    + \partial_X(h_1 + h_2) + F_1 + F_2 + F_3
  \end{equation}
  for errors $h_i$ and $F_i$ that are suitably small.

  For the shocking component, we collect all the $\theta_\Dist$-commutators from $\Sigma_{0,\ell}$ in $G_1$ and the $\theta_\Euc$-commutators from $\Sigma_{[1,K],\ell}$ in $G_2$.
  The former are order $\nu^{o_\ell(1)} \tbf{1}_{\Dist \leq 1}$, and the latter are comparable to the size of $\Aop$ and $\Bop$ in the relevant sums.
  Using (iii) and (iv) in Lemma~\ref{lem:inner-prod} and accounting for the factor of $\Dist^{-3}$ from $\partial_X$, we find
  \begin{equation*}
    G_1 + G_2 \lesssim_{\eps,\beta,K,\ell} (\Dist \vee 1)^{-2} \nu^{-(\ell + 1)(1/4 - \beta/3)} \tbf{1}_{\Euc \lesssim 1}.
  \end{equation*}
  Let $G_3 \coloneqq (\theta_\Dist - \theta_\Euc) F_\ell^\Sigma$ and note that $\theta_\Dist - \theta_\Euc$ is supported where $\Dist \gtrsim 1$ and $\Euc \lesssim 1$, so $\Dist \lesssim \nu^{-1/12}$.
  Using Proposition~\ref{prop:inviscid-row}, we can compute $F_\ell^\Sigma \lesssim_\ell \nu^{1/4 - o_\ell(1)} \Dist^{\ell}$, so \eqref{eq:exponent} yields
  \begin{equation*}
    G_3 \lesssim_\ell (\Dist \vee 1)^{-2} \nu^{1/4 - (\ell + 1)/12 - 1/12 - o_\ell(1)} \tbf{1}_{\Euc \lesssim 1} \lesssim_{\eps,\beta,K,\ell} (\Dist \vee 1)^{-2} \nu^{-(\ell + 1)(1/4 - \beta/3) + \eta} \tbf{1}_{\Euc \lesssim 1}.
  \end{equation*}
  Then we can write
  \begin{equation}
    \label{eq:Sigma-commutators}
    \partial_T \cut{\Sigma}{[K],\ell} = \partial_X \sum_\bigsub{\abs{\al} + \#\al = \ell + 1\\ \abs{\kappa} \leq K} \partial_X\Aop_1 \cut{\Psi}{\kappa,\al}
    + \partial_X \sum_\bigsub{\abs{\al} + \#\al = \ell\\ \abs{\kappa} \leq K - 1} \Bop_1 \cut{\Psi}{\kappa,\al}
    + G_1 + G_2 + G_3.
  \end{equation}

  We now turn to the algebraic residue arising from the action of $\Aop$ and $\Bop$ on partial sums.
  For a given $\al$, the inner equations \eqref{eq:Omega-diff} and \eqref{eq:Sigma-diff} involve $\Aop \cut{\Psi}{[K],\al}$ while \eqref{eq:Omega-commutators} and \eqref{eq:Sigma-commutators} feature $\sum_{\abs{\kappa} \leq K} \Aop \cut{\Psi}{\kappa,\al}$ (and similarly for $\Bop$).
  If we expand the former, we will find terms with $\abs{\kappa} > K$ not contained in the latter.
  We group these terms into $h_3$ for $\Omega$ and $G_4$ for $\Sigma$:
  \begin{equation*}
    h_3 \coloneqq \sum_\bigsub{\abs{\al} + \#\al = \ell\\\kappa_j \leq K\\ \abs{\kappa} > K} \Aop_\sperp \cut{\Psi}{[K],\al} + \sum_\bigsub{\abs{\al} + \#\al = \ell - 1\\\kappa_j \leq K\\ \abs{\kappa} > K - 1} \Bop_\sperp \cut{\Psi}{[K],\al},
  \end{equation*}
  and
  \begin{equation*}
    G_4 \coloneqq \partial_X \sum_\bigsub{\abs{\al} + \#\al = \ell + 1\\\kappa_j \leq K\\ \abs{\kappa} > K} \Aop_1 \cut{\Psi}{[K],\al} + \partial_X \sum_\bigsub{\abs{\al} + \#\al = \ell\\\kappa_j \leq K\\ \abs{\kappa} > K - 1} \Bop_1 \cut{\Psi}{[K],\al}.
  \end{equation*}
  Now Lemma~\ref{lem:inner-prod} yields
  \begin{equation*}
    h_3 \lesssim_{\eps,\beta,K,\ell} \nu^{-(\ell + 1)(1/4 - \beta/3) + \eta} (\Euc \vee 1)^{-2K-1},
  \end{equation*}
  and
  \begin{equation*}
    G_4 \lesssim_{\eps,\beta,K,\ell} (\Dist \vee 1)^{-2} \nu^{-(\ell + 1)(1/4 - \beta/3) + \eta} (\Euc \vee 1)^{-2K-1}.
  \end{equation*}
  We are able to cut off $\Euc$ and $\Dist$ in these bounds because $\abs{\kappa} > K > 0$ in the above sums, so each product $\Aop$ and $\Bop$ contains at least one factor of $\theta_\Euc$.
  Hence $h_3$ and $G_4$ are supported where $\Euc \gtrsim 1/2$.

  There is one final source of error in the shocking component: \eqref{eq:Sigma-diff} involves advection by $\inn{\Sigma}{0}$.
  We have implicitly replaced this by $\cut{\Sigma}{[K],0}$ in the above sums.
  Therefore let $G_5 \coloneqq  - \partial_1 A_1^1(0) \partial_X(P_{K,0} \Psi_{[K],\ell})$, which accounts for this difference.
  Writing $\Psi_{[K],\ell} = \Psi_{0,\ell} + \Psi_{[1,K],\ell}$ and using Propositions~\ref{prop:inviscid-row}, \ref{prop:grid-est}, and \ref{prop:Burgers-inner}, we find
  \begin{align*}
    G_5 &\lesssim_{\eps,\beta,K,\ell} \nu^{-o_\ell(1)} \Dist^{\ell - 4K - 5} + \nu^{-(\ell + 1)(1/4 - \beta/3) + \eta} \Euc^{-1} \Dist^{-4K - 6}\\
        &\lesssim_{\eps,\beta,K,\ell} [\nu^{-o_\ell(1)} \Dist^{\ell - 1} + \nu^{-(\ell + 1)(1/4 - \beta/3) + \eta}] \Euc^{-2K-1} \Dist^{-2}
        \\ &\lesssim_{\eps,\beta,K,\ell} \nu^{-(\ell + 1)(1/4 - \beta/3) + \eta} \Euc^{-2K-1} \Dist^{-2}.
  \end{align*}
  Here we have used $\Dist \lesssim \nu^{1/4 - \beta/3}$ and $2(1/4 - \beta/3) > \eta$.

  If we now set $h^\res \coloneqq \sum_{i=1}^3 h_i$, $F^\res \coloneqq \sum_{i=1}^3 F_i$, and $G \coloneqq \sum_{i=1}^5 G_i$, \eqref{eq:residual-nonshocking} and \eqref{eq:residual-shocking} hold with the desired bounds for $m = n = 0$.
  Higher derivative estimates require repeated commutation with $\theta_\Euc$, but the outcome is little changed.
\end{proof}

\subsection{Matching estimates}
Throughout this section, we inductively assume that Proposition~\ref{prop:inner-matching} holds for all $\ell' < \ell$.
Combining Lemma~\ref{lem:inner-diff} and \ref{lem:residual-inner}, we can rewrite \eqref{eq:Omega-diff} and \eqref{eq:Sigma-diff} as
\begin{align}
  [\nu^{1/4} \partial_T + A_\sperp^\sperp(0) \partial_X] P_{K,\ell}^\Omega &= \partial_X h + \h F,\label{eq:Omega-diff-simple}\\
  \partial_T P_{K,\ell}^\Sigma + \partial_1 A_1^1(0) \partial_X(\inn{\Sigma}{0} P_{K,\ell}^\Sigma) - B_1^1(0) \partial_X^2 P_{K,\ell}^\Sigma &= G + G_\ell^\match
    \label{eq:Sigma-diff-simple}
\end{align}
with
\begin{align}
  h &\lesssim_{\eps,\beta,K,\ell} \nu^{-(\ell + 1)(1/4 - \beta/3) + \eta} (\Euc \vee 1)^{-2K - 1},\label{eq:h-inner}\\
  \h F &\lesssim_{\eps,\beta,K,\ell} \nu^{-(\ell + 1)(1/4 - \beta/3) + \eta + 1/4} \tbf{1}_{\Euc \leq 1},\nonumber\\
  G &\lesssim_{\eps,\beta,K,\ell} (\Dist \vee 1)^{-2} \nu^{-(\ell + 1)(1/4 - \beta/3) + \eta} (\Euc \vee 1)^{-2K - 1}\nonumber
\end{align}
satisfying analogous derivative bounds.

We first treat the nonshocking component $P_{K,\ell}^\Omega$.
By construction, $P^\Omega|_{\Gamma_\beta} = 0$.
Because the right side of \eqref{eq:Omega-diff-simple} is small, we can integrate \eqref{eq:Omega-diff-simple} from $\Gamma_\beta$ to show that $P^\Omega$ itself is small.
For convenience, we transform \eqref{eq:nonshock} of Lemma~\ref{lem:integrals} into the inner coordinates.
\begin{lemma}
  \label{lem:integral-inner}
  For all $2 \leq I' \leq N$ and $r \in \N$, $\int (\Euc \vee 1)^{-r} \circ \gamma_{I'} \lesssim_b \nu^{-1/4} (\Euc \vee 1)^{-r + 1}$.
\end{lemma}
With this, we can easily treat the nonshocking component.
\begin{proof}[Proof of Proposition~\textnormal{\ref{prop:inner-matching}} for $\Omega$]
  Applying Lemma~\ref{lem:renorm}, we let $R \coloneqq \mathbf{R}_0[A_\sperp^\sperp(0), h]$ and write \eqref{eq:Omega-diff-simple} as
  \begin{equation*}
    [\nu^{1/4} \partial_T + A_\sperp^\sperp(0) \partial_X] (P_{K,\ell}^\Omega - R) = -\nu^{1/4} A_\sperp^\sperp(0)^{-1} \partial_t h + \h F \eqqcolon F,
  \end{equation*}
  which satisfies $F \lesssim_{\eps,\beta,K,\ell} \nu^{-(\ell + 1)(1/4 - \beta/3) + \eta + 1/4} (\Euc \vee 1)^{-2K - 2}$.
  Integrating from zero data on $\Gamma_\beta$, Lemma~\ref{lem:integral-inner} implies that
  \begin{equation*}
    P_{K,\ell}^\Omega - R \lesssim_{\eps,\beta,K,\ell} \nu^{-(\ell + 1)(1/4 - \beta/3) + \eta + 1/4} (\Euc \vee 1)^{-2K - 1}.
  \end{equation*}
  Now $R$ is a linear transformation of $h$, so \eqref{eq:h-inner} yields \eqref{eq:inner-matching} for $P_{K,\ell}^\Omega$ with $m = n = 0$.
  We inductively control $\partial_T^m \partial_X^n p_{K,\ell}^\Omega$ using the operator $\mathbf{R}_n$ from Lemma~\ref{lem:renorm} to renormalize.
  We bound the derivative data using the procedure described in Step 1 of the proof of Proposition~\ref{prop:grid-est}.
  Then the desired estimates follow from Lemma~\ref{lem:integral-inner}.
\end{proof}
For the shocking component $P^\Sigma$, we exploit the scalar nature of the advection-diffusion equation \eqref{eq:Sigma-diff-simple} to control $P^\Sigma$ via the comparison principle.
We set $P^\Sigma = 0$ below a certain time and use $G_\ell^\match$ to suppress the force near $\Gamma_\beta$.
This simplifies the analysis of higher derivatives.

Far from the origin, the shocking component is dominated by the advective derivative
\begin{equation}
  \label{eq:inner-advective}
  \bar\partial_T \coloneqq \partial_T + \partial_1A_1^1(0) \Sigma_{0,0} \partial_X.
\end{equation}
For our comparison argument, we require a lower bound on its action on $\Euc$.
\begin{lemma}
  \label{lem:unit-constant}
  We have $\abss{\bar\partial_T \Euc} \geq \Euc \Dist^{-2}$.
\end{lemma}
\begin{proof}
  Scaling to the outer coordinates, it suffices to show that $\dist^2 \euc \abss{\bar\partial_t \euc} \geq \euc^2$.
  Using \eqref{eq:fu}, we write
  \begin{equation*}
    \bar\partial_t = \partial_t + \partial_1A_1^1(0) \sigma_{0,0} \partial_x = \partial_t - a \cub \partial_x,
  \end{equation*}
  with $a$ from \eqref{eq:fu}.
  Hence $\bar\partial_t \euc^2 = \bar\partial_t(t^2 + x^2) = 2 t - 2 a x \cub.$
  The constants $a$ and $b$ in \eqref{eq:fu} have the same sign, so \eqref{eq:fu} implies that $a x \cub \geq 0$ and
  \begin{equation*}
    \euc \abss{\bar\partial_t \euc} = \frac{1}{2} \abs{\bar\partial_t \euc^2} = \abs{t} + \abss{a x \cub}.
  \end{equation*}
  Thus \eqref{eq:fu} and \eqref{eq:fm} yield
  \begin{equation*}
    \dist^2 \euc \abss{\bar\partial_t \euc} = (\abs{t} + 3 a^{-1} b \cub^2) (\abs{t} + \abss{a x \cub}) \geq t^2 + \abs{x} \abs{-a t\cub + b \cub^3} = \euc^2.
    \qedhere
  \end{equation*}
\end{proof}
We use the comparison principle to control the shocking component via the following majorant.
\begin{lemma}
  \label{lem:super}
  For all $r > 1$ and $s \geq 0$, there exists an increasing uniformly bounded Lipschitz function $\kappa \colon (-\infty, 0] \to \R_+$ such that $Q \coloneqq \kappa(T) (\Euc \vee 1)^{-r} (\Dist \vee 1)^{-3s}$ satisfies
  \begin{equation}
    \label{eq:super}
    \begin{aligned}
      \partial_T Q + \partial_1 A_1^1(0) \inn{\Sigma}{0} \partial_X Q + (s + 1) \partial_1 A_1^1(0) (\partial_X&\inn{\Sigma}{0}) Q - B_1^1(0) \partial_X^2 Q\\
      &\geq (\Euc \vee 1)^{-r} (\Dist \vee 1)^{-(3s + 2)}.
    \end{aligned}
  \end{equation}
\end{lemma}
\begin{proof}
  There exists $T_1 < 0$ such that $\Euc,\Dist \geq 1$ where $T \leq T_1$.
  First consider $T \leq T_1$.

  By Proposition~\ref{prop:Burgers-inner}, $\inn{\Sigma}{0} = \Sigma_{0,0} + \m{O}(\Dist^{-3})$.
  Using \eqref{eq:inner-advective}, we can write
  \begin{equation*}
  \begin{aligned}
    \partial_T Q + \partial_1 A_1^1(0) \inn{\Sigma}{0} \partial_X Q &+ (s + 1) \partial_1 A_1^1(0) (\partial_X\inn{\Sigma}{0}) Q
    \\ &= \bar\partial_T Q - (s + 1)\Dist^{-2} Q + \m{O}(\Dist^{-6}) Q + \m{O}(\Dist^{-3}) \partial_X Q.
    \end{aligned}
  \end{equation*}
  We can readily verify that $\partial_X Q \lesssim \Dist^{-3} Q$ and $\partial_X^2 Q \lesssim \Dist^{-6} Q$, so there exists $C > 0$ such that
  \begin{equation*}
  \begin{aligned}
    \partial_T Q + \partial_1 A_1^1(0) \inn{\Sigma}{0} \partial_X Q &+ (s + 1) \partial_1 A_1^1(0) (\partial_X\inn{\Sigma}{0}) Q - B_1^1(0) \partial_X^2 Q
    \\ &\geq \bar\partial_T Q - (s + 1)\Dist^{-2} Q - C \abs{T}^{-3} Q.
    \end{aligned}
  \end{equation*}
  Recall that $\Dist^2 = \abs{T} + 3 a^{-1} b \Sigma_{0,0}^2$ and $\bar\partial_T \Sigma_{0,0} = 0$, so $\bar\partial_T \Dist^2 = - 1$.
  Multiplying by the integrating factor $\Dist^{2(s + 1)}$, we therefore wish to arrange
  \begin{equation}
    \label{eq:super-multiplied}
    \bar\partial_T\big(\Dist^{2(s + 1)} Q\big) - C\abs{T}^{-3} \Dist^{2(s + 1)} Q \geq \Euc^{-r} \Dist^{-s}.
  \end{equation}
  Let $\kappa(T) \coloneqq K \exp\big(\frac{C}{2\abs{T}^2}\big)$ for $T < T_1$ and $K > 0$ to be determined, so $\dot \kappa = C \abs{T}^{-3} \kappa$ and
  \begin{equation*}
  \begin{aligned}
    \bar\partial_T\big(\Dist^{2(s + 1)} Q\big) - C\abs{T}^{-3} \Dist^{2(s + 1)} Q &= \kappa \bar{\partial}_T(\Euc^{-r} \Dist^{2-s})
    \\ &= r\kappa \Euc^{-r-1} \Dist^{2-s} \abs{\partial_T \Euc} - \frac{\kappa}{2} (2 -  s) \Euc^{-r} \Dist^{-s}.
    \end{aligned}
  \end{equation*}
  Using Lemma~\ref{lem:unit-constant}, and $s \geq 0$, we find
  \begin{equation*}
    \bar\partial_T\big(\Dist^{2(s + 1)} Q\big) - C\abs{T}^{-3} \Dist^{2(s + 1)} Q \geq \kappa(r-1) \Euc^{-r} \Dist^{-s}.
  \end{equation*}
  Since $r > 1$, \eqref{eq:super-multiplied} follows if we choose $K \coloneqq (r - 1)^{-1}$.
  \smallskip

  Now consider $T \in [T_1, 0]$.
  We can readily check that
  \begin{equation*}
    \abss{\partial_X(\inn{\Sigma}{0} Q)} + \abss{\partial_X^2 Q} + (\Euc \vee 1)^{-r} (\Dist \vee 1)^{-(3s + 2)} \lesssim Q,
  \end{equation*}
  Hence a multiplier of the form $\kappa(T) = \kappa(T_1) \e^{\Lambda(T - T_1)}$ with $\Lambda \gg 1$ ensures \eqref{eq:super}.
\end{proof}
We can finally construct the shocking component of the inner expansion.
\begin{proof}[Proof of Proposition~\textnormal{\ref{prop:inner-matching}} for $\Sigma$]
  Recall the cutoff $\theta$ from the beginning of the section, which is $0$ on $[0, 1/2]$ and $1$ on $[1, \infty)$.
  Let $\theta_\beta \coloneqq \theta \circ \big(\mathsf{R}^{-1} \nu^{1/2-\beta} \Euc\big)$, which is $1$ where $\Euc \geq \mathsf{R} \nu^{-(1/2 - \beta)}$ and $0$ where $\Euc \leq \tfrac{\mathsf{R}}{2} \nu^{-(1/2-\beta)}$.
  Now fix $K,\ell \in \N$.
  Recalling \eqref{eq:Sigma-diff-simple}, let $G_\ell^\match \coloneqq -\theta_\beta G$.
  Then $P_{K,\ell}^\Sigma$ solves
  \begin{equation*}
      \partial_T P_{K,\ell}^\Sigma + \partial_1 A_1^1(0) \partial_X(\inn{\Sigma}{0} P_{K,\ell}^\Sigma) - B_1^1(0) \partial_X^2 P_{K,\ell}^\Sigma = (1 - \theta_\beta)G.
  \end{equation*}
  The forcing $(1 - \theta_\beta)G$ vanishes where $\Euc \geq \mathsf{R} \nu^{-(1/2 - \beta)}$, and in particular where $T \leq - \mathsf{R} \nu^{-(1/2 - \beta)}$.
  We let $P_{K,\ell}^\Sigma$ be identically zero at  such times.
  We can interpret this as a Cauchy problem on the line with zero data at $T = - \mathsf{R} \nu^{-(1/2 - \beta)}$.
  We observe that $\theta_\beta$ vanishes where $\Euc \leq \tfrac{\mathsf{R}}{2} \nu^{-(1/2-\beta)}$, or equivalently $\euc \leq \tfrac{\mathsf{R}}{2} \nu^\beta$.
  So $G_\ell^\match$ satisfies the support condition in Proposition~\ref{prop:inner-matching}.

  We induct on the number of derivatives.
  Fix $m, n \in \N_0$ and suppose we have shown Proposition~\ref{prop:inner-matching} for all $(m', n')$ with $n' < n$ or $n' = n$ and $m' < m$.
  Let $Z \coloneqq \partial_T^m\partial_X^n P_{K,\ell}^\Sigma$.
  When we apply $\partial_T^m\partial_X^n$ to \eqref{eq:Sigma-diff-simple}, we must expand the product in the advection and commute derivatives with $\theta_\beta$.
  The product terms have the form $\partial_X^{n + 1 - n'} \inn{\Sigma}{0} \cdot \partial_X^{n'}P_{K,\ell}^\Sigma$ for $n' \leq n + 1$.
  When $n' < n$, we can use the inductive hypothesis and \eqref{eq:Burgers-inner} to fold these into the forcing.
  For the commutators, we observe that $\partial_T \theta_\beta \lesssim \nu^{1/2 - \beta} \lesssim \Euc^{-1}$ and $\partial_X \theta_\beta \lesssim \nu^{3/4 - \beta} \lesssim \Dist^{-3}$, as both are supported where $\Euc \asymp \nu^{-(1/2 - \beta)}$ and hence $\Dist^3 \lesssim \nu^{-1/4} \Euc \lesssim \nu^{-(3/4 - \beta)}$.
  That is, these commutators contribute the proper factors of $\Euc^{-1}$ and $\Dist^{-3}$.
  In sum, we obtain
  \begin{equation*}
    \partial_T Z + \partial_1 A_1^1(0) \inn{\Sigma}{0} \partial_X Z + (n + 1) \partial_1 A_1^1(0) (\partial_X \inn{\Sigma}{0}) Z - B_1^1(0) \partial_X^2 Z = H
  \end{equation*}
  for $H \lesssim_{\eps,\beta,K,\ell,m,n} \nu^{-(\ell + 1)(1/4 - \beta/3) + \eta} (\Euc \vee 1)^{-2K - m-1} (\Dist \vee 1)^{-3n - 2}$.
  By Lemma~\ref{lem:super} and linearity, there exists $0 < Q \lesssim_{\eps,\beta,K,\ell,m,n} \nu^{-(\ell + 1)(1/4 - \beta/3) + \eta} (\Euc \vee 1)^{-2K - m-1} (\Dist \vee 1)^{-3n}$ such that
  \begin{equation*}
    \partial_T Q + \partial_1 A_1^1(0) \inn{\Sigma}{0} \partial_X Q + (n + 1) \partial_1 A_1^1(0) (\partial_X \inn{\Sigma}{0}) Q - B_1^1(0) \partial_X^2 Q \geq H.
  \end{equation*}
  Because $Z(-\mathsf{R} \nu^{-(1/2 - \beta)}, \anon) \equiv 0$, the comparison principle on $[-\mathsf{R} \nu^{-(1/2 - \beta)}, 0] \times \R$ implies that $Z \leq Q$.
  Replacing $Q$ by $-Q$, a symmetric argument yields $Z \geq -Q$, so $\abs{Z} \leq Q$.
  The proposition follows from induction.
\end{proof}
We next express these estimates in outer coordinates.
\begin{proof}[Proof of Corollary~\textnormal{\ref{cor:worst-inner-size}}]
  Converting \eqref{eq:inner-matching} to outer coordinates, we find
  \begin{equation*}
    \absb{\partial_t^m \partial_x^n(\inn{\psi}{\ell} - \cut{\psi}{[K],\ell})} \lesssim_{\eps,\beta,K,\ell,m,n} \nu^{K + 1} \nu^{(\ell + 1)\beta/3 - 2\beta/3 - \eps \beta} (\euc \vee \nu^{1/2})^{-2K - m - 1} (\dist \vee \nu^{1/4})^{-3n}.
  \end{equation*}
  This implies \eqref{eq:cut-diff}.
  Proposition~\ref{prop:inner-matching} and Lemma~\ref{lem:inner-partial} imply that $\cut{\Psi}{[K],\ell}$ dominates $\inn{\Psi}{\ell}$, so \eqref{eq:inner-partial} implies \eqref{eq:inner-bounds}.

  Next, in $\m{I}_\beta \setminus \{\tfrac{\mathsf{R}}{2} \nu^\beta\}$ we have $\euc^{-1} \asymp \nu^{-\beta}$ and $\dist^{-1} \lesssim \euc^{-1/2} \asymp \nu^{-\beta/2}$.
  We can thus simplify the previous display to
  \begin{align*}
    \absb{\partial_t^m \partial_x^n(\inn{\psi}{\ell} - \cut{\psi}{[K],\ell})} &\lesssim_{\eps,\beta,K,\ell,m,n} \nu^{K + 1} \nu^{(\ell + 1)\beta/3 - 2\beta/3 - \eps \beta} \nu^{-2\beta K - m\beta - \beta} \nu^{-3n\beta/2},\\
                                                                              &\lesssim_{\eps,\beta,K,\ell,m,n} \nu^{(1 - 2\beta)K} \nu^{(\ell + 1)\beta} \nu^{1 - (2/3 + \eps + m + 1 + 3n/2)\beta}.
  \end{align*}
  Since $\eps \leq 1/6$ and $\beta < 1/2$,
  \begin{equation*}
    1 - (2/3 + \eps + m + 1 + 3n/2)\beta > 1 - (m + 3n/2 + 11/6)/2 > -m/2 - 3n/4
  \end{equation*}
  and \eqref{eq:inner-matching-outer-form} follows.
\end{proof}
Finally, we show that the partial outer sum $\out{\psi}{[K]}$ matches the partial inner sum $\inn{\psi}{[L]} \coloneqq \sum_{\ell = 0}^L \inn{\psi}{\ell}$ to high order in the matching zone $\euc \asymp \nu^\beta$ when $K,L \gg 1$.
\begin{proposition}
  \label{prop:outer-inner-matching}
  For all $K,L \in \N$ and $n \in \N_0$,
  \begin{equation}
    \label{eq:outer-inner-matching}
    \abss{\partial_x^n(\out{\psi}{[K]} - \inn{\psi}{[L]})} \lesssim_{K,L,n} (\nu^{(1 - 2 \beta) K} + \nu^{L\beta/3 + 1/4}) \nu^{-3n/4}
  \end{equation}
  where $\frac{\mathsf{R}}{4} \nu^\beta \leq \euc \leq \frac{\mathsf{R}}{2} \nu^\beta.$
\end{proposition}
\begin{proof}
  First consider $n = 0$.
  Writing
  \begin{equation*}
    \out{\psi}{K} - \inn{\psi}{L} = (\out{\psi}{[K]} - \psi_{[K,L]}) - (\inn{\psi}{[L]} - \psi_{[K,L]}),
  \end{equation*}
  we reduce to matching between the outer and inner expansions and the grid.
  Using Propositions~\ref{prop:inviscid-row} and \ref{prop:outer-matching}, the outer difference is dominated by
  \begin{equation*}
    \out{\psi}{0} - \psi_{0,[L]} \lesssim_L \abs{\log \nu}^L \nu^{(L + 2)\beta/3}.
  \end{equation*}
  By Propositions~\ref{prop:inner-matching} and \ref{prop:Burgers-inner}, the inner difference is dominated by
  \begin{equation*}
    \inn{\psi}{0} - \psi_{[K],0} \lesssim_K \nu^{(1 - 2\beta)(K + 1)}.
  \end{equation*}
  Combining and simplifying these expressions, we find \eqref{eq:outer-inner-matching}.
  Derivative estimates follow because $\dist^{-3} \lesssim \nu^{-3\beta/2} < \nu^{-3/4}$ in this region.
\end{proof}

\subsection{Approximate inner solution}

To close this section, we show that the inner partial sum is a good approximate solution of \eqref{eq:main} where $\euc \leq \tfrac{\mathsf{R}}{2} \nu^\beta$ provided $L \gg 1$.
\begin{proposition}
  \label{prop:inner-approx}
  Fix $\beta \in (6/13, 1/2)$.
  For all $L \in \N$,
  \begin{equation*}
    \partial_t \inn{\psi}{[L]} + A(\inn{\psi}{[L]}) \partial_x \inn{\psi}{[L]} = \nu \partial_x\big[B(\inn{\psi}{[L]} \partial_x \inn{\psi}{[L]})] + \inn{F}{L} \quad \text{in } \euc \leq \frac{\mathsf{R}}{2} \nu^\beta
  \end{equation*}
  for a force $\inn{F}{L}$ satisfying $\abss{\partial_x^n \inn{F}{L}} \lesssim_{L,n} \nu^{L\beta/3 - 3n/4 - 1/2}$ for all $n \in \N_0$.
\end{proposition}
\begin{proof}
  We begin in the inner coordinates by substituting $\inn{\Psi}{[L]} = \sum_{\ell=0}^L \nu^{\ell/4} \inn{\Psi}{\ell}$ in \eqref{eq:inner} and controlling the residual error.

  Propositions~\ref{prop:inner-matching} and \ref{prop:Burgers-inner} imply that $\inn{\Psi}{[L]}$ is dominated by its first term:
  \begin{equation*}
    \inn{\Psi}{[L]} \lesssim_L \sup_{\{\euc \leq \mathsf{R}\nu^\beta/2\}} \abss{\inn{\Psi}{0}} \lesssim \sup_{\{\euc \leq \mathsf{R}\nu^\beta/2\}} \abs{\Dist} \lesssim \nu^{-(1/4 - \beta/3)}.
  \end{equation*}
  We now write $A(\nu^{1/4} \inn{\Psi}{[L]}) \partial_X \inn{\Psi}{[L]} = \nu^{-1/4} \partial_X[f(\nu^{1/4}\inn{\Psi}{[L]})]$.
  Expanding the flux about $0$, we find
  \begin{equation*}
    \nu^{-1/4} [f(\nu^{1/4}\inn{\Psi}{[L]}) - f(0)] = \nu^{-1/4}\sum_{r = 0}^{L + 1} \frac{\Der^rA(0)}{(r + 1)!} (\nu^{1/4}\inn{\Psi}{[L]})^{r + 1} + \m{O}_L(\nu^{(L + 3)\beta/3-1/4}).
  \end{equation*}
  Moreover,
  \begin{equation}
    \label{eq:power-decomp-inner}
    \nu^{-1/4}\sum_{r = 0}^{L + 1} \frac{\Der^rA(0)}{(r + 1)!} (\nu^{1/4}\inn{\Psi}{[L]})^{r + 1} = \sum_\bigsub{\al_j \leq L\\ \#\al + \abs{\al} \leq L + 1} \nu^{(\# \al + \abs{\al})/4} \m{A} \inn{\Psi}{\al} + \sum_\bigsub{\al_j \leq L\\ \#\al \leq L + 1\\ \#\al + \abs{\al} > L + 1} \nu^{(\# \al + \abs{\al})/4} \m{A} \inn{\Psi}{\al}.
  \end{equation}
  Now Proposition~\ref{prop:inner-matching} implies that $\inn{\Psi}{\ell} \lesssim_\ell \nu^{-(\ell + 1)(1/4 - \beta/3)}$, so
  \begin{equation*}
    \nu^{(\# \al + \abs{\al})/4} \m{A} \inn{\Psi}{\al} \lesssim_\al \nu^{(\# \al + \abs{\al})\beta/3 + \beta/3 - 1/4}.
  \end{equation*}
  In the second sum on the right in \eqref{eq:power-decomp-inner}, $\# \al + \abs{\al} \geq L+2$, so
  \begin{equation*}
    \sum_\bigsub{\al_j \leq L\\ \#\al \leq L + 1\\ \#\al + \abs{\al} > L + 1} \nu^{(\# \al + \abs{\al})/4} \m{A} \inn{\Psi}{\al} \lesssim_L \nu^{(L + 2)\beta/3 + \beta/3 - 1/4} \lesssim_L \nu^{L\beta/3 + \beta - 1/4}.
  \end{equation*}
  Finally, the first sum on the right in \eqref{eq:power-decomp-inner} is nearly equal to the sums of the $\m{A}$ terms in \eqref{eq:Omega-inner} and \eqref{eq:Sigma-inner}, suitably weighted by factors of $\nu^{1/4}$.
  It only includes one extra term in the non-shocking component of the form $\nu^{(\# \al + \abs{\al})/4} \m{A}_\sperp \inn{\Psi}{\al}$ with $\# \al = L + 1$ and $\al = (0, \ldots, 0)$.
  This term satisfies
  \begin{equation*}
    \nu^{(\# \al + \abs{\al})/4} \m{A}_\sperp \inn{\Psi}{\al} \lesssim_L \nu^{(L + 1)\beta/3 + \beta/3 - 1/4} \lesssim \nu^{L\beta/3}.
  \end{equation*}
  This dominates the earlier errors.
  Accounting for the additional $\nu^{-1/4}\partial_X$ in \eqref{eq:inner}, which contributes $\nu^{-1/4}$, we see that the advective residue from $\inn{\Psi}{[L]}$ in \eqref{eq:inner} is controlled by $\nu^{L\beta/3 - 1/4}$.
  A similar analysis of the diffusive residue due to $B$ yields error of order $\nu^{L\beta/3 + \beta/3 - 1/4} \lesssim \nu^{L\beta/3 - 1/4}$.
  So the residue in the inner coordinates is bounded by $\nu^{L\beta/3 - 1/4}$.
  Using $\partial_T = \nu^{1/2} \partial_t$ and $\Psi = \nu^{-1/4}\psi$, this corresponds to $\inn{F}{L} \lesssim_L \nu^{L\beta/3 - 1/2}$ in the outer coordinates.
  Derivative estimates follow as usual.
\end{proof}

\section{The approximate solution and nonlinear closure}
\label{sec:closure}

We have shown how to construct $\out{\psi}{[K]}$ where $\euc \gtrsim \nu^\beta$ and $\inn{\psi}{[L]}$ where $\euc \lesssim \nu^\beta$ that very nearly solve \eqref{eq:main} on their respective domains and agree to high order on the overlap $\{\euc \asymp \nu^\beta\}$.
Interpolating between these, we now form a highly accurate approximate solution $\app{\psi}{K,L}$ on the entire domain $(t_0, 0) \times \R$.
Drawing on local well-posedness theory, we show that the true solution $\psi^{(\nu)}$ remains arbitrarily close to $\app{\psi}{K,L}$ in any Sobolev norm provided $K,L$ are sufficiently large.

We have defined the outer and inner partial sums
\begin{equation*}
  \out{\psi}{[K]} = \sum_{k=0}^K \nu^k \out{\psi}{k} \And \inn{\psi}{[L]} = \sum_{\ell=0}^L \inn{\psi}{\ell}.
\end{equation*}
Recalling \eqref{eq:fit} and the cutoff $\theta$ from Section~\ref{sec:inner-matching}, define $\zeta \coloneqq \theta\big(\tfrac{2\euc}{\mathsf{R} \nu^\beta} \big)$.
Then $\zeta \equiv 1$ where $\euc \geq \tfrac{\mathsf{R}}{2} \nu^\beta$ and $\zeta \equiv 0$ where $\euc \leq \tfrac{\mathsf{R}}{4} \nu^\beta$.
Given $K,L \in \N$, we define
\begin{equation}
  \label{eq:def-approx}
  \app{\psi}{K,L} \coloneqq \zeta \out{\psi}{[K]} + (1 - \zeta) \inn{\psi}{[L]}.
\end{equation}
Note that $1 - \zeta$ is supported where $\euc \leq \tfrac{\mathsf{R}}{2} \nu^\beta$.
By Proposition~\ref{prop:inner-matching}, $\inn{\Sigma}{\ell}$ solves \eqref{eq:Sigma-inner} exactly in this region.
This underpins the precise choice of the constant in the cutoff $\zeta$.
Before proceeding, we record a simple bound on $\app{\psi}{}$.
\begin{lemma}
  \label{lem:app-size}
  Fix $\beta \in (6/13, 1/2)$.
  For all $K,L \in \N$ and $m,n \in \N_0$,
  \begin{equation*}
    \absb{\partial_t^m \partial_x^n \app{\psi}{[K,L]}} \lesssim_{K,L,m,n} (\dist \vee \nu^{1/4})^{-3n + 1} (\euc \vee \nu^{1/2})^{-m} \quad \text{in } [t_0, 0] \times [-1, 1].
  \end{equation*}
\end{lemma}
\begin{proof}
  Proposition~\ref{prop:outersize} and Corollary~\ref{cor:worst-inner-size} establish this bound for $\out{\psi}{[K]}$ and $\inn{\psi}{[L]}$, respectively.
  Because $\abs{\Der \zeta} \lesssim \euc^{-1}$, derivatives landing on the cutoff do not disrupt the estimate.
\end{proof}
We wish to show that $\psi^{(\nu)} \approx \app{\psi}{K,L}$ when $K, L \gg 1$ and $\nu \ll 1$.
To see this, we must treat the full nonlinear system \eqref{eq:main}.
This is not trivial because we permit the diffusion matrix $B$ to be degenerate.
In this situation, it is necessary to impose additional structural assumptions to ensure local well-posedness.
These conditions go back to foundational work of Kawashima and Shizuta \mbox{\cite{Kawashima,ShiKaw85}}.
Here, we work in a related framework of Serre~\cite{Serre10a,Serre10b}.

In \ref{hyp:advection}, we assume that there is a smooth entropy-flux pair $(\eta,q)$ such that $\eta$ is strictly convex.
That is, $\eta,q \colon \m{V} \to \R$, $\nab \eta = (\nab q)^\top A$, and $\Der^2 \eta > 0$.
(In fact we only need this pair in a connected neighborhood containing the range of $\psi^{(\nu)}$; we suppress this point below.)
We interpret the entropy gradient $\nab \eta \colon \m{V} \to \m{V}$ as a change of unknown.
Convexity ensures that its Jacobian does not vanish, so it is a diffeomorphism.

Next, in \ref{hyp:diffusion} we assume that the entropy is \emph{strongly dissipated} by the diffusion $B$.
That is, if we interpret $\Der^2\eta$ as a bilinear form on $\m{V}$,
\begin{equation}
  \label{eq:entropy-dissipation2}
  \Der^2\eta\big(\xi, B \xi\big) \gtrsim \abs{B \xi}^2 \ForAll \xi \in \m{V}.
\end{equation}
Informally, $\eta$ experiences dissipation whenever $\psi$ does.

Finally, we assume that $B$ has constant rank $N - s$ for some $s \in [N]$ and there exists a fixed basis of $\m{V}$ (different from that used above) in which the first $s$ rows of $B$ vanish identically.
For example, in the full Navier--Stokes system \eqref{eq:NSE}, there is dissipation in momentum and energy but not mass, so $s = 1$.
We work in this basis in the remainder of the section, and thus set aside the distinction between shocking and non-shocking components.
As shown in~\cite{Serre10a}, these hypotheses generalize and simplify the classical Kawashima--Shizuta conditions~\cite{Kawashima,ShiKaw85} for local well-posedness.
They similarly yield local well-posedness~\cite{Serre10b}, and we use related \emph{a priori} estimates to show that $\app{\psi}{[K,L]} \approx \psi^{(\nu)}$.

In combination, \eqref{eq:entropy-dissipation2} and the constant range of $B$ imply that \eqref{eq:main} becomes more favorable if we transform the range of $B$ under the entropic diffeomorphism $\nab \eta$.
To this end, let $\Pi$ denote projection onto the last $N-s$ components in the basis fixed above, and let $\iota \colon \R^{N-s} \hookrightarrow \m{V}$ denote the corresponding embedding.
We define a diffeomorphism $\Eta \colon \m{V} \to \m{V}$ by $\Eta \psi \coloneqq (1 - \Pi) \psi + \Pi \nab \eta$.
This preserves the components of $\psi$ that do not directly experience diffusion and entropically transforms those that do.
Then let
\begin{equation*}
  \varphi^{(\nu)} \coloneqq \Eta \psi^{(\nu)}, \quad \app{\varphi}{[K,L]} \coloneqq \Eta \app{\psi}{[K,L]}, \And v \coloneqq \varphi^{(\nu)} - \app{\varphi}{[K,L]}.
\end{equation*}
We will show that $v \ll 1$ provided $K,L \gg 1$.
Because $\Eta$ is a diffeomorphism, it will follow that $\psi^{(\nu)} - \app{\psi}{[K,L]} \asymp v \ll 1$.

To begin, we show that $v$ solves an equation with very small forcing.
We first focus on the equation for $\varphi^{(\nu)}$.
In~\cite{Serre10b}, Serre shows:
\begin{proposition}
  \label{prop:diffeo-eq}
  There exist smooth matrix-valued functions $S_0, S_1 \colon \m{V} \to \R^N \otimes \R^N$ and ${Z \colon \m{V} \to \R^{N - s} \otimes \R^{N-s}}$ such that $S_0$ and $S_1$ are symmetric, $S_0$ is positive-definite and $(s, N - s)$-block diagonal, $\braket{Z \xi, \xi} \gtrsim \abs{\xi}^2$ for all $\xi \in \R^{N-s}$, and
  \begin{equation}
    \label{eq:diffeo-eq}
    S_0\big(\varphi^{(\nu)}\big) \partial_t \varphi^{(\nu)} + S_1\big(\varphi^{(\nu)}\big) \partial_x \varphi^{(\nu)} = \nu \iota \partial_x Z\big(\varphi^{(\nu)}\big) \partial_x \Pi \varphi^{(\nu)}.
  \end{equation}
\end{proposition}
The change of unknown $\Eta$ is advantageous because it fixes the \emph{kernel} of the diffusion in addition to its range.
As a result, the diffusion becomes block diagonal and coercive in the last $N - s$ components.
This prevents derivatives of $Z$ from ``leaking'' into its kernel, where they cannot be controlled.
The price we pay for the change of variables is the nonlinear matrix $S_0$ acting on $\partial_t$.
Nonetheless, because $\eta$ is a convex entropy for the hyperbolic system, we can arrange that $S_0$ is PSD and $S_1$ is symmetric.
These properties enable our energy estimates.

To control $v$, we make use of the following $W^{1,\infty}$-bootstrap.
Given $\tb \in (t_0, 0]$, we say $v$ satisfies the bootstrap on $[t_0,\tb]$ if
\begin{equation}
  \label{eq:bootstrap}
  \norm{v}_{L_x^\infty}(t) + \norm{\Der v}_{L_x^\infty}(t) \leq 1 \ForAll t \in [t_0, \tb].
\end{equation}
In particular, the functions $\varphi^{(\nu)}$ and $\app{\varphi}{[K,L]}$ map to a compact subset of $\m{V}$, which may depend on $\beta$, $K$, and $L$ but not on $\nu$.

With this bootstrap, we can check that $\app{\varphi}{[K,L]}$ approximately solves \eqref{eq:diffeo-eq}.
\begin{proposition}
  \label{prop:diffeo-approx}
  Fix $\beta \in (6/13, 1/2)$ and let $S_0, S_1$, and $Z$ be as in Proposition~\textnormal{\ref{prop:diffeo-eq}}.
  Then for all $K,L \in \N$, there is a force $F$ such that
  \begin{equation}
    \label{eq:diffeo-approx}
    S_0\big(\app{\varphi}{[K,L]}\big) \partial_t \app{\varphi}{[K,L]} + S_1\big(\app{\varphi}{[K,L]}\big) \partial_x \app{\varphi}{[K,L]} = \nu \iota \partial_x Z\big(\app{\varphi}{[K,L]}\big) \partial_x  \Pi \app{\varphi}{[K,L]} + F
  \end{equation}
  and $\partial_x^n F \lesssim_{K,L,n} [\nu^{(1 - 2\beta)K - 3/4} + \nu^{L\beta/3 - 1/2}] \nu^{-3n/4}$ on $[t_0, \tb] \times \R$ for all $n \in \N_0$.
\end{proposition}
\begin{proof}
  In~\cite{Serre10b}, the transformation from \eqref{eq:main} to \eqref{eq:diffeo-eq} consists of multiplication by a smooth matrix-valued function $S$ of the solution.
  By \eqref{eq:bootstrap}, $S$ is uniformly bounded.
  It thus suffices to show that
  \begin{equation}
    \label{eq:approx}
    \partial_t \app{\psi}{[K,L]} + A(\app{\psi}{[K,L]}) \partial_x \app{\psi}{[K,L]} = \nu \partial_x[B(\app{\psi}{[K,L]})\partial_x \app{\psi}{[K,L]}] + \h{F}
  \end{equation}
  with $\partial_x^n \h F \lesssim_{K,L,n} \nu^{(1 - 2\beta)K - 1/4 - 3n/4}$.
  This implies \eqref{eq:diffeo-approx} for $n = 0$, and Lemma~\ref{lem:app-size} implies that spatial derivatives landing on $S(\app{\psi}{[K,L]})$ do not alter the estimate.

  We analyze \eqref{eq:approx} separately in the outer, inner, and matching regions, which roughly correspond to $\{\euc \gtrsim \nu^\beta\}$, $\{\euc \lesssim \nu^\beta\}$, and $\{\euc \asymp \nu^\beta\}$.
  In the outer region, $\app{\psi}{[K,L]} = \out{\psi}{[K]}$, so by Proposition~\ref{prop:outer-approx}, $\h{F} = \out{F}{K}$ is suitably small.
  In the inner region, $\app{\psi}{[K,L]} = \inn{\psi}{[L]}$ and the conclusion likewise follows from Proposition~\ref{prop:inner-approx}.

  In the matching region, we perturb $\out{\psi}{[K]}$ by the small quantity $p \coloneqq \zeta (\inn{\psi}{[L]} - \out{\psi}{[K]})$.
  We can thus decompose $\h{F}$ into $\out{F}{K}$ plus terms arising when derivatives land on $p$.
  Because $\partial_t \zeta \lesssim \euc^{-1}$ and $\partial_x \zeta \lesssim \euc^{-1} \zeta \lesssim \dist^{-3} \zeta$, $p$ satisfies the usual derivative estimates.
  Collecting $p$-terms in the matching residue and using Proposition~\ref{prop:outer-inner-matching}, the largest is
  \begin{equation*}
    \partial_x p \lesssim_{K, L} \nu^{(1 - 2\beta)K - 3/4} + \nu^{L\beta/3 - 1/2}.
  \end{equation*}
  Derivative estimates follow as usual.
\end{proof}
For the sake of brevity, we write $S_0^{(\nu)}$ for $S_0(\varphi^{(\nu)})$, $\app{S}{0}$ for $S_0(\app{\varphi}{[K,L]})$, and similarly for $S_1$ and $Z$.
Taking the difference between \eqref{eq:diffeo-eq} and \eqref{eq:diffeo-approx}, we obtain:
\begin{corollary}
  Fix $\beta \in (6/13, 1/2)$ and let $S_0,S_1,Z,$ and $F$ be as in Propositions~\ref{prop:diffeo-eq} and \ref{prop:diffeo-approx}.
  Then there exist smooth 3-tensor-valued functions $Q_0,Q_1,$ and $Q_Z$ such that $v(t_0, \anon) \equiv 0$ and
  \begin{align}
    \label{eq:diff}
      \big[S_0^{(\nu)} \partial_t + S_1^{(\nu)} \partial_x\big] v& + \big[Q_0(\app{\varphi}{[K,L]}, \varphi^{(\nu)}) (\partial_t \app{\varphi}{[K,L]}) + Q_1(\app{\varphi}{[K,L]}, \varphi^{(\nu)}) (\partial_x \app{\varphi}{[K,L]})\big]v\nonumber\\
                                                          &= \nu \iota \partial_x \big[Z^{(\nu)} \partial_x \Pi v + Q_Z(\app{\varphi}{[K,L]}, \varphi^{(\nu)}) (\partial_x \Pi \app{\varphi}{[K,L]}) v\big] + F.
  \end{align}
\end{corollary}
\noindent
Here $Q_0$, $Q_1$, and $Q_Z$ represent suitable difference quotients of $S_0, S_1$, and $Z$.
By \eqref{eq:bootstrap}, these functions are uniformly smooth in our domain.

We now prove an energy estimate for the difference equation \eqref{eq:diff}.
\begin{lemma}
  \label{lem:energy}
  For each $n \in \N_0$, there exists  $C(n) > 0$ such that for all $t \in [t_0, \tb]$,
  \begin{equation*}
    \norms{\partial_x^n v}_{L_x^2}\!(t) \lesssim_{K,L,n} \e^{C \int_{t_0}^t \big(\norms{\Der \app{\varphi}{[K,L]}}_{L_x^\infty} + \norm{\Der v}_{L_x^\infty} + \nu \norms{\partial_x \app{\varphi}{[K,L]}}_{L_x^\infty}^2 + \nu \norms{\partial_x v}_{L_x^\infty}^2\big)} \sup_{[t_0, t]}\norm{F_n}_{L_x^2},
  \end{equation*}
  for $F_n^2 \lesssim (\partial_x^n F)^2 + \sum_{n'=0}^{n-1} [H_{n,n'} \partial_x^{n'} v]^2$, where $H_{n,n'}$ is a function of polynomial growth in $\app{\varphi}{[K,L]},\ldots,\partial_x^{n+2}\app{\varphi}{[K,L]}$ and $v,\ldots,\partial_x^{n'-1}v$.
\end{lemma}
\noindent
In fact, given that we assume \eqref{eq:bootstrap}, there is no need for $\norm{\Der v}$ in the above exponent, as it is bounded by $1$.
We include it to highlight the structure of the argument and the importance of the bootstrap.
\begin{proof}
  This essentially shown in Section~3.1 of \cite{Serre10b}.
  Here we sketch details salient to our context.
  We apply $\partial_x^n$ to \eqref{eq:diff}, take its inner product in $\m{V}$ against $\partial_x^n v$, and integrate in space.
  Applications of the product rule, integration by parts, and Young's inequality produce ``forcing'' terms that only involve derivatives of $v$ of order $n-1$ or less.
  We group these in $F_n$.
  Because $v$ appears in all but the last term of \eqref{eq:diff}, the only term in $F_n$ with no factor of $\partial_x^{n'}v$ is $\partial_x^n F$.
  This explains the functional form of $F_n$ in the lemma.

  We now explain each term in the Gr\"onwall exponent.
  The difference equation \eqref{eq:diff} includes the zeroth-order term $\big(Q_0 \partial_t \app{\varphi}{[K,L]} + Q_1\partial_x \app{\varphi}{[K,L]}\big)v$, which yields $\norms{\Der \app{\varphi}{[K,L]}}_{L_x^\infty}$ in Gr\"{o}nwall.
  When $n \geq 1$, derivatives can land on $S_i$, producing terms such as $\Der S_1 \partial_x v \partial_x^n v$.
  These contribute $\norms{\Der v}_{L_x^\infty}$ in Gr\"onwall.

  The quadratic terms in the exponent arise from the diffusion, and thus include factors of $\nu$.
  To treat the diffusion in our energy estimate, we observe that it is a total derivative.
  Integrating by parts, we obtain
  \begin{equation*}
    -\nu \int_{\R} \partial_x^{n+1} (\Pi v) \cdot \partial_x^n \big[Z^{(\nu)} \partial_x \Pi v + Q_Z \partial_x (\Pi \app{\varphi}{[K,L]}) v\big] \d x.
  \end{equation*}
  When $\partial_x^n$ lands on $\partial_x \Pi v$ in the first term, we obtain $-\nu \int \partial_x^{n+1} (\Pi v) \cdot Z^{(\nu)} \partial_x^{n+1} (\Pi v)$.
  Because $Z^{(\nu)}$ is positive on the range of $\Pi$, this is equivalent to $-\nu \norms{\partial_x^{n+1} \Pi v}_{L_x^2}^2$.
  We use this favorable term to absorb lower-order contributions.

  For example, if $\partial_x^n$ lands on $v$ in the second term, Young's inequality yields
  \begin{equation*}
    \nu \partial_x^{n+1} (\Pi v) Q_Z \partial_x (\Pi \app{\varphi}{[K,L]}) \partial_x^n v \leq \nu \delta (\partial_x^{n+1} \Pi v)^2 + \nu \delta^{-1} Q_Z^2 (\partial_x \Pi \app{\varphi}{[K,L]})^2 (\partial_x^n v)^2.
  \end{equation*}
  Choosing $\delta > 0$ small relative to the ellipticity constant of $Z^{(\nu)}$, we can absorb the former into our favorable term.
  The latter is controlled by $\nu \norms{\partial_x \app{\varphi}{[K,L]}}_{L_x^\infty}^2 (\partial_x^n v)^2$, which contributes $\nu \norms{\partial_x \app{\varphi}{[K,L]}}_{L_x^\infty}^2$ in Gr\"onwall.

  If $\partial_x^n$ instead lands on $Z^{(\nu)}$, we obtain a term of the form
  \begin{equation*}
    \nu \partial_x^{n+1} (\Pi v) \Der Z \partial_x(\Pi v) \partial_x^nv.
  \end{equation*}
  An identical computation shows that we can coercively absorb the higher-order part, and the lower-order part contributes $\nu \norms{\partial_x v}_{L_x^\infty}^2$ to Gr\"onwall.

  Finally, it is straightforward to check that all terms in the calculation can be controlled by one of the above approaches.
\end{proof}
To apply this energy estimate, we require bounds on $\Der \app{\varphi}{[K,L]}$.
\begin{lemma}
  \label{lem:der-app-soln}
  For all $K,L \in \N_0$ and $t \in [t_0, 0]$,
  \begin{equation*}
    \norms{\Der \app{\varphi}{[K,L]}}_{L_x^\infty}(t) \lesssim_{K,L} (\abs{t} \vee \nu^{1/2})^{-1}.
  \end{equation*}
\end{lemma}
\begin{proof}
  First consider the outer sum $\out{\psi}{[K]} = \sum_{k=0}^K \nu^k \out{\psi}{k}$ where $\euc \gtrsim \nu^\beta$.
  By \eqref{eq:out-bound}, $\Der \out{\psi}{k} = \Der \psi_{k,0} + \m{O}(\euc^{-2k + 1} \dist^{-3})$.
  Note that $\nu \dist^{-4} \lesssim \nu \euc^{-2} \lesssim \nu^{1-2\beta} \lesssim 1.$
  Because $\psi_{k,0}$ is $(-4k+1)$-homogeneous, $\nu^k \Der \psi_{k,0} \lesssim_k \nu^k \dist^{-4k-2} \lesssim_k \dist^{-2}$.
  Treating the error,
  \begin{equation*}
    \nu^k \euc^{-2k + 1} \dist^{-3} \lesssim_k (\nu \euc^{-2})^k \euc \dist^{-3} \lesssim_k \dist^{-1} \lesssim \dist^{-2}.
  \end{equation*}
  Therefore $\abss{\out{\psi}{[K]}} \lesssim_K \dist^{-2}$ where $\euc \gtrsim \nu^\beta$.

  Next consider the inner sum $\inn{\psi}{[L]} = \sum_{\ell = 0}^K \inn{\psi}{\ell}$ where $\euc \lesssim \nu^\beta$.
  Interpreting Proposition~\ref{prop:Burgers-inner} in the outer coordinates, we find $\abss{\Der \inn{\psi}{0}} \lesssim (\dist \vee \nu^{1/4})^{-2}$.
  For higher inner terms, \eqref{eq:inner-bounds} yields
  \begin{equation*}
    \abss{\Der \inn{\psi}{\ell}} \lesssim_\ell \abs{\log \nu}^{\ell-1} \nu^{(\ell + 1)\beta/3} (\dist \vee \nu^{1/4})^{-3}.
  \end{equation*}
  Now $2\beta/3 > 4/13 > 1/4$, so
  \begin{equation*}
    \abss{\Der \inn{\psi}{\ell}} \lesssim_{\ell,\eps} \nu^{(\ell + 1)\beta/3 - \eps} (\dist \vee \nu^{1/4})^{-3} \lesssim_\ell \nu^{1/4} (\dist \vee \nu^{1/4})^{-3} \lesssim (\dist \vee \nu^{1/4})^{-2}.
  \end{equation*}
  We conclude that $\abss{\Der \inn{\psi}{[L]}} \lesssim_L (\dist \vee \nu^{1/4})^{-2}$.

  Our matching estimates imply that $\out{\psi}{[K]}$ and $\inn{\psi}{[L]}$ agree with $\psi_{[K,L]}$ to higher order where $\euc \asymp \nu^\beta$, so the cutoff in $\app{\psi}{[K,L]}$ does not change the estimate:
  \begin{equation*}
    \abss{\Der \app{\psi}{[K,L]}} \lesssim_{K,L} (\dist \vee \nu^{1/4})^{-2} \lesssim (\abs{t} \vee \nu^{1/2})^{-1}.
  \end{equation*}
  Finally, the transformation $\Eta$ is smooth, so the same bound holds for $\Der \app{\varphi}{[K,L]}$.
\end{proof}
We now show that the difference $v$ is small.
\begin{theorem}
  \label{thm:small}
  Fix $\beta \in (6/13, 1/2)$.
  There exists a sequence of exponents $(\Lambda_n)_{n \in \N_0}$ depending only on $t_0$, $A,$ $B$, and $\mr{\psi}$ such that for all $K,L,n \in \N_0$ and $\nu \in (0, 1]$,
  \begin{equation*}
    \norms{v}_{L_t^\infty H_x^n([t_0, 0] \times \R)} \lesssim_{K, L, n} \nu^{\min\{\!(1 - 2 \beta) K,\, L\beta/3\} - \Lambda_n}.
  \end{equation*}
\end{theorem}
\begin{proof}
  Suppose for the moment that the bootstrap \eqref{eq:bootstrap} holds on $[t_0, t]$ for some $t \in (t_0, 0].$
  Using Lemma~\ref{lem:der-app-soln} in Lemma~\ref{lem:energy}, we see that
  \begin{equation*}
    \int_{t_0}^t \big(\norms{\Der \app{\varphi}{[K,L]}}_{L_x^\infty} + \norm{\Der v}_{L_x^\infty} + \nu \norms{\partial_x \app{\varphi}{[K,L]}}_{L_x^\infty}^2 + \nu \norms{\partial_x v}_{L_x^\infty}^2\big) \lesssim \log \nu^{-1}.
  \end{equation*}
  Thus Lemma~\ref{lem:energy} implies that $\norms{\partial_x^n v}_{L_t^\infty L_x^2} \lesssim \nu^{-C(n)} \norms{F_n}_{L_t^\infty L_x^2}$, where the norms apply on $[t_0, t] \times \R$.
  Recall that $F_n$ depends polynomially on $\partial_x^n F$, derivatives of $\app{\phi}{[K,L]}$, and lower derivatives of $v$.
  We controlled $\partial_x^n F$ in Proposition~\ref{prop:diffeo-approx}, and derivatives of  $\app{\phi}{[K,L]}$, like those of $\app{\psi}{[K,L]}$, can be bounded by powers of $\nu^{-1}$.
  Thus we lose a power of $\nu^{-1}$ with each additional derivative.
  Using Proposition~\ref{prop:diffeo-approx} and induction, there exists a sequence $(\Lambda_n)_{n \in \N_0}$ independent of $K$ and $L$ such that
  \begin{equation}
    \label{eq:small}
    \norms{v}_{L_t^\infty H_x^n([t_0, t] \times \R)} \lesssim_{K, L, n} \nu^{\min\{\!(1 - 2 \beta) K,\, L\beta/3\} - \Lambda_n}.
  \end{equation}
  Taking $K,L \gg 1$, we can arrange $\norms{v}_{L_t^\infty H_x^2} \ll 1$.
  Then Sobolev embedding yields
  \begin{equation*}
    \norm{v}_{L_x^\infty}(s) + \norm{\Der v}_{L_x^\infty}(s) \ll 1 \ForAll s \in [t_0, t].
  \end{equation*}
  That is, we have improved on the bootstrap \eqref{eq:bootstrap}.

  Because $v(t_0, \anon) = 0$, a standard argument from local well-posedness and continuity implies that in fact \eqref{eq:bootstrap} holds for all $t \in (t_0, 0]$, and in particular for $t = 0$.
  So \eqref{eq:small} holds on $[t_0, 0] \times \R$ for sufficiently large $K$ and $L$.
  For smaller $K$ and $L$, the theorem follows from the triangle inequality.
\end{proof}

\subsection{Inviscid limit}
We now use Theorem~\ref{thm:small} to study the inviscid limit $\psi^{(\nu)} \to \psi^{(0)}$ first in $L^\infty$, then in H\"older norms.
\begin{proof}[Proof of Corollary~\ref{cor:rate}]
  By Theorem~\ref{thm:small} and Sobolev embedding, we can choose $\beta \in (6/13, 1/2)$ and $K,L \in \N_0$ such that $\norms{\psi^{(\nu)} - \app{\psi}{[K,L]}}_{L^\infty} \lesssim \nu^{1/3}$, where the norm applies to $[t_0, 0] \times \R$.
  Consider the difference $\app{\psi}{[K,L]} - \out{\psi}{0}$.
  In the outer region where $\euc \gtrsim \nu^\beta$, $\out{\psi}{[K]} - \out{\psi}{0} = \sum_{k=1}^K \nu^k \out{\psi}{k}$.
  Using \eqref{eq:outerest1}, this partial sum is dominated by the $k=1$ term:
  \begin{equation*}
    \out{\psi}{[K]} - \out{\psi}{0} \lesssim \nu \dist^{-3} + \nu \abs{\log \nu} \euc^{-1} \lesssim \nu^{1/4}.
  \end{equation*}
  Here we use \eqref{eq:distance-relation} and $\euc \gtrsim \nu^\beta \geq \nu^{1/2}$.
  Next, in the inner region where $\euc \lesssim \nu^\beta$, we can write
  \begin{equation}
    \label{eq:inner-breakdown}
    \inn{\psi}{[L]} - \out{\psi}{0} = \sum_{\ell = 0}^L (\inn{\psi}{\ell} - \psi_{0,\ell}) - (\out{\psi}{0} - \psi_{0,[L]}).
  \end{equation}
  By Proposition~\ref{prop:inviscid-row}, $\out{\psi}{0} - \psi_{0,[L]}$ is quite small once $L$ is large, so we need only treat $\inn{\psi}{\ell} - \psi_{0,\ell}$.
  At leading order, Proposition~\ref{prop:Burgers-inner} implies that
  \begin{equation*}
    \inn{\psi}{0} - \psi_{0,0} \lesssim \nu (\dist \vee \nu^{1/4})^{-3} \lesssim \nu^{1/4}.
  \end{equation*}
  For $\ell \geq 1$, Proposition~\ref{prop:inviscid-row}, Corollary~\ref{cor:worst-inner-size}, and the triangle inequality yield
  \begin{equation*}
    \inn{\psi}{\ell} - \psi_{0,\ell} \lesssim \abs{\log \nu} \nu^{2\beta/3} \lesssim \nu^{1/4}.
  \end{equation*}
  Since $\app{\psi}{[K,L]}$ is a convex combination of $\out{\psi}{[K]}$ and $\inn{\psi}{[L]}$ in the matching zone, the upper bound in the corollary follows.

  For the lower bound, we work with the shocking component in the diffusive zone $\{\dist \leq \nu^{1/4}\},$ which we implicitly assume in the coming bounds.
  In this region $\app{\sigma}{[K,L]} = \inn{\sigma}{[L]}$, so by our choice of $K$ and $L$, $\sigma^{(\nu)} = \inn{\sigma}{[L]} + \m{O}(\nu^{1/3})$.
  Using \eqref{eq:inner-bounds} from Lemma~\ref{cor:worst-inner-size}, we see that
  \begin{equation*}
    \sigma^{(\nu)} = \inn{\sigma}{0} + \m{O}(\nu^{1/3}).
  \end{equation*}
  On the other hand, \ref{hyp:expansion} yields $\sigma^{(0)} = \sigma_{0,0} + \m{O}(\nu^{1/2})$, so
  \begin{equation}
    \label{eq:inner-reduction}
    \sigma^{(\nu)} - \sigma^{(0)} = \inn{\sigma}{0} - \sigma_{0,0} + \m{O}(\nu^{1/3}).
  \end{equation}
  It thus suffices to show that $\abss{\inn{\sigma}{0} - \sigma_{0,0}} \gtrsim \nu^{1/4}$ somewhere in the diffusive zone.

  In the inner coordinates, $\partial_X \inn{\Sigma}{0} \lesssim 1$ while $\partial_X \Sigma_{0,0} \gg 1$ near the origin.
  Hence there exists $X_1 \in \R$ such that $\Dist(0, X_1) \leq 1$ and $\inn{\Sigma}{0}(0, X_1) \neq \Sigma_{0,0}(0, X_1)$.
  Scaling to the original coordinates, it follows that
  \begin{equation}
    \label{eq:gap}
    \abss{\inn{\sigma}{0}(0, x_1) - \sigma_{0,0}(0, x_1)} \gtrsim \nu^{1/4}
  \end{equation}
  for $x_1 \coloneqq \nu^{3/4} X_1$.
  Since $\dist(0, x_1) \leq \nu^{1/4}$, the lower bound in the corollary follows from \eqref{eq:inner-reduction}.
\end{proof}
For our H\"older estimates, we make the following simple observation:
\begin{lemma}
  \label{lem:homog-Holder}
  If $\abs{\partial_x f} \lesssim \dist^{-a}$ for some $a \in [0, 3)$, then $f \in L_t^\infty \m{C}_x^{1 - a/3}$.
\end{lemma}
\begin{proof}
  Fix $t < 0$ and let $g(x) \coloneqq f(t, x)$, so $\partial_x g \lesssim \abs{x}^{-a/3}$.
  Let $b \coloneqq 1 - a/3$, and first suppose $y > x \geq 0$.
  Integrating, we see that $g(y) - g(x) \lesssim y^b - x^b \leq (y - x)^b$, where the last inequality comes from comparing the derivative of $z^b$ on different intervals of width $y - x$.
  If $y > 0 > x$, we instead write
  \begin{equation*}
    g(y) - g(x) = [g(y) - g(0)] + [g(0) - g(x)] \lesssim y^b + \abs{x}^b \leq 2 (y + \abs{x})^b = (y - x)^b.
  \end{equation*}
  The remaining cases follow symmetrically.
\end{proof}
We now examine the vanishing viscosity limit in H\"older spaces.
We do not pursue sharp decay rates in $\nu$, but they can in principle be obtained from Theorem~\ref{thm:small}.
\begin{proof}[Proof of Corollary~\ref{cor:sharp-conv}]
  Fix $\delta \in (0, 1/3)$ and take $\beta \in (\max\{6/13, (2 + \delta)^{-1}\}, 1/2)$.
  Using this $\beta$ in Theorem~\ref{thm:small}, Morrey's inequality implies that we can choose $K,L$ sufficiently large that $\psi^{(\nu)} - \app{\psi}{[K,L]} \to 0$ in $L_t^\infty \m{C}_x^{2/3}$.
  It thus suffices to consider $\app{\psi}{} \coloneqq \app{\psi}{[K,L]}$.
  Recall that $\psi^{(0)} = \out{\psi}{0}$.

  In the outer region $\euc \gtrsim \nu^\beta$, $\out{\psi}{[K]} - \out{\psi}{0} = \sum_{k=1}^K \nu^k \out{\psi}{k}$.
  Note that \eqref{eq:distance-relation} yields $\dist \gtrsim \nu^{\beta /2} \gg \nu^{1/4}$.
  For the leading parts $\nu^k \psi_{k,0}$, $(-4k+1)$-homogeneity implies that
  \begin{equation*}
    \nu^k \partial_x \psi_{k,0} \lesssim_k \nu^k \dist^{-4k - 2} \lesssim \nu^\eta \dist^{-2}
  \end{equation*}
  for some $\eta(\delta,\beta) > 0$ that may change from line to line.
  By \eqref{eq:distance-relation} and Proposition~\ref{prop:out-bound},
  \begin{equation}
    \label{eq:outer-diff-Holder}
    \nu^k \partial_x(\out{\psi}{k} - \psi_{k,0}) \lesssim_k \abs{\log \nu} \nu^k \euc^{-2k + 1} \dist^{-3} \lesssim \nu^\eta \euc \dist^{-3} \lesssim \nu^\eta \dist^{-1}.
  \end{equation}
  So $\partial_x(\out{\psi}{[K]} - \out{\psi}{0}) \lesssim \nu^\eta \dist^{-2}$ in the outer region.

  In the inner region where $\euc \lesssim \nu^\beta$, we use \eqref{eq:inner-breakdown}.
  By Proposition~\ref{prop:inviscid-row}, $\out{\psi}{0} - \psi_{0,[L]}$ is quite regular once $L$ is large, so we need only treat $\inn{\psi}{\ell} - \psi_{0,\ell}$.
  At leading order, Proposition~\ref{prop:Burgers-inner} implies that
  \begin{equation*}
    \partial_x(\inn{\psi}{0} - \psi_{0,0}) \lesssim \nu (\dist \vee \nu^{1/4})^{-6} \lesssim (\dist \vee \nu^{1/4})^{-2} \lesssim \nu^\eta \dist^{-2 - 3\delta}.
  \end{equation*}
  For $\ell \geq 1$, we take $\eps = \delta$ in Corollary~\ref{cor:worst-inner-size} to find
  \begin{equation*}
    \partial_x(\inn{\psi}{\ell} - \cut{\psi}{0,\ell}) \lesssim \nu^{(1 - \delta)/2} (\dist \vee \nu^{1/4})^{-3},
  \end{equation*}
  recalling that $\cut{\psi}{0,\ell}$ is $\psi_{0,\ell}$ but cut off where $\dist \leq \nu^{1/4}$.
  By Proposition~\ref{prop:inviscid-row},
  \begin{equation*}
    \partial_x(\cut{\psi}{0,\ell} - \psi_{0,\ell}) \lesssim \dist^{-1} \tbf{1}_{\dist \leq \nu^{1/4}}.
  \end{equation*}
  Thus by the triangle inequality,
  \begin{equation*}
    \partial_x(\inn{\psi}{\ell} - \psi_{0,\ell}) \lesssim \nu^{(1 - \delta)/2} (\dist \vee \nu^{1/4})^{-3} + \dist^{-1} \tbf{1}_{\dist \leq \nu^{1/4}}.
  \end{equation*}
  Now
  \begin{equation*}
    \nu^{(1 - \delta)/2} (\dist \vee \nu^{1/4})^{-3} = \nu^{\delta/4} (\nu^{1/4})^{-2 - 3 \delta} (\dist \vee \nu^{1/4})^{-3} \lesssim \nu^\eta \dist^{-1 - 3 \delta}
  \end{equation*}
  and
  \begin{equation*}
    \dist^{-1} \tbf{1}_{\dist \leq \nu^{1/4}} \leq \nu^{3\delta/4} \dist^{-1 - 3 \delta} \tbf{1}_{\dist \leq \nu^{1/4}} \leq \nu^\eta \dist^{-1 - 3 \delta}.
  \end{equation*}
  We conclude that
  \begin{equation}
    \label{eq:cubic-square}
    \partial_x(\inn{\psi}{\ell} - \psi_{0,\ell}) \lesssim \nu^\eta \dist^{-1 - 3 \delta}.
  \end{equation}

  It only remains to consider the effect of the cutoff $\zeta$ in the matching zone where $\euc \asymp \nu^\beta$.
  Using \eqref{eq:def-approx}, we must bound $(\out{\psi}{[K]} - \inn{\psi}{[L]})\partial_x \zeta$.
  But $\partial_x \zeta \asymp \nu^{-\beta}$, and Proposition~\ref{prop:outer-inner-matching} ensures that we can choose $K,L \gg 1$ such that the matching error $\out{\psi}{[K]} - \inn{\psi}{[L]}$ is much less than $\nu^\beta$.
  So this term is suitably small, and we see that $\partial_x (\app{\psi}{} - \out{\psi}{0}) \lesssim \nu^\eta \dist^{-2 - 3 \delta}$.
  By Lemma~\ref{lem:homog-Holder},
  \begin{equation*}
    \norms{\app{\psi}{} - \out{\psi}{0}}_{L_t^\infty \m{C}_x^{1/3 - \delta}} \lesssim \nu^\eta.
  \end{equation*}
  Since $\delta > 0$ can be arbitrarily small, we have shown that $\psi^{(\nu)} \to \psi^{(0)}$ in $L_t^\infty \m{C}_x^{1/3-}$ as $\nu \to 0$.
  In particular, the same holds for the shocking component $\sigma$.

  We next consider the nonshocking components.
  Recall that the entire zeroth column ($\ell = 0$) of the grid and inner expansion has only a shocking component.
  In the absence of $\psi_{k,0}$, \eqref{eq:outer-diff-Holder} implies that
  \begin{equation*}
    \partial_x(\out{\omega}{[K]} - \out{\omega}{0}) \lesssim \nu^\eta \dist^{-1}
  \end{equation*}
  in the outer region where $\euc \gtrsim \nu^\beta$.
  In the inner region, we use \eqref{eq:inner-breakdown} and ignore the pure-shocking $\ell = 0$ terms.
  For $\ell \geq 1$, \eqref{eq:cubic-square} yields
  \begin{equation*}
    \partial_x(\inn{\omega}{[L]} - \out{\omega}{0}) \lesssim \nu^\eta \dist^{-1 - 3 \delta}.
  \end{equation*}
  Contributions from $\zeta$ in the matching region can be treated as before, so we find
  \begin{equation*}
    \partial_x(\omega^{(\nu)} - \omega^{(0)}) \lesssim \nu^\eta \dist^{-1 - 3 \delta}.
  \end{equation*}
  By Lemma~\ref{lem:homog-Holder}, $\omega^{(\nu)} \to \omega^{(0)}$ in $L_t^\infty \m{C}_x^{2/3 - \delta}$.

  To conclude, we must show that $\sigma$ does not converge in $\m{C}^{1/3}.$
  By Theorem~\ref{thm:small}, we can choose $K,L$ sufficiently large that $\psi^{(\nu)} - \app{\psi}{[K,L]} \to 0$ in $L_t^\infty\m{C}_x^{1/3}$.
  Because $\app{\psi}{[K,L]} = \inn{\psi}{[L]}$ in the diffusive zone $\{\dist \leq \nu^{1/4}\}$, it suffices to show that $\inn{\sigma}{[L]} - \out{\sigma}{0} \not\to 0$ in $L_t^\infty\m{C}_x^{1/3}(\{\dist \leq \nu^{1/4}\})$.

  Let $w \coloneqq \inn{\sigma}{[L]} - \out{\sigma}{0}.$
  By \eqref{eq:inner-reduction}, $w = \inn{\sigma}{0} - \sigma_{0,0} + \m{O}(\nu^{1/3})$.
  Both $\inn{\sigma}{0}$ and $\sigma_{0,0}$ are odd in $x$, so $\inn{\sigma}{0}(0, 0) = \sigma_{0,0}(0, 0) = 0$.
  Combining these observations with \eqref{eq:gap}, we see that
  \begin{equation*}
    \abs{w(0, x_1) - w(0, 0)} \gtrsim \nu^{1/4}
  \end{equation*}
  for some $x_1 \in \R$ such that $\dist(0, x_1) \leq \nu^{1/4}$.
  It follows that $\abs{x_1} \lesssim \nu^{3/4}$, so
  \begin{equation*}
    \frac{\abs{w(0, x_1) - w(0, 0)}}{\abs{x_1}^{1/3}} \gtrsim 1.
  \end{equation*}
  This implies that $\sigma^{(\nu)} \not \to \sigma^{(0)}$ in $L_t^\infty \m{C}_x^{1/3}$.
\end{proof}
\begin{remark}
  \label{rem:unfrozen}
  As noted in the introduction, the $2/3$ regularity threshold for $\omega$ may seem somewhat lax.
  Indeed, by Theorem~4.2 of \cite{AndChaGra25a}, the true nonshocking components $l^{I'}(\psi^{(0)}) \partial_x \psi^{(0)}$ are smooth in the eikonal coordinates $(\tau, u)$.
  Because the eikonal function $u$ is $\m{C}_x^{1/3}$, the inviscid solution has a $\m{C}_x^{1,1/3}$-character in the nonshocking tangent directions.
  We emphasize that this result requires the use of the variable left eigenvector $l^{I'}(\psi^{(0)})$ of $A(\psi^{(0)})$ rather than the frozen vector $l^{I'}(0)$ used to define $\omega$.
  Thus, the $2/3$ threshold in Corollary~\ref{cor:sharp-conv} might appear to be an artifact of our frozen coefficients.
  This is not the case.

  In fact, the $2/3$ threshold is intrinsic to the viscous problem, and cannot generally be improved through a redefinition of the nonshocking component.
  Indeed, a principal source of irregularity in $\omega^{(\nu)}$ is a diffusive cross term in \eqref{eq:main} with schematic form $B_1^\sperp \partial_x^2 \sigma$.
  Because the diffusion has no reason to be diagonal in the eigenbasis of $A$, we generally expect $B_1^\sperp \neq 0$, so the second derivative of the shocking component forces the nonshocking component.
  This force has order $\nu (\dist \vee \nu^{1/4})^{-5}$.
  Through renormalization along the nonshocking characteristic, it contributes a term $\omega_{\text{cross}}$ of order $\nu (\dist \vee \nu^{1/4})^{-2}$ to the nonshocking component.
  Differentiating and following the reasoning of the above proof, we find
  \begin{equation*}
    \partial_x \omega_{\text{cross}} \lesssim \nu (\dist \vee \nu^{1/4})^{-5} \lesssim \nu^\eps \dist^{-1 - 3 \delta}.
  \end{equation*}
  So $\omega_{\text{cross}}$ tends to zero in $\m{C}_x^{2/3-}$, but not in $\m{C}_x^{2/3}$.
  This is not an artifact of our frozen definition of $\omega$.
  Indeed, we can easily check that $l^{I'}(\psi^{(\nu)}) \partial_x \psi^{(\nu)}$ also includes $\partial_x \omega_{\text{cross}}$.
  It follows that $l^{I'}(\psi^{(\nu)}) \partial_x \psi^{(\nu)} \to l^{I'}(\psi^{(0)}) \partial_x \psi^{(0)}$ in $L_t^\infty \m{C}_x^\al$ if and only if $\al < - 1/3$ (drawing on negative H\"older spaces).
  This holds despite the fact that the limit itself is considerably more regular: as noted above, $l^{I'}(\psi^{(0)}) \partial_x \psi^{(0)} \in L_t^\infty \m{C}_x^{1/3}$.
  
  This limited degree of convergence is unavoidable unless a special relationship between $A$ and $B$ leads $B_1^\sperp(0)$ to vanish.
  For example, consider the simple system
  \begin{equation*}
    \rd_t v + v \rd_x v = \nu \partial_x^2 v \And \rd_t w - \rd_x w = \nu \mathsf{B} \partial_x^2 v,
  \end{equation*}
  in which viscous Burgers forces a linear transport equation through its second derivative.
  If $\mathsf{B} = B_1^\sperp(0) \neq 0$, then $l^{I'}(\psi^{(\nu)}) \partial_x \psi^{(\nu)} \not\to l^{I'}(\psi^{(0)}) \partial_x \psi^{(0)}$ in $L_t^\infty \m{C}_x^{-1/3}$.
  (The eigenspace of $A$ is constant in this problem, so $l^{I'}(\psi) \partial_x \psi = \partial_x \omega = \partial_x w$.)
  On the other hand, if $\mathsf{B} = 0$, $B$ is diagonal in the eigenbasis of $A$.
  Then $\omega_{\text{cross}}$ does not arise, and $\partial_x w^{(\nu)} \to \partial_x \omega^{(0)}$ in $L_t^\infty \m{C}_x^{1/3-}$.
  Indeed, in this simple problem we have $\partial_x w^{(\nu)} = \partial_x \omega^{(0)}$ when $\mathsf{B} = 0$.
\end{remark}
Finally, we prove that $\Psi^{(\nu)}$ has a universal limit as $\nu \to 0$.
\begin{proof}[Proof of Corollary~\ref{cor:universal}]
  Let $\m{K} \subset (-\infty,0] \times \R$ be compact.
  By Theorem~\ref{thm:small}, we can choose $K$ and $L$ sufficiently large that $\Psi^{(\nu)} = \inn{\Psi}{[L]} + \smallO(1)$ in $\m{K}$.
  Moreover, Proposition~\ref{prop:inner-matching} implies that $\inn{\Psi}{[L]} = \inn{\Psi}{0} + \smallO(1)$.
  Now $\inn{\Psi}{0} = \inn{\Sigma}{0} e_1$, so $\Psi^{(\nu)} \to \inn{\Sigma}{0} e_1$ locally uniformly as $\nu \to 0$.
  By \eqref{eq:s0}, $\inn{\Sigma}{0}$ satisfies viscous Burgers.
  Precisely, it is the unique solution of viscous Burgers that converges to the cubic $\cub(T, X)$ at infinity with error $\m{O}(\Dist^{-3})$ \cite[Proposition~4.2]{CG23}.

  This solution is parametrized by three scalars: $a,b \neq 0$ in \eqref{eq:fu} and the shocking viscosity $B_1^1(0) > 0$ in \eqref{eq:s0}.
  (It may appear that the advective coefficient $\partial_1 A_1^1(0)$ is a fourth parameter, but this is bound to $a$ through the relation $a = -\partial_1 A_1^1(0)$.)
  Through an appropriate scaling of $T,X$, and $\Sigma$, we can set these parameters to $1$.
  Thus $\inn{\Sigma}{0}$ is unique modulo scaling.
\end{proof}

\printbibliography
\end{document}